\documentclass[12pt,a4paper]{amsart}

\usepackage{amssymb}
\usepackage{array}

\newtheorem{theorem}{Theorem}[section]
\newtheorem{lemma}[theorem]{Lemma}
\newtheorem{corollary}[theorem]{Corollary}

\newtheorem{proposition}[theorem]{Proposition}

\theoremstyle{definition}

\theoremstyle{remark}
\newtheorem{remark}[theorem]{Remark}

\numberwithin{equation}{section}

\def\br{\mathbb R}
\def\bc{\mathbb C}

\def\bbf{\mathbb{F}}
\def\mcf{\mathcal{F}}
\def\mct{\mathcal{T}}
\def\mcm{\mathcal{M}}
\def\bq{\mathbb Q}

\def\mand{\quad\mbox{and}\quad}

\def\bp{\begin{pmatrix}}
\def\ep{\end{pmatrix}}

\def\sdi{\,\mbox{$\rhd\kern-.55em<$}\,}
\def\edots{\mathinner{\mkern1mu\raise1pt\hbox{.}\mkern2mu\raise4pt\hbox{.}\mkern2mu\raise7pt\vbox{\kern7pt\hbox{.}}\mkern1mu}}

\begin{document}

\title[Orthogonal multiple flag varieties]{Orthogonal multiple flag varieties of finite type I : odd degree case}

\author{Toshihiko MATSUKI}

\thanks{Supported by JSPS Grant-in-Aid for Scientific Research (C) \# 25400030.}

\address{Faculty of Letters\\
        Ryukoku University\\
        Kyoto 612-8577, Japan}
\email{matsuki@let.ryukoku.ac.jp}
\date{}

\begin{abstract}
Let $G$ be the split orthogonal group of degree $2n+1$ over an arbitrary field $\bbf$ of ${\rm char}\,\bbf\ne 2$. In this paper, we classify multiple flag varieties $G/P_1\times\cdots\times G/P_k$ of finite type. Here a multiple flag variety is called of finite type if it has a finite number of $G$-orbits with respect to the diagonal action of $G$ when $|\bbf|=\infty$.
\end{abstract}

\maketitle

\section{Introduction}

In \cite{MWZ1}, Magyar, Weyman and Zelevinsky classified multiple flag varieties for ${\rm GL}_n(\bbf)$ of finite type. In their subsequent paper \cite{MWZ2}, they classified multiple flag varieties for ${\rm Sp}_{2n}(\bbf)$ of finite type. (They assume $\bbf$ is algebraically closed.)

Recently the author gave explicit orbit decomposition for an example of orthogonal case in \cite{M2}. In this paper, we will classify multiple flag varieties of finite type for the split orthogonal group of degree $2n+1$.

Let $\bbf$ be an arbitrary commutative field of ${\rm char}\,\bbf\ne 2$. Let $(\ ,\ )$ denote the symmetric bilinear form on $\bbf^{2n+1}$ defined by
$$(e_i,e_j)=\delta_{i,2n+2-j}$$
for $i,j=1,\ldots,2n+1$ where $e_1,\ldots,e_{2n+1}$ is the canonical basis of $\bbf^{2n+1}$. Define the split orthogonal group
$$G=\{g\in {\rm GL}_{2n+1}(\bbf) \mid (gu,gv)=(u,v)\mbox{ for all }u,v\in\bbf^{2n+1}\}$$
with respect to this form. Let us write $G={\rm O}_{2n+1}(\bbf)$ in this paper.  

A subspace $V$ of $\bbf^{2n+1}$ is called isotropic if
$(V,V)=\{0\}$. 
For a sequence ${\bf a}=(\alpha_1,\ldots,\alpha_p)$ of positive integers such that $\alpha_1+\cdots+\alpha_p\le n$, there corresponds the flag variety
$$M_{\bf a}=\{V_1\subset\cdots\subset V_p\mid \dim V_i=\alpha_1+\cdots+\alpha_i\mbox{ for }i=1,\ldots,p, \ (V_p,V_p)=\{0\}\}.$$
For the canonical flag
$$\mcf_0: \bbf e_1\oplus\cdots\oplus \bbf e_{\alpha_1}\subset\cdots\subset \bbf e_1\oplus\cdots\oplus \bbf e_{\alpha_1+\cdots+\alpha_p}$$
in $M_{\bf a}$, the isotropic subgroup for $\mcf_0$ in $G$ is a standard parabolic subgroup $P_{\bf a}$ consisting of elements in $G$ of the form
$$\bp A_1 &&&&&& * \\ & \ddots &&&&& \\ && A_p &&&& \\ &&& A_0 &&& \\ &&&& A_p^* && \\ &&&&& \ddots & \\ 0 &&&&&& A_1^*
\ep$$
with $A_i\in {\rm GL}_{\alpha_i}(\bbf)$ for $i=1,\ldots,p$ and $A_0\in {\rm O}_{2\alpha_0+1}(\bbf)$ ($\alpha_0=n-(\alpha_1+\cdots+\alpha_p)$) where $A_i^*=J_{\alpha_i}{}^tA_i^{-1}J_{\alpha_i}$ for $i=1,\ldots,p$ and $J_m$ is the $m\times m$ matrix given by
$$J_m=\bp 0 && 1 \\ & \edots & \\ 1 && 0 \ep.$$
Since $M_{\bf a}$ is $G$-homogeneous, we can identify $M_{\bf a}$ with $G/P_{\bf a}$.

\begin{remark} Define the split special orthogonal group
$$G_0=\{g\in G\mid \det g=1\}\ (={\rm SO}_{2n+1}(\bbf)).$$
Then $G=G_0\sqcup (-I_{2n+1})G_0$. Since $-I_{2n+1}$ acts trivially on $M_{\bf a}$, $G_0$-orbits on $M_{\bf a}$ are the same as $G$-orbits.
\end{remark}

Consider a multiple flag variety
$$\mcm=M_{{\bf a}_1}\times\cdots\times M_{{\bf a}_k}\cong (G/P_{{\bf a}_1})\times\cdots\times (G/P_{{\bf a}_k})$$
with the diagonal $G$-action
$$g(\mcf_1,\ldots \mcf_k)=(g\mcf_1,\ldots,g\mcf_k)$$
for $g\in G$ and $(\mcf_1,\ldots,\mcf_k)\in M_{{\bf a}_1}\times\cdots\times M_{{\bf a}_k}$.
The multiple flag variety $\mcm$ is called of finite type if it has a finite number of $G$-orbits with respect to the diagonal $G$-action when $|\bbf|=\infty$.

If $k=2$, then we have $G\backslash ((G/P_{{\bf a}_1})\times (G/P_{{\bf a}_2})) \cong P_{{\bf a}_2}\backslash G/P_{{\bf a}_1}$ by the map
$$(g_1,g_2)\mapsto g_2^{-1}g_1.$$
So it is of finite type by the Bruhat decomposition. In this paper we first show the following.

\begin{proposition} If $k\ge 4$, then $\mcm$ is of infinite type.
\label{prop1.1}
\end{proposition}

So we have only to consider triple flag varieties
$$\mct=\mct_{{\bf a},{\bf b},{\bf c}}=M_{\bf a}\times M_{\bf b}\times M_{\bf c}$$ with
${\bf a}=(\alpha_1,\ldots,\alpha_p),\ {\bf b}=(\beta_1,\ldots,\beta_q)$ and ${\bf c}=(\gamma_1,\ldots,,\gamma_r)$.
We may assume
$$p\le q\le r.$$

\begin{proposition} If $\mct$ is of finite type, then $p=q=1$.
\label{prop1.2}
\end{proposition}

So we may assume $p=q=1$ in the following. We may assume
$$\alpha_1\le \beta_1.$$
When $r=1$, we may moreover assume
$$\alpha_1\le \beta_1\le \gamma_1.$$

We can also prove:

\begin{proposition} Suppose $\mct_{{\bf a},{\bf b},{\bf c}}$ is of finite type. Then the condition
$${\rm (C)}\qquad \max(\alpha_1,\beta_1,\gamma_1)<n\Longrightarrow |\bbf^\times/(\bbf^\times)^2|<\infty$$
holds.
\label{prop1.4}
\end{proposition}

\begin{remark} (i) If $\bbf$ is algebraically closed, then $(\bbf^\times)^2=\bbf^\times$.

(ii) If $\bbf=\br$ or $\bbf$ is a finite field, then $|\bbf^\times/(\bbf^\times)^2|=2$.

(iii) There are many fields $\bbf$ such that $|\bbf^\times/(\bbf^\times)^2|=\infty$. For example, $\bbf=\bq,\ \bc(x)$ (the field of rational functions with one variable) etc.
\end{remark}

We can classify triple flag varieties for $G$ of finite type as follows.

\begin{theorem} Suppose $1=p=q\le r$ and $\alpha_1\le\beta_1$. When $r=1$, suppose
$\alpha_1\le \beta_1\le \gamma_1$. Furthermore suppose the above condition {\rm (C)}. Then $\mct=\mct_{{\bf a},{\bf b},{\bf c}}$ is of finite type if and only if $({\bf a},{\bf b},{\bf c})$ satisfies one of the following four conditions.

{\rm (I)} $\alpha_1=\beta_1=n$.

{\rm (II)} $\alpha_1=1$.

{\rm (III)} $r=1$ and $\gamma_1=n$.

{\rm (IV)} $r=2$ and $\beta_1=n$.
\label{th1.3}
\end{theorem}

\begin{remark} (i) Let $G$ be a simple algebraic group with parabolic subgroups $P_1,P_2$ and $P_3$. Then $G$-orbit decomposition of the triple flag variety $(G/P_1)\times (G/P_2)\times (G/P_3)$ is identified with $P_3$-orbit decomposition of the double flag variety $(G/P_1)\times (G/P_2)$ by the map $(g_1,g_2,g_3)\mapsto (g_3^{-1}g_1,g_3^{-1}g_2)$.

When $P_3$ is a Borel subgroup $B$, Littelmann and Stembridge classified double flag varieties with open $B$-orbits in \cite{L} and  \cite{S}. (In \cite{L}, $P_1$ and $P_2$ are maximal parabolic subgroups.) The two cases (I) and (II) in Theorem \ref{th1.3} are written in Table 1 of \cite{L}. So the double flag variety has an open orbit for these cases.

(ii) Suppose $\bbf$ is an algebraically closed field of characteristic zero. Then a double flag variety $(G/P_1)\times (G/P_2)$ has an open $B$-orbit if and only if it has a finite number of $B$-orbits by Brion-Vinberg's theorem (\cite{B},\cite{V}).

(iii) In \cite{M2}, we explicitly described $G$-orbit decomposition of the triple flag variety $\mct=\mct_{(n),(n),(1^n)}$ (the case (I) in Theorem \ref{th1.3} with $P_{\bf c}=B$) over an arbitrary field $\bbf$ of ${\rm char}\,\bbf\ne 2$.

(iv) The dimension of the flag variety $M_{\bf a}=M_{(\alpha_1,\ldots,\alpha_p)}$ is given by
$$\dim M_{\bf a}=n^2-\frac{\alpha_1(\alpha_1-1)}{2}-\cdots- \frac{\alpha_p(\alpha_p-1)}{2}-(n-\alpha_1-\cdots-\alpha_p)^2.$$
If $\mct=\mct_{{\bf a},{\bf b},{\bf c}}$ is of finite type, then
$$\dim \mct=\dim M_{\bf a}+\dim M_{\bf b}+\dim M_{\bf c}\le \dim G=n(2n+1).$$
For each case in Theorem \ref{th1.3}, we can also get this inequality by direct computation. The equality $\dim \mct=\dim G$ holds for the following $({\bf a},{\bf b},{\bf c})$:

(I) $(n)(n)(1^n)$.

(II) $(1)(2)(11)$.

(III) $(1)(1)(1)\ (n=1)$.

(IV) $(1)(2)(11),\ (2)(2)(11),\ (2)(3)(11),\ (2)(3)(12),\ (2)(3)(21),\ (3)(4)(22)$,

\noindent
$(3)(4)(12),\ (3)(4)(21),\ (3)(5)(22),\ (4)(5)(22),\ (4)(6)(23),\ (4)(6)(32),\ (5)(7)(33)$.
\end{remark}

This paper is organized as follows.

In Section 2, we first prove Proposition \ref{prop1.1} by a similar argument as in \cite{MWZ1} and \cite{MWZ2}. So we have only to consider triple flag varieties in the rest of this paper. We prove in this section that one of the four conditions in Theorem \ref{th1.3} holds if the triple flag variety is of finite type. Our arguments are essentially the same as those in 
\cite{MWZ1} and \cite{MWZ2}. Proposition \ref{prop1.4} is also proved in the same way.

In Section 3, we describe representatives of $G$-orbits on the triple flag variety $\mct_{(\alpha),(\beta),(n)}$ (Theorem \ref{th3.10}). This in particular implies that the triple flag varieties in (III) of Theorem \ref{th1.3} are of finite type. We fix $\alpha$ and $\beta$-dimensional isotropic subspaces $U_+$ and $U_-$, respectively, of $\bbf^{2n+1}$. Then we have only to classify $R$-orbits of maximally isotropic subspaces in $\bbf^{2n+1}$ where $R=\{g\in G\mid gU_+=U_+\mbox{ and }gU_-=U_-\}$.

In Section 4, we prove that every triple flag variety $\mct_{(\alpha),(n),(\gamma_1,\gamma_2)}$ in (IV) of Theorem \ref{th1.3} is of finite type. Changing the order of flag varieties, we consider triple flag varieties of the form $\mct_{(\alpha_1,\alpha_2),(\beta),(n)}$. We solve this problem in the following way. Put $\alpha=\alpha_1+\alpha_2$. Then we may fix $\alpha,\beta$ and $n$-dimensional isotropic subspaces $U_+,U_-$ and $V$, respectively, by Theorem \ref{th3.10}. Write $R_V=\{g\in G\mid gU_+=U_+,\ gU_-=U_-,\ gV=V\}$. Then we have only to consider $R_V$-orbit decomposition of the Grassmann variety consisting of $\alpha_1$-dimensional subspaces of $U_+$. Our arguments are so complicated that we only show ``finiteness'' of the number of orbits for this case.

In Section 5, we prove that the triple flag varieties $\mct_{(1),(\beta_1),{\bf c}}$ in (II) are of finite type. By the condition (C), we may assume
$$\beta_1=n,\quad \gamma_1=n\quad\mbox{or}\quad |\bbf^\times/(\bbf^\times)^2|<\infty.$$
(Note that the condition (C) is automatically satisfied for the other cases (I), (III) and (IV).) If $\gamma_1=n$, then ${\bf c}=(n)$. So this case is included in (III) and hence we may assume
$$\beta_1=n\quad\mbox{or}\quad |\bbf^\times/(\bbf^\times)^2|<\infty.$$
We may moreover assume ${\bf c}=(1^n)$. Changing the order of flag varieties, we may consider triple flag varieties of the form $\mct_{(\alpha),(1),(1^n)}$ with the condition
$${\rm (C')}\qquad \alpha<n\Longrightarrow |\bbf^\times/(\bbf^\times)^2|<\infty.$$

\section{Exclusion of multiple flag varieties of infinite type}

\subsection{A technical lemma}

In this section, we use the following technical lemma.

\begin{lemma} Let $W$ be a nondegenerate subspace of $\bbf^{2n+1}$. Let $W_1,\ldots,W_k,W'_1,\ldots,W'_k$ be subspaces of $W$ such that
$$W=W_1+\cdots+ W_k=W'_1+\cdots+ W'_k.$$
Let $U_1$ be an isotropic subspace of $W^\perp$ and $U_2,\ldots,U_k$ be subspaces of $U_1$. Suppose
$$g(W_i\oplus U_i)=W'_i\oplus U_i\quad \mbox{for }i=1,\ldots,k$$
for a $g\in G$. Then we have$:$

{\rm (i)} $g(W\oplus U_1)=W\oplus U_1$.

{\rm (ii)} $gU_1=U_1$.

\noindent$($ By {\rm (i)} and {\rm (ii)}, $g$ induces an isometry on the factor space $(W\oplus U_1)/U_1\cong W$.$)$
\label{lem2.1'}
\end{lemma}

\begin{proof} (i) The assertion is clear since
$$(W_1\oplus U_1)+\cdots+(W_k\oplus U_k)=(W'_1\oplus U_1)+\cdots+(W'_k\oplus U_k)=W\oplus U_1.$$

The assertion (ii) follows from (i) since the orthogonal space of $W\oplus U_1$ is $U_1$.
\end{proof}

\subsection{The case of $n=1$}

Suppose $n=1$. Define one-dimensional isotropic subspaces
$$W_\lambda=\bbf\left(\lambda e_1+e_2-{1\over 2\lambda}e_3\right)$$
of $\bbf^3$ for $\lambda\in\bbf^\times$. Then the flag variety $M=M_{(1)}\cong P^1(\bbf)$ consists of
$\bbf e_1,\ \bbf e_3$ and $W_\lambda$ with $\lambda\in\bbf^\times$.

\begin{lemma} Let $g$ be an element of $G={\rm O}_3(\bbf)$ such that
$g\bbf e_1=\bbf e_1,\ g\bbf e_3=\bbf e_3$ and that $gW_1=W_1$.
Then $g=\pm {\rm id}$.
\label{lem2.1}
\end{lemma}

\begin{proof} It follows from $g\bbf e_1=\bbf e_1$ and $g\bbf e_3=\bbf e_3$ that $g$ is of the form
$$\bp \mu & 0 & 0 \\ 0 & \varepsilon & 0 \\ 0 & 0 & \mu^{-1} \ep$$
with some $\mu\in\bbf^\times$ and $\varepsilon=\pm 1$. Since $gW_1=W_1$, we have $\mu=\varepsilon$.
\end{proof}

\begin{corollary} When $n=1$, the multiple flag variety $M\times M\times M\times M$ is of infinite type.
\end{corollary}

\begin{remark} If $\bbf$ is algebraically closed, then the corollary also follows from
$$\dim (M\times M\times M\times M)=4>3=\dim {\rm O}_3(\bbf).$$
\end{remark}

\subsection{Proof of Proposition \ref{prop1.1}}

Write $U_{[\ell]}=\bbf e_1\oplus\cdots\oplus \bbf e_\ell$ and $W=\bbf e_n\oplus \bbf e_{n+1}\oplus \bbf e_{n+2}$. We may assume $k=4$ and
${\bf a}_i=(\ell_i)$
for $i=1,\ldots,4$. 
Define one-dimensional isotropic subspaces
$$W_\lambda=\bbf\left(\lambda e_n+e_{n+1}-{1\over 2\lambda}e_{n+2}\right)$$
of $W$ for $\lambda\in\bbf^\times$. 
Take isotropic subspaces
\begin{align*}
V_1 & =\bbf e_n\oplus U_{[\ell_1-1]}, \quad
V_2 =\bbf e_{n+2}\oplus U_{[\ell_2-1]}, \\
V_3 & =W_1\oplus U_{[\ell_3-1]} 
\mand V_{4,\lambda} =W_\lambda \oplus U_{[\ell_4-1]}
\end{align*}
of $\bbf^{2n+1}$. Then $m_\lambda=(V_1,V_2,V_3,V_{4,\lambda})$ are elements of $\mcm$ for $\lambda\in\bbf^\times$. We have only to show that $Gm_\lambda\ne Gm_\mu$ if $\lambda\ne \mu$.

Suppose $gm_\lambda=m_\mu$ with some $g\in G$. Then
$gV_i=V_i\mbox{ for }i=1,2,3$ and $gV_{4,\lambda}=V_{4,\mu}$. Since
$$\bbf e_n+\bbf e_{n+2}+W_1+W_\lambda=\bbf e_n+\bbf e_{n+2}+W_1+W_\mu=W,$$
we have
$$g(W\oplus U)=W\oplus U\mand gU=U$$
where $U=U_{[\max(\ell_1,\ell_2,\ell_3,\ell_4)-1]}$ by Lemma \ref{lem2.1'}. Hence $g$ induces an isometry $\widetilde{g}$ on the factor space $\widetilde{W}=(W\oplus U)/U\cong W$. Let $\pi$ denote the projection $W\oplus U\to W$. Then
$$\pi(V_1)=\bbf e_n,\quad \pi(V_2)=\bbf e_{n+2},\quad \pi(V_3)=W_1,\quad \pi(V_\lambda)=W_\lambda\mand \pi(V_\mu)=W_\mu.$$
Since $gV_i=V_i\mbox{ for }i=1,2,3$ and $gV_{4,\lambda}=V_{4,\mu}$, it follows that
$$\widetilde{g}\bbf e_n=\bbf e_n,\quad \widetilde{g}\bbf e_{n+2}=\bbf e_{n+2},\quad \widetilde{g}W_1=W_1\mand \widetilde{g}W_\lambda=W_\mu.$$
Hence we have $\widetilde{g}=\pm{\rm id}$ by Lemma \ref{lem2.1} and therefore $\lambda=\mu$.
\hfill$\square$

\subsection{A lemma for ${\rm O}_5(\bbf)$}

Consider $\bbf^5$ with the canonical basis $f_1,\ldots,f_5$ and the symmetric bilinear form $(\ ,\ )$ such that
$(f_i,f_j)=\delta_{i,6-j}$.
Take a flag
$\mcf: \bbf(f_1+f_2)\subset \bbf f_1\oplus \bbf f_2$
and isotropic subspaces
$U=\bbf f_4\oplus \bbf f_5,\ U'=\bbf(f_1+f_3-{1\over 2}f_5)$
of $\bbf^5$.

\begin{lemma} If $g\mcf=\mcf, gU=U$ and $gU'=U'$ for a $g\in G={\rm O}_5(\bbf)$, then $g=\pm{\rm id}$.
\label{lem2.5}
\end{lemma}

\begin{proof} Since $g(\bbf f_1\oplus \bbf f_2)=\bbf f_1\oplus \bbf f_2$ and since $g(\bbf f_4\oplus \bbf f_5)=\bbf f_4\oplus \bbf f_5$, $g$ is of the form
$$\bp A && 0 \\ & \varepsilon & \\ 0 && J\,{}^tA^{-1}J \ep$$
with some $A\in{\rm GL}_2(\bbf)$ and $\varepsilon=\pm 1$ where $\displaystyle{J=\bp 0 & 1 \\ 1 & 0 \ep}$. Hence $gf_3=\varepsilon f_3$. It follows from $gU'=U'$ that
$$gf_1=\varepsilon f_1\mand gf_5=\varepsilon f_5.$$
It follows from the orthogonality of $g$ that
$$gf_2\in\bbf f_2\mand gf_4\in\bbf f_4.$$
Finally it follows from $g(\bbf(f_1+f_2))=\bbf(f_1+f_2)$ and $(gf_2,gf_4)=(f_2,f_4)=1$ that $gf_2=\varepsilon f_2$ and $gf_4=\varepsilon f_4$. Hence $g=\varepsilon{\rm id}$.
\end{proof}

\begin{corollary} The triple flag varieties $\mct_{(2),(1,1),(1,1)}$ and $\mct_{(1),(1,1),(1,1)}$ for $G={\rm O}_5(\bbf)$ are of infinite type.
\end{corollary}

\begin{proof} Suppose $\bbf$ is infinite. Then there are infinitely many isotropic two-dimensional subspaces containing $U'$. So it follows from Lemma \ref{lem2.5} that the triple flag variety $\mct_{(2),(1,1),(1,1)}$ is of infinite type. Similarly, there are infinitely many one-dimensional subspaces of $U$. So the triple flag variety $\mct_{(1),(1,1),(1,1)}$ is of infinite type.
\end{proof}

\begin{remark} If $\bbf$ is algebraically closed, then the corollary also follows from
$$\dim (M_{\bf a}\times M_{(1,1)}\times M_{(1,1)})=3+4+4=11>10=\dim {\rm O}_5(\bbf)$$
for ${\bf a}=(2)$ and $(1)$.
\end{remark}

\subsection{Proof of Proposition \ref{prop1.2}}

Proposition \ref{prop1.2} is equivalent to the following proposition:

\begin{proposition} If $q\ge 2$, then the triple flag variety $\mct=\mct_{{\bf a},{\bf b},{\bf c}}$ is of infinite type.
\label{prop2.8}
\end{proposition}

\begin{proof} We may assume $p=1$ and $q=r=2$. So we can write
$${\bf a}=(\alpha_1),\quad {\bf b}=(\beta_1,\beta_2)\mand {\bf c}=(\gamma_1,\gamma_2).$$
Define an isometric inclusion $\iota: \bbf^5\to\bbf^{2n+1}$ by
$$\iota(f_i)=e_{i+n-2}.$$

First suppose $\alpha_1\ge 2$. Let $M_{U'}\cong P^1(\bbf)$ denote the variety consisting of two-dimensional isotropic subspaces of $\bbf^5$ containing $U'=\bbf(f_1+f_3-{1\over 2}f_5)$. For an element $W$ of $M_{U'}$, define
\begin{align*}
V_1 & =\bbf e_{n+2}\oplus \bbf e_{n+3} \oplus U_{[\alpha_1-2]}, \quad
V_2 =\bbf (e_{n-1}+e_n)\oplus U_{[\beta_1-1]}, \\
V_3 & =\bbf e_{n-1}\oplus \bbf e_n\oplus U_{[\beta_1+\beta_2-2]}, \quad
V_4 =\iota(U')\oplus U_{[\gamma_1-1]} 
\mand V_{5,W} =\iota(W)\oplus U_{[\gamma_1+\gamma_2-2]}.
\end{align*}
Then $t_W=(V_1,(V_2\subset V_3),(V_4\subset V_{5,W}))$ is a triple flag contained in $\mct_{{\bf a},{\bf b},{\bf c}}$. We can prove $Gt_W\ne Gt_{W'}$ for two distinct $W,W'\in M_{U'}$ as in the proof of Proposition \ref{prop1.1} using Lemma \ref{lem2.5}.

Next suppose $\alpha_1=1$. Let $M_U\cong P^1(\bbf)$ denote the variety consisting of one-dimensional subspaces of $U=\bbf f_4\oplus\bbf f_5$. For an element $W$ of $M_U$, define
\begin{align*}
V_1 & =\iota(U'), \quad
V_2 =\bbf (e_{n-1}+e_n)\oplus U_{[\beta_1-1]}, \quad
V_3 =\bbf e_{n-1}\oplus \bbf e_n\oplus U_{[\beta_1+\beta_2-2]}, \\
V_{4,W} & =\iota(W)\oplus U_{[\gamma_1-1]} \quad
\mand V_5 =\bbf e_{n+2}\oplus\bbf e_{n+3}\oplus U_{[\gamma_1+\gamma_2-2]}.
\end{align*}
Then $t_W=(V_1,(V_2\subset V_3),(V_{4,W}\subset V_5))$ is a triple flag contained in $\mct_{{\bf a},{\bf b},{\bf c}}$. We can prove $Gt_W\ne Gt_{W'}$ for two distinct $W,W'\in M_U$ as in the proof of Proposition \ref{prop1.1} using Lemma \ref{lem2.5}.
\end{proof}

\subsection{First lemma for ${\rm O}_6(\bbf)$}

Consider $\bbf^6$ with the canonical basis $f_1,\ldots,f_6$ and the symmetric bilinear form $(\ ,\ )$ such that
$(f_i,f_j)=\delta_{i,7-j}$
and let $G$ denote the orthogonal group for this bilinear form. Define two-dimensional isotropic subspaces
\begin{align*}
U_1 & =\bbf f_1\oplus\bbf f_2, \quad
U_2 =\bbf f_5\oplus\bbf f_6, \\
U_{3,\lambda} & =\bbf(f_1+f_3+f_5)\oplus \bbf(\lambda f_2-f_4+(1-\lambda) f_6)
\end{align*}
of $\bbf^6$ for $\lambda\in\bbf$.

\begin{lemma} If $gU_1=U_1,\ gU_2=U_2$ and $gU_{3,\lambda}=U_{3,\mu}$ for some $g\in G={\rm O}_6(\bbf)$ and $\lambda,\mu\in\bbf$. Then $\lambda=\mu$ or $1-\mu$.
\label{lem2.9}
\end{lemma}

\begin{proof} Suppose $gU_1=U_1$ and $gU_2=U_2$ for a $g\in G$. Then $g$ is of the form
\begin{equation}
\bp A && 0 \\ & B(b,\det g) & \\ 0 && J\,{}^tA^{-1}J \ep \label{eq2.1}
\end{equation}
with some $A\in {\rm GL}_2(\bbf)$ and $b\in\bbf^\times$. Here
$$B(b,1)= \bp b & 0 \\ 0 & b^{-1} \ep ,\quad B(b,-1)= \bp 0 & b^{-1} \\ b & 0 \ep \mand J=\bp 0 & 1 \\ 1 & 0 \ep.$$
First suppose $\det g=1$ and $gU_{3,\lambda}=U_{3,\mu}$. Then we have
$$g\bbf(f_1+f_3+f_5)=\bbf(f_1+f_3+f_5)$$
and
$$g\bbf(\lambda f_2-f_4+(1-\lambda) f_6)=\bbf(\mu f_2-f_4+(1-\mu) f_6).$$
Hence
\begin{align*}
gf_1 & =bf_1,\quad gf_3=bf_3,\quad gf_5=bf_5,\\
g\lambda f_2 & =b^{-1}\mu f_2,\quad gf_4=b^{-1}f_4\mand g(1-\lambda)f_6=b^{-1}(1-\mu)f_6
\end{align*}
with some $b\in\bbf^\times$. Thus we have
$$\lambda=(f_5,\lambda f_2)=(gf_5,g\lambda f_2)=(bf_5,b^{-1}\mu f_2)=\mu.$$

On the other hand, suppose $\det g=-1$ and $gU_{3,\lambda}=U_{3,\mu}$. Then we have
$$g\bbf(f_1+f_3+f_5)=\bbf(\mu f_2-f_4+(1-\mu) f_6)$$
and
$$g\bbf(\lambda f_2-f_4+(1-\lambda) f_6)=\bbf(f_1+f_3+f_5).$$
Hence
\begin{align*}
gf_1 & =-b\mu f_2,\quad gf_3=bf_4,\quad gf_5=-b(1-\mu)f_6,\\
g\lambda f_2 & =-b^{-1} f_1,\quad gf_4=b^{-1}f_3\mand g(1-\lambda)f_6=-b^{-1}f_5
\end{align*}
with some $b\in\bbf^\times$. Thus we have
$$\lambda=(f_5,\lambda f_2)=(gf_5,g\lambda f_2)=(-b(1-\mu)f_6,-b^{-1}f_1)=1-\mu.$$
\end{proof}

\begin{corollary} The triple flag variety $\mct_{(2),(2),(2)}$ for $G={\rm O}_6(\bbf)$ is of infinite type.
\end{corollary}

\begin{remark} Suppose that $\bbf$ is algebraically closed. Then we have
$$\dim \mct_{(2),(2),(2)}=5+5+5=15=\dim G.$$
But $\mct_{(2),(2),(2)}$ has no open $G$-orbit.
\end{remark}

\subsection{Second lemma for ${\rm O}_6(\bbf)$}

Define isotropic subspaces
\begin{align*}
U_1 & =\bbf f_1\oplus\bbf f_2, \quad
U_2 =\bbf f_5\oplus\bbf f_6, \quad
U_{4,\lambda} =\bbf(\lambda f_1-f_3+(1-\lambda) f_5) \\
\mand U_5 & =\bbf(f_1-f_5)\oplus \bbf(f_1-f_3)\oplus \bbf(f_2+f_4+f_6)
\end{align*}
of $\bbf^6$ for $\lambda\in\bbf$. Then $t_\lambda=(U_1,U_2,(U_{4,\lambda}\subset U_5))$ are triple flags in $\mct=\mct_{(2),(2),(1,2)}$.

\begin{lemma} If $gt_\lambda=t_\mu$ for a $g\in G$ and $\lambda,\mu\in\bbf$, then $\lambda=\mu$.
\label{lem2.12}
\end{lemma}

\begin{proof} Since $gU_1=U_1$ and $gU_2=U_2$, $g$ is of the form (\ref{eq2.1}). Since $(U_1\oplus U_2)\cap U_5=\bbf(f_1-f_5)$, it follows from $gU_5=U_5$ that
$$gf_1=af_1\mand gf_5=af_5$$
with some $a\in\bbf^\times$. Finally it follows from $gU_{4,\lambda}=U_{4,\mu}$ that $\det g=1$ and that $\lambda=\mu$.
\end{proof}

\begin{corollary} The triple flag variety $\mct=\mct_{(2),(2),(1,2)}$ for $G={\rm O}_6(\bbf)$ is of infinite type.
\end{corollary}

\begin{remark} Suppose that $\bbf$ is algebraically closed. Let $U'_4=\bbf(f_1+f_2-f_3+f_4+f_6)$ and $m'=(U_1,U_2,(U'_4\subset U_5))$. Then $Gm'$ is the open $G$-orbit in the triple flag variety $\mct=\mct_{(2),(2),(1,2)}$ for $G={\rm O}_6(\bbf)$.
\end{remark}

\subsection{Some conditions on ${\bf a},{\bf b}$ and ${\bf c}$}

By Proposition \ref{prop1.2}, we may assume $\mct=\mct_{{\bf a},{\bf b},{\bf c}}$ with
$${\bf a}=(\alpha_1),\quad {\bf b}=(\beta_1)\mand {\bf c}=(\gamma_1,\ldots, \gamma_r).$$

\begin{proposition} Suppose that $\mct=\mct_{{\bf a},{\bf b},{\bf c}}$ is of finite type.

{\rm (i)} If $r=1$ and $\alpha_1\le \beta_1\le \gamma_1$, then $\alpha_1=1$ or $\gamma_1=n$.

{\rm (ii)} If $r\ge 2$ and $\alpha_1\le \beta_1$, then $\alpha_1=1$ or $\beta_1=n$.
\label{prop2.15}
\end{proposition}

\begin{proof} (i) Suppose $r=1$ and $1< \alpha_1\le \beta_1\le\gamma_1<n$.
Define an isometric inclusion $\iota: \bbf^6\to\bbf^{2n+1}$ by
\begin{align*}
\iota(f_1) & =e_{n-2},\quad \iota(f_2)=e_{n-1},\quad \iota(f_3)=e_n, \\
\iota(f_4) & =e_{n+2},\quad \iota(f_5)=e_{n+3}\mand \iota(f_6)=e_{n+4}.
\end{align*}
Define isotropic subspaces
\begin{align*}
V_1 & =\bbf e_{n-2}\oplus \bbf e_{n-1}\oplus U_{[\alpha_1-2]}, \quad
V_2 =\bbf e_{n+3}\oplus \bbf e_{n+4}\oplus U_{[\beta_1-2]} \\
\mand V_{3,\lambda} & =\iota(U_{3,\lambda})\oplus U_{[\gamma_1-2]}
\end{align*}
of $\bbf^{2n+1}$ for $\lambda\in\bbf$. Then $t_\lambda=(V_1,V_2,V_{3,\lambda})$ is a triple flag contained in $\mct_{{\bf a},{\bf b},{\bf c}}$. We can prove $Gt_\lambda \ne Gt_\mu$ if $\lambda\ne \mu,1-\mu$ as in the proof of Proposition \ref{prop1.1} using Lemma \ref{lem2.9}. Hence $\mct=\mct_{{\bf a},{\bf b},{\bf c}}$ is of infinite type.

(ii) First suppose $r\ge 3$ and $1< \alpha_1\le \beta_1<n$.
Write
$\gamma=\gamma_1+\gamma_2$.
Then we can consider the natural projection
$$M_{(\gamma_1,\ldots, \gamma_r)}\to M_{(\gamma)}.$$
Since $1<\gamma<n$, it follows from (i) that $\mct_{{\bf a},{\bf b},(\gamma)}$ is of infinite type. Hence $\mct_{{\bf a},{\bf b},{\bf c}}$ is of infinite type.

So we have only to consider the case of $r=2$. Suppose
$1< \alpha_1\le \beta_1<n$.
Then we can show that $\mct_{{\bf a},{\bf b},{\bf c}}$ is of infinite type as follows.

Suppose first $\gamma_1>1$. Then $\mct_{{\bf a},{\bf b},(\gamma_1)}$ is of infinite type by (i). Hence $\mct_{{\bf a},{\bf b},{\bf c}}$ is of infinite type.

Next suppose $\gamma_1+\gamma_2<n$. Then $\mct_{{\bf a},{\bf b},(\gamma_1+\gamma_2)}$ is of infinite type by (i). Hence $\mct_{{\bf a},{\bf b},{\bf c}}$ is of infinite type.

So we have only to consider the case of
${\bf c}=(1,n-1)$.
Let $\iota: \bbf^6\to\bbf^{2n+1}$ be as in (i) and define isotropic subspaces
\begin{align*}
V_1 & =\bbf e_{n-2}\oplus \bbf e_{n-1}\oplus U_{[\alpha_1-2]}, \quad
V_2 =\bbf e_{n+3}\oplus \bbf e_{n+4}\oplus U_{[\beta_1-2]}, \\
V_{4,\lambda} & =\iota(U_{4,\lambda}), \quad
V_5 =\iota(U_5)\oplus U_{[n-3]}
\end{align*}
of $\bbf^{2n+1}$ for $\lambda\in\bbf$. Then $(V_1,V_2,(V_{4,\lambda}\subset V_5))$ is a triple flag contained in $\mct_{{\bf a},{\bf b},{\bf c}}$. We can prove $Gt_\lambda \ne Gt_\mu$ if $\lambda\ne \mu$ as in the proof of Proposition \ref{prop1.1} using Lemma \ref{lem2.12}. Hence $\mct=\mct_{{\bf a},{\bf b},{\bf c}}$ is of infinite type.
\end{proof}

\subsection{A lemma for ${\rm O}_7(\bbf)$}

Consider $\bbf^7$ with the canonical basis $f_1,\ldots,f_7$ and the symmetric bilinear form $(\ ,\ )$ such that
$(f_i,f_j)=\delta_{i,8-j}$.
Write $U_+=\bbf f_1\oplus \bbf f_2$ and $U_-=\bbf f_6\oplus \bbf f_7$. Take a maximal isotropic subspace
$$U=\bbf(f_2+f_3)\oplus \bbf\left(f_1+f_4-{1\over 2}f_7\right)\oplus \bbf(f_5-f_6)$$
in $\bbf^7$.

\begin{lemma} If $g\bbf(f_1+f_2)=\bbf(f_1+f_2),\ gU_+=U_+,\ gU_-=U_-$ and $gU=U$ for $g\in G={\rm O}_7(\bbf)$, then $g=\pm {\rm id}$.
\label{lem2.16}
\end{lemma}

\begin{proof} It follows from $gU_+=U_+$ and $gU_-=U_-$ that $g$ is of the form
$$\bp A && 0 \\ & B & \\ 0 && J\,{}^tA^{-1}J \ep$$
with some $A\in{\rm GL}_2(\bbf)$ and $B\in {\rm O}_3(\bbf)$. It follows from $gU=U$ that
\begin{align*}
gf_2 & =af_2,\quad gf_3=af_3,\quad gf_5=a^{-1}f_5,\quad gf_6=a^{-1}f_6,\\
gf_1 & =\varepsilon f_1,\quad gf_4=\varepsilon f_4\mand gf_7=\varepsilon f_7
\end{align*}
with some $a\in\bbf^\times$ and $\varepsilon=\pm 1$. Finally it follows from $g\bbf(f_1+f_2)=\bbf(f_1+f_2)$ that $a=\varepsilon$. Hence $g=\pm {\rm id}$.
\end{proof}

Let $M_{U_+}\cong P^1(\bbf)$ denote the variety consisting of three-dimensional (maximal) isotropic subspaces in $\bbf^7$ containing $U_+$.

\begin{corollary} The triple flag variety $\mct=\mct_{(2),(3),(1,1,1)}$ for $G={\rm O}_7(\bbf)$ is of infinite type.
\end{corollary}

\begin{proof} For an element $W\in M_{U_+}$, we can take a triple flag
$$t_W=(U_-,U,(\bbf(f_1+f_2)\subset U_+\subset W))$$
in $\mct$. By lemma \ref{lem2.16}, we have $Gt_W\ne Gt_{W'}$ for two distinct element $W$ and $W'$ in $M_{U_+}$.
\end{proof}

\begin{remark} If $\bbf$ is algebraically closed, then the corollary also follows from
$$\dim \mct=7+6+9=22>21=\dim {\rm O}_7(\bbf).$$
\end{remark}

\subsection{Exclusion by Lemma \ref{lem2.16}}

As in Proposition \ref{prop2.15}, we may assume that $\mct=\mct_{{\bf a},{\bf b},{\bf c}}$ with
$${\bf a}=(\alpha_1),\quad {\bf b}=(\beta_1)\mand {\bf c}=(\gamma_1,\ldots, \gamma_r).$$
We may moreover assume $\alpha_1\le \beta_1$.

\begin{proposition} Suppose $r\ge 3$ and $\mct$ is of finite type. Then $\alpha_1=1$ or $\alpha_1=\beta_1=n$.
\label{prop2.19}
\end{proposition}

\begin{proof} We may assume $r=3$. Suppose $\alpha_1>1$. Then by Proposition \ref{prop2.15} (ii), we have $\beta_1=n$. Suppose $\alpha_1<n$. Define an isometric inclusion $\iota: \bbf^7\to \bbf^{2n+1}$ by $\iota(f_i)=e_{i+n-3}$. Then we can define isotropic subspaces
\begin{align*}
V_1 & =\bbf e_{n+3}\oplus \bbf e_{n+4}\oplus U_{[\alpha_1-2]}, \quad
V_2 =\iota(U)\oplus U_{[\beta_1-3]}, \\
V_3 & =\bbf(e_{n-2}+e_{n-1})\oplus U_{[\gamma_1-1]}, \quad
V_4 =\bbf e_{n-2}\oplus \bbf e_{n-1}\oplus U_{[\gamma_1+\gamma_2-2]} \\
\mand V_{5,W} & =\iota(W)\oplus U_{[\gamma_1+\gamma_2+\gamma_3-3]}
\end{align*}
of $\bbf^{2n+1}$ for $W\in M_{U_+}$. Then the triple flags $t_W=(V_1,V_2,(V_3\subset V_4\subset V_{5,W}))$ are contained in $\mct=\mct_{{\bf a},{\bf b},{\bf c}}$ for $W\in M_{U_+}$. We can prove $Gt_W \ne Gt_{W'}$ if $W$ and $W'$ are two distinct elements in $M_{U_+}$ as in the proof of Proposition \ref{prop1.1} using Lemma \ref{lem2.16}. Hence $\mct=\mct_{{\bf a},{\bf b},{\bf c}}$ is of infinite type.
\end{proof}

\subsection{Conclusion}

Combining Proposition \ref{prop2.15} and Proposition \ref{prop2.19}, we have the following theorem.

\begin{theorem} Suppose $1=p=q\le r$ and $\alpha_1\le\beta_1$. When $r=1$, suppose $\alpha_1\le \beta_1\le \gamma_1$. If $\mct=\mct_{{\bf a},{\bf b},{\bf c}}$ is of finite type, then $({\bf a},{\bf b},{\bf c})$ satisfies one of the following four conditions.

{\rm (I)} $\alpha_1=\beta_1=n$.

{\rm (II)} $\alpha_1=1$.

{\rm (III)} $r=1$ and $\gamma_1=n$.

{\rm (IV)} $r=2$ and $\beta_1=n$.
\label{th2.20}
\end{theorem}

\subsection{Proof of Proposition \ref{prop1.4}}

Proposition \ref{prop1.4} is equivalent to the following proposition.

\begin{proposition} If $\max(\alpha_1,\beta_1,\gamma_1)<n$ and $|\bbf^\times/(\bbf^\times)^2|=\infty$, then $\mct=\mct_{{\bf a},{\bf b},{\bf c}}$ is of infinite type.
\end{proposition}

\begin{proof} We may assume $({\bf a},{\bf b},{\bf c})=((\alpha_1),(\beta_1),(\gamma_1))$. Put
\begin{align*}
U_+ & =U_{[\alpha_1-1]}\oplus \bbf e_{n-1}, \quad
U_- =U_{[\beta_1-1]}\oplus \bbf e_{n+3}, \\
V_\lambda & =U_{[\gamma_1-1]}\oplus \bbf(e_{n-1}+e_n+\lambda e_{n+2}-\lambda e_{n+3}) \quad
\mand m_\lambda =(U_+,U_-,V_\lambda)\in \mct
\end{align*}
with $\lambda\in \bbf^\times$. Suppose $gm_\lambda=m_\mu$ for some $g\in G$. Then we have only to show that $\lambda\in \mu(\bbf^\times)^2$.

Since $gU_+=U_+$ and $gU_-=U_-$, we have
$$ge_{n-1}\in ae_{n-1}+U_{[\alpha_1-1]}\mand ge_{n+3}\in be_{n+3}+U_{[\beta_1-1]}$$
with some $a,b\in \bbf$. Since $(ge_{n-1},ge_{n+3})=(e_{n-1},e_{n+3})=1$, we have $a\in \bbf^\times$ and $b=a^{-1}$. On the other hand, since $e_n+\lambda e_{n+2}\in U_++U_-+V_\lambda$, we have
$$g(e_n+\lambda e_{n+2})\in U_++U_-+V_\mu.$$
Since $e_n+\lambda e_{n+2}$ is orthogonal to $e_{n-1}$ and $e_{n+3}$, we have
$$g(e_n+\lambda e_{n+2})=k(e_n+\mu e_{n+2})+U_{[\max(\alpha_1,\beta_1,\gamma_1)-1]}$$
with some $k\in \bbf^\times$. Hence we have
\begin{align*}
2\lambda & =(e_n+\lambda e_{n+2},e_n+\lambda e_{n+2})= (g(e_n+\lambda e_{n+2}),g(e_n+\lambda e_{n+2})) \\
& =(k(e_n+\mu e_{n+2}),k(e_n+\mu e_{n+2}))=2k^2\mu
\end{align*}
and therefore $\lambda\in \mu(\bbf^\times)^2$.
\end{proof}

\section{Orbits on $\mct_{(\alpha),(\beta),(n)}$}

\subsection{Preliminaries}

First we prepare some general notations and results. Write $\overline{i}=2n+2-i$ for $i\in I=\{1,\ldots,2n+1\}$. For $d\le n$, let
\begin{align*}
W & =\bbf e_1\oplus\cdots\oplus \bbf e_d,\quad \overline{W}=\bbf e_{\overline{d}}\oplus\cdots\oplus \bbf e_{\overline{1}}, \quad
U =\bbf e_{d+1}\oplus\cdots\oplus \bbf e_{\overline{d+1}}
\end{align*}
and $m=2n+1-2d=\dim U$. For $A\in {\rm GL}_d(\bbf)$ and $B\in {\rm O}_m(\bbf)$, define elements
\begin{align*}
\ell(A) & =\bp A && 0 \\ & I_m & \\ 0 && J_d\,{}^tA^{-1}J_d \ep
\mand \ell_0(B) = \bp I_d && 0 \\ & B & \\ 0 && I_d \ep
\end{align*}
of $G$. As in \cite{M2}, the maximal parabolic subgroup $P_W=\{g\in G\mid gW=W\}$ is decomposed as
$P_W=L_WL_UN_W=N_WL_WL_U$
where
$$L_W=\{\ell(A)\mid A\in {\rm GL}_d(\bbf)\},\quad L_U=\{\ell_0(B)\mid B\in {\rm O}_m(\bbf)\}$$
$$\mand N_W=\left\{\bp I_d & * & * \\ 0 & I_m & * \\ 0 & 0 & I_d \ep\in G\right\}.$$
The subgroup $L_WL_U\cong L_W\times L_U$ is a Levi subgroup of $P_W$ and $N_W$ is the unipotent radical of $P_W$ consisting of elements of the form
\begin{equation}
g(X,Z)=\bp I_d & X & Z \\ 0 & I_m & -J_m{}^tXJ_d \\ 0 & 0 & I_d \ep
\label{eq3.1''}
\end{equation}
with matrices $X=\{x_{i,j}\}^{i=1,\ldots,d}_{j=1,\ldots,m}$ and $Z=\{z_{i,j}\}^{i=1,\ldots,d}_{j=1,\ldots,d}$ satisfying
\begin{equation}
Z+J_d{}^tZJ_d=-XJ_m{}^tXJ_d.
\label{eq3.2''}
\end{equation}
So we have a bijection between $\bbf^{dm+d(d-1)/2}$ and $N_W$ given by
$$(X,\{z_{i,j}\}_{i+j\le d})\mapsto g(X,Z).$$
Put $Y=-J_m{}^tXJ_d$. Then $g(X,Z)$ is rewritten as
\begin{equation}
g(X,Z)=g'(Y,Z)=\bp I_d & -J_d{}^tYJ_m & Z \\ 0 & I_m & Y \\ 0 & 0 & I_d \ep.
\label{eq3.3''}
\end{equation}
The conditon (\ref{eq3.2''}) is rewritten as
\begin{equation}
Z+J_d{}^tZJ_d=-J_d{}^tYJ_mY.
\label{eq3.4'''}
\end{equation}

A pair of isotropic subspaces $(V,V')$ is called nondegenerate if
$$V\cap {V'}^\perp=V'\cap V^\perp=\{0\}.$$
Let $W'$ be an isotropic subspace of $\bbf^{2n+1}$ such that $(W,W')$ is nondegenerate. Then $W'$ is written as
$$W'=\bbf v_1\oplus\cdots\oplus \bbf v_d$$
with vectors
$$v_j=e_{2n+1-d+j}+\sum_{i=1}^d z_{i,j}e_i+\sum_{i=1}^m y_{i,j}e_{d+i}.$$
Define matrices $Z=\{z_{i,j}\}$ and $Y=\{y_{i,j}\}$. Since the condition (\ref{eq3.4'''}) is equivalent to $(v_i,v_j)=0$ for $i,j=1,\ldots,d$, we have:

\begin{lemma} $g'(Y,Z)\in G={\rm O}_{2n+1}(\bbf)$.
\label{lem6.2'}
\end{lemma}

\begin{corollary} Let $W=\bbf e_1\oplus\cdots\oplus \bbf e_d$ with $d\le n$ and let $W'$ be an isotropic subspace of $\bbf^{2n+1}$ such that $(W,W')$ is nondegenerate. Then there exists an $h\in G$ such that
$$hW'=\overline{W}=\bbf e_{\overline{d}}\oplus\cdots\oplus \bbf e_{\overline{1}}$$
and that $h$ acts trivially on $W$.
\end{corollary}

\begin{proof} $h=g'(Y,Z)^{-1}$ is a desired element.
\end{proof}

\begin{corollary} Let $(V,V')$ be a nondegenerate pair of isotropic subspaces in $\bbf^{2n+1}$ with $\dim V=\dim V'=\ell$. Then there exists a $g\in G$ such that
$$gV=\bbf e_1\oplus\cdots\oplus \bbf e_\ell\quad\mbox{and that}\quad gV'=\bbf e_{\overline{\ell}}\oplus\cdots\oplus \bbf e_{\overline{1}}.$$
\label{cor6.4}
\end{corollary}

Let $p_{\overline{W}}$ denote the projection
$$p_{\overline{W}}: \bbf^{2n+1}=W^\perp\oplus \overline{W}\to \overline{W}.$$

\begin{lemma} Let $V$ be a maximally isotropic subspace in $\bbf^{2n+1}$. Then

{\rm (i)} $p_{\overline{W}} (V)$ is the orthogonal space of $W\cap V$ in $\overline{W}$.

{\rm (ii)} $(W^\perp\cap V)/(W\cap V)$ is maximally isotropic in $W^\perp/W$.
\label{lem3.4}
\end{lemma}

\begin{proof} The spaces $W\cap V$ and $p_{\overline{W}} (V)$ are orthogonal since $W\perp W^\perp$ and $V$ is isotropic. Hence we have
\begin{equation}
\dim(W\cap V)+\dim p_{\overline{W}} (V)\le \dim W
\label{eq3.5}
\end{equation}
since $(W,\overline{W})$ is nondegenerate. On the other hand, we have
\begin{equation}
\dim((W^\perp\cap V)/(W\cap V))\le n-\dim W
\label{eq3.6}
\end{equation}
since $(W^\perp\cap V)/(W\cap V)$ is isotropic in $W^\perp/W$. Since
\begin{align*}
n & =\dim V=\dim (W^\perp\cap V)+\dim p_{\overline{W}} (V) \\
& =\dim ((W^\perp\cap V)/(W\cap V))+ \dim (W\cap V)+\dim p_{\overline{W}} (V),
\end{align*}
we have the equalities in (\ref{eq3.5}) and (\ref{eq3.6}).
\end{proof}

We can easily rewrite Lemma \ref{lem6.2'} in the following form.

\begin{lemma} Let $K$ be an index subset of $I$ such that $K\cap \overline{K}=\phi$. Let $v_k$ be vectors in $\bbf^{2n+1}$ of the form
$$v_k=e_k+\sum_{i\in I-K} c_{i,k}e_i$$
with some $c_{i,k}\in\bbf$ for $k\in K$ such that $(v_k,v_\ell)=0$ for $k,\ell\in K$. Then there exists a $g\in G$ such that
$$ge_k=\begin{cases} e_k & \text{if $k\in \overline{K}$}, \\
v_k & \text{if $k\in K$}, \\
e_k-\sum_{i\in \overline{K}} c_{\overline{k},\overline{i}}e_i & \text{if $k\in I-K-\overline{K}$}.
\end{cases}$$
\label{lem3.5}
\end{lemma}

\subsection{Normalization of $U_+$ and $U_-$}

Let $U_+$ and $U_-$ be $\alpha$ and $\beta$-dimensional isotropic subspaces of $\bbf^{2n+1}$, respectively. Define subspaces
$$W_0=U_+\cap U_-,\quad W_+=U_+\cap U_-^\perp\mand W_-=U_-\cap U_+^\perp$$
of $\bbf^{2n+1}$. Write
$a_0 =\dim W_0, \ 
a_+  =\dim W_+ -a_0$ and $a_- =\dim W_- -a_0$.
Since the bilinear form $(\ ,\ )$ is nondegenerate on the pair $(U_+/W_+,U_-/W_-)$, we have
$$\alpha-a_0-a_+=\beta-a_0-a_-.$$
Put $a_1=\alpha-a_0-a_+=\beta-a_0-a_-$ and $d=a_0+a_++a_-$. Define subspaces
\begin{align*}
W_{(0)} & =\bbf e_1\oplus\cdots\oplus \bbf e_{a_0}, \quad 
W_{(+)} =\bbf e_{a_0+1}\oplus\cdots\oplus\bbf e_{a_0+a_+}, \\ 
W_{(-)} & =\bbf e_{a_0+a_++1}\oplus\cdots\oplus\bbf e_d, \quad
U_{(+)} =\bbf e_{d+1}\oplus\cdots\oplus\bbf e_{d+a_1} \\
\quad\mbox{and}\quad U_{(-)} & =\bbf e_{\overline{d+a_1}}\oplus\cdots\oplus\bbf e_{\overline{d+1}}
\end{align*}
of $\bbf^{2n+1}$. Put $W=W_{(0)}\oplus W_{(+)}\oplus W_{(-)}=\bbf e_1\oplus\cdots\oplus \bbf e_d$. Let $\overline{W}$ and $U$ be as in Section 3.1.

\begin{proposition} There exists an element $g\in G$ such that
$$gU_+=W_{(0)}\oplus W_{(+)}\oplus U_{(+)}\quad\mbox{and that}\quad gU_-=W_{(0)}\oplus W_{(-)}\oplus U_{(-)}.$$
\label{prop3.1}
\end{proposition}

\begin{proof} Since $W_++W_-$ is an isotropic subspace of $\bbf^{2n+1}$ with $\dim(W_++W_-)=a_0+a_++a_-=d$, we can take a $g_1\in G$ such that $g_1(W_++W_-)=W$. Since $g_1W_+$ and $g_1W_-$ are $a_0+a_+$-dimensional and $a_0+a_-$-dimensional subspaces, respectively, of $W$, there exists a $g_2=\ell(A)$ with some $A\in {\rm GL}_d(\bbf)$ such that
\begin{equation}
g_2g_1W_+=W_{(0)}\oplus W_{(+)}\quad\mbox{and that}\quad g_2g_1W_-=W_{(0)}\oplus W_{(-)}
\label{eq3.1'}
\end{equation}
by Lemma \ref{lem6.1'} in the appendix.

Since the pair $((g_2g_1U_++W)/W,(g_2g_1U_-+W)/W)$ is a nondegenerate pair of isotropic subspaces in the factor space $W^\perp/W\cong U$, we can take a $g_3\in {\rm O}(U)$ such that
$$g_3g_2g_1U_++W=U_{(+)}\oplus W\quad\mbox{and that}\quad g_3g_2g_1U_-+W=U_{(-)}\oplus W$$
by Corollary \ref{cor6.4}. By (\ref{eq3.1'}) we can write
\begin{align*}
g_3g_2g_1U_+ & =W_{(0)}\oplus W_{(+)}\oplus \bbf v_1\oplus\cdots\oplus \bbf v_{a_1} \\
\mand g_3g_2g_1U_- & =W_{(0)}\oplus W_{(-)}\oplus \bbf v'_1\oplus\cdots\oplus \bbf v'_{a_1}
\end{align*}
with some vectors $v_1,\ldots,v_{a_1},v'_1,\ldots,v'_{a_1}$ of the form
$$v_j=e_{d+j}+\sum_{i=1}^{a_-} y_{i,j}e_{a_0+a_++i} \mand v'_j=e_{d'+j}+\sum_{i=1}^{a_+} y'_{i,j}e_{a_0+i}$$
for $j=1,\ldots,a_1$. Define a $d\times m$ matrix
$$X=\bp 0 & 0 & 0 \\ 0 & 0 & \{y'_{i,j}\} \\ \{y_{i,j}\} & 0 & 0 \ep$$
and $g_4=g(X,Z)\in N_W$ with a suitable $Z$. Then we have
$$g_4^{-1}v_j=e_{d+j}\mand g_4^{-1}v'_j=e_{d'+j}$$
for $j=1,\ldots,a_1$. Thus the element $g=g_4^{-1}g_3g_2g_1\in G$ satisfies the desired property.
\end{proof}

\subsection{Structure of $R=P_{U_+}\cap P_{U_-}$}

By Proposition \ref{prop3.1}, we may assume
$$U_+=W_{(0)}\oplus W_{(+)}\oplus U_{(+)}\quad\mbox{and that}\quad U_-=W_{(0)}\oplus W_{(-)}\oplus U_{(-)}.$$
Write $R=P_{U_+}\cap P_{U_-}=\{g\in G\mid gU_+=U_+,\ gU_-=U_-\}$.

\begin{proposition} {\rm (i)} $R=(N_W\cap R)(L_U\cap R)(L_W\cap R)$.

{\rm (ii)} $L_W\cap R$ consists of elements
$$\ell\left(\bp A & Y_1 & Y_2 \\ 0 & B & 0 \\ 0 & 0 & C \ep \right)$$
with $A\in {\rm GL}_{a_0}(\bbf),\ B\in {\rm GL}_{a_+}(\bbf),\ C\in {\rm GL}_{a_-}(\bbf),\ Y_1\in \mcm(a_0,a_+;\bbf)\ (=\{a_0\times a_+\mbox{-matrices with entries in }\bbf\})$ and $Y_2\in \mcm(a_0,a_-;\bbf)$.

{\rm (iii)} $L_U\cap R$ consists of elements $\ell_{00}(D,D')$ with $D\in {\rm GL}_{a_1}(\bbf)$ and $D'\in {\rm O}_{m-2a_1}(\bbf)$ where
$$\ell_{00}(D,D')=\ell_0\left(\bp D && 0 \\ & D' & \\ 0 && J_{a_1}{}^tD^{-1}J_{a_1} \ep \right).$$

{\rm (iv)} $N_W\cap R$ consists of elements $g(X,Z)\in N_W$ satisfying
\begin{align*}
x_{i,j}=0\quad\mbox{for }(i,j) & \in\{a_0+a_++1,\ldots,d\}\times\{1,\ldots,a_1\} \\
& \quad\sqcup \{a_0+1,\ldots,a_0+a_+\}\times\{m-a_1+1,\ldots,m\}.
\end{align*}
\label{prop3.2}
\end{proposition}

\begin{proof} (ii), (iii) and (iv) are clear. So we have only to prove (i). Let $g$ be an element of $R$. Since $gW=W$, $g$ is contained in $P_W$. So we can write
$g=g_Ng_Ug_W$
with $g_W\in L_W,\ g_U\in L_U$ and $g_N\in N_W$.

Since $g_Ng_U$ acts trivially on $W$, we have
$g_W(W_+)=W_+$ and $g_W(W_-)=W_-$.
Hence $g_W$ is of the form in (ii) and $g_W\in R$. This implies also $g_Ng_U=gg_W^{-1}\in R$.

Let $u$ be an element of $U_{(+)}$. Then $g_U(u)\in U$. We can write
$$g_Ng_U(u)=g_U(u)+w$$
with some $w\in W$. On the other hand, it follows from $g_Ng_U\in R$ that
$$g_Ng_U(u)\in U_+\subset U_{(+)}\oplus W.$$
Hence we have $g_U(u)\in U_{(+)}$. Thus we have proved
$g_U(U_{(+)})=U_{(+)}$.
We also have $g_U(U_{(-)})=U_{(-)}$ in the same way. So the element $g_U$ is of the form in (iii) and therefore $g_U\in R$. This implies also $g_N\in R$.

Thus we have proved
$$R\subset (N_W\cap R)(L_U\cap R)(L_W\cap R).$$
The converse inclusion is trivial.
\end{proof}

\subsection{Invariants of the $R$-orbit of $V\in M=M_{(n)}$} \label{sec3.3}

Let $V$ be a maximal isotropic subspace in $\bbf^{2n+1}$. Then the following $b_1,\ldots,b_{11}$ are clearly invariants of the $R$-orbit of $V$.
\begin{align*}
b_1 & =\dim(W_0\cap V),\quad b_2=a_0-b_1,\\
b_3 & =\dim(W_+\cap V)-b_1,\quad b_4=\dim(W_-\cap V)-b_1, \\
b_5 & =\dim(U_+\cap V)-b_1-b_3,\quad b_6=\dim(U_-\cap V)-b_1-b_4, \\
b_7 & =\dim(W\cap V)-b_1-b_3-b_4, \\
b_8 & =\dim((W_++U_-)\cap V)-\dim(W\cap V)-b_6, \\
b_9 & =\dim((U_++W_-)\cap V)-\dim(W\cap V)-b_5, \\
b_{10} & =a_+-b_3-b_7-b_8,\quad b_{11}=a_--b_4-b_7-b_9.
\end{align*}

We will deduce the remaining essential invariants from
$$X=X(V)=(U_++U_-)\cap V\mand X'=X'(V)=(U_++U_-)\cap V^\perp.$$
Consider the projection $\pi: W^\perp\to \pi(W^\perp)=W^\perp/W$. Then $\pi(W^\perp)$ is decomposed as
$$\pi(W^\perp)=\pi(U_+)\oplus\pi(U_-)\oplus\pi(Z)$$
where $Z=(U_++U_-)^\perp\subset W^\perp$. Put $\displaystyle{a_2=\frac{m-1}{2} -a_1=n-d-a_1}$. Then
$$\dim \pi(Z)=2a_2+1\mand \dim Z=2a_2+1+d.$$

Let $\pi_+,\pi_-$ and $\pi_Z$ denote the projections of $W^\perp$ onto
$\pi(U_+),\ \pi(U_-)$ and $\pi(Z)$,
respectively.

\begin{lemma} $a_1-\dim\pi(X')=a_2-\dim\pi(Z\cap V)$.
\label{lem3.3}
\end{lemma}

\begin{proof} Since $Z+V$ is the orthogonal space of $X'=(U_++U_-)\cap V^\perp$, we have
\begin{align*}
\dim X' & = 2n+1-\dim(Z+V) \\
& = 2n+1-\dim Z-\dim V+\dim(Z\cap V) \\
& =2n+1-(2a_2+1+d)-n+\dim(Z\cap V) \\
& =n-2a_2-d+\dim(Z\cap V) \\
& =a_1-a_2+\dim(Z\cap V).
\end{align*}
Since $\pi(X')\cong X'/(W\cap V^\perp),\ \pi(Z\cap V)\cong (Z\cap V)/(W\cap V)$ and $W\cap V^\perp=W\cap V$, we have the desired equality.
\end{proof}

The bilinear form $(\ ,\ )$ naturally induces a nondegenerate bilinear form on $W^\perp/W$. It is nondegenerate on the pair $(\pi(U_+),\pi(U_-))\cong (U_+/W_+,U_-/W_-)\cong (U_{(+)},U_{(-)})$. Put
$$X_0=((U_++W_-)\cap V)+((W_++U_-)\cap V).$$
For $v\in X'$, write $v=v_++v_-$ with $v_+\in U_++W_-$ and $v_-\in W_++U_-$. Then $\pi(v_+)$ is uniquely defined from $v$. Define a subspace
$$X_1=X_1(V)=\{v\in X'\mid (v_+,X')=\{0\}\}$$
of $X'$.

\begin{lemma} $X_0\subset X_1\subset X$.
\end{lemma}

\begin{proof} Let $v$ be an element of $X_0$. Then $v=v_++v_-$ with $v_+\in (U_++W_-)\cap V$ and $v_-\in (W_++U_-)\cap V$. So we have
$$(v_+,X')\subset (V,V^\perp)=\{0\}.$$
Thus we have $X_0\subset X_1$.

Let $v=v_++v_-$ be an element of $X'$ which is not contained in $X$. Then $(v,v)\ne 0$. This implies
$$(v_+,v)=(v_+,v_-)\ne 0.$$
Hence $v\notin X_1$. Thus we have proved $X_1\subset X$.
\end{proof}

Put
$b_{12}=\dim X_1-\dim X_0,\ b_{15}=\dim X'-\dim X_1$ and $\varepsilon=\dim X'-\dim X\in\{0,1\}$.

\begin{lemma} The projections $\pi_+$ and $\pi_-$ induce bijections
$$X'/X_0\stackrel{\sim}{\to} \pi_+(X')/\pi_+(X_0)\mand X'/X_0\stackrel{\sim}{\to} \pi_-(X')/\pi_-(X_0),$$
respectively.
\label{lem3.6}
\end{lemma}

\begin{proof} Note that the kernel of the projection $\pi_+|_{U_++U_-}:U_++U_-\to\pi(U_+)$ is $W_++U_-$. So the kernel of $\pi_+|_{X'}$ is $(W_++U_-)\cap X'=(W_++U_-)\cap V$. Since $\pi_+(X_0)=\pi((U_++W_-)\cap V^\perp)\cong ((U_++W_-)\cap V^\perp)/(W\cap V)$, we get the first bijection. In the same way, we also get the second one.
\end{proof}

\begin{corollary} $b_{12}\le a_1-\dim \pi(X')$.
\end{corollary}

\begin{proof} By Lemma \ref{lem3.6}, we have
$$\dim (\pi_+(X')\oplus \pi_-(X_1))=\dim \pi(X')+b_{12}.$$
On the other hand, $\pi_+(X')\oplus \pi_-(X_1)$ is an isotropic subspace of $\pi(U_++U_-)=\pi_+(U_+)\oplus \pi_-(U_-)\cong U_{(+)}\oplus U_{(-)}$ by the definition of $X_1$. So we have
$$\dim \pi(X')+b_{12}\le a_1.$$
\end{proof}

Put $b_{13}=a_1-\dim \pi(X')-b_{12}$ and $b_{14}=a_2-b_{12}-b_{13}$. Then $b_{14}=\dim \pi(Z\cap V)$ by Lemma \ref{lem3.3}.

By the definition of $X_1$, the bilinear form $(\ ,\ )$ is nondegenerate on the pair $(\pi_+(X')/\pi_+(X_1),\pi_-(X')/\pi_-(X_1))$. By Lemma \ref{lem3.6}, there exists a bijection
$$f: \pi_+(X')/\pi_+(X_1)\stackrel{\sim}{\to} \pi_-(X')/\pi_-(X_1)$$
induced by $\pi_-|_{X'}\circ \pi_+|_{X'}^{-1}$. Define a bilinear form $\langle\ ,\ \rangle$ on $\pi_+(X')/\pi_+(X_1)$ by
$$\langle u,v\rangle =(u,f(v))\quad \mbox{for }u,v\in \pi_+(X').$$
(This is well-defined since $(u,f(v))=0$ if $u$ or $v$ is contained in $\pi_+(X_1)$.)

\begin{lemma} If $u\in X'$ and $v\in X$, then $\langle \pi_+(u),\pi_+(v)\rangle =-\langle  \pi_+(v), \pi_+(u)\rangle$.
\end{lemma}

\begin{proof} Write $u=u_++u_-$ and $v=v_++v_-$ with $u_+,v_+\in U_+$ and $u_-,v_-\in U_-$. Then
$$0=(u,v)=(u_++u_-,v_++v_-)=(u_+,v_-)+(u_-,v_+)=\langle \pi_+(u),\pi_+(v)\rangle +\langle \pi_+(v),\pi_+(u)\rangle.$$
\end{proof}

\begin{corollary} $\varepsilon=0 \Longrightarrow b_{15}$ is even.
\label{cor3.9}
\end{corollary}

Summarizing the arguments in this subsection, we have:

\begin{proposition} The invariants $b_1,\ldots,b_{15}$ and $\varepsilon$ for $V$ satisfy the following equalities.
\begin{align}
a_0 & =b_1+b_2, \label{eq3.8} \\
a_+ & =b_3+b_7+b_8+b_{10}, \label{eq3.9} \\
a_- & =b_4+b_7+b_9+b_{11}, \label{eq3.10} \\
a_1 & =b_5+b_6+b_8+b_9+2b_{12}+b_{13}+b_{15}, \label{eq3.11'} \\
a_2 & =b_{12}+b_{13}+b_{14} \label{eq3.12} \\
\mand \varepsilon & =\begin{cases} 0 & \text{if $b_{15}=0$}, \\
1 & \text{if $b_{15}$ is odd}, \\
0\mbox{ or }1 & \text{if $b_{15}$ is even and positive}.
\end{cases} \label{eq3.13}
\end{align}
\label{prop3.15}
\end{proposition}

\begin{proof} The equality (\ref{eq3.11'}) follows from the definitions of $b_{12},b_{13},b_{15}$ and
$$\dim \pi(X_0)=b_5+b_6+b_8+b_9.$$
The other equalities follow from the definitions of $b_1,\ldots,b_{15}$ and Corollary \ref{cor3.9}.
\end{proof}

\subsection{Representative of the $R$-orbit of $V$}

Conversely suppose that nonnegative numbers $b_1,\ldots,b_{15}$ and $\varepsilon$ satisfy the equalities in Proposition \ref{prop3.15}. Then we define a maximally isotropic subspace
$$V(b_1,\ldots,b_{15},\varepsilon)=(\bigoplus_{j=1}^{14} V_{(j)})\oplus V_{(15)}^\varepsilon$$
as follows.

Define subsets $I_{(j)}$ of $I$ for $j\in J_1=\{1,2,3,4,5,6,10,11,14\}$ by
\begin{align*}
I_{(1)} & =\{1,\ldots,b_1\}, \quad
I_{(2)} =\{b_1+1,\ldots,a_0\}, \quad
I_{(3)} =\{a_0+1,\ldots,a_0+b_3\}, \\
I_{(4)} & =\{a_0+a_++1,\ldots,a_0+a_++b_4\}, \quad
I_{(5)} =\{d+1,\ldots,d+b_5\}, \\
I_{(6)} & =\{d'+1,\ldots,d'+b_6\}, \quad
I_{(10)} =\{a_0+a_+-b_{10}+1,\ldots,a_0+a_+\}, \\
I_{(11)} & =\{d-b_{11}+1,\ldots,d\}
\mand I_{(14)} =\{d+a_1+1,\ldots,d+a_1+b_{14}\}
\end{align*}
where $d'=\overline{d+a_1}-1=2n+1-d-a_1$.
For $j\in J_1$, put $U_{(j)}=\bigoplus_{i\in \widetilde{I_{(j)}}} \bbf e_i$ where $\widetilde{I_{(j)}}=I_{(j)}\sqcup \overline{I_{(j)}}$ and define maximally isotropic subspaces $V_{(j)}$ of $U_{(j)}$ by
$$V_{(j)}=\begin{cases} \bigoplus_{i\in I_{(j)}} \bbf e_i & \text{if $j=1,3,4,5,6,14$}, \\
\bigoplus_{i\in I_{(j)}} \bbf e_{\overline{i}} & \text{if $j=2,10,11$}.
\end{cases}$$

Define subsets $I_{(j)}$ of $I$ and maps $\eta_j: I_{(j)}\to I$ for $j\in J_2=\{7,8,9,13\}$ by
\begin{align*}
I_{(7)} & =\{a_0+b_3+1,\ldots,a_0+b_3+b_7\}, \\
& \eta_7(a_0+b_3+k) =a_0+a_++b_4+k\quad \mbox{for }k=1,\ldots,b_7 \\
I_{(8)} & =\{a_0+b_3+b_7+1,\ldots,a_0+b_3+b_7+b_8\}, \\
& \eta_8(a_0+b_3+b_7+k) =d'+b_6+k\quad \mbox{for }k=1,\ldots,b_8 \\
I_{(9)} & =\{a_0+a_++b_4+b_7+1,\ldots,a_0+a_++b_4+b_7+b_9\}, \\
& \eta_9(a_0+a_++b_4+b_7+k) =d+b_5+k\quad \mbox{for }k=1,\ldots,b_9, \\
I_{(13)} & =\{d+b_5+b_9+b_{12}+b_{15}+1,\ldots,d+b_5+b_9+b_{12}+b_{15}+b_{13}\} \\
\mand & \eta_{13}(d+b_5+b_9+b_{12}+b_{15}+k) =d+a_1+b_{14}+b_{12}+k\quad \mbox{for }k=1,\ldots,b_{13}.
\end{align*}
For $j\in J_2$, put $U_{(j)}=\bigoplus_{i\in \widetilde{I_{(j)}}} \bbf e_i$ where $\widetilde{I_{(j)}}=I_{(j)}\sqcup \eta_j(I_{(j)})\sqcup \overline{I_{(j)}}\sqcup \overline{\eta_j(I_{(j)})}$.
Define maximally isotropic subspaces $V_{(j)}=V_{(j)}^1\oplus V_{(j)}^2$ of $U_{(j)}$ for $j\in J_2$ by
$$V_{(j)}^1=\bigoplus_{i\in I_{(j)}} \bbf (e_i+e_{\eta_j(i)}) \mand  V_{(j)}^2=\bigoplus_{i\in I_{(j)}} \bbf (e_{\overline{i}}-e_{\overline{\eta_j(i)}}).$$

Define a subset $I_{(12)}=\{d+b_5+b_9+1,\ldots,d+b_5+b_9+b_{12}\}$ of $I$ and maps $\kappa,\lambda:I_{(12)}\to I$ given by
$$\kappa(d+b_5+b_9+k)=d'+b_6+b_8+k,\quad
\lambda(d+b_5+b_9+k)=d+a_1+b_{14}+k$$
for $k=1,\ldots,b_{12}$. Put $U_{(12)}=\bigoplus_{i\in \widetilde{I_{(12)}}} \bbf e_i$ where $\widetilde{I_{(12)}}=I_{(12)}\sqcup \kappa(I_{(12)})\sqcup \lambda(I_{(12)})\sqcup \overline{I_{(12)}}\sqcup \overline{\kappa(I_{(12)})}\sqcup \overline{\lambda(I_{(12)})}$.
Define a maximally isotropic subspace $V_{(12)}=V_{(12)}^1\oplus V_{(12)}^2\oplus V_{(12)}^3$ of $U_{(12)}$ by
\begin{align*}
V_{(12)}^1 & =\bigoplus_{i\in I_{(12)}} \bbf (e_i+e_{\kappa(i)}),\quad V_{(12)}^2=\bigoplus_{i\in I_{(12)}} \bbf (e_i+e_{\lambda(i)}) \\
\mand V_{(12)}^3 & =\bigoplus_{i\in I_{(12)}} \bbf (e_{\overline{i}}-e_{\overline{\kappa(i)}}-e_{\overline{\lambda(i)}}).
\end{align*}

Put $U_{(15)}=(\bigoplus_{i\in I_{15}\sqcup \overline{I_{(15)}}} \bbf e_i)\oplus \bbf e_{n+1}$ where
$$I_{(15)}=\{d+b_5+b_9+b_{12}+1,\ldots,d+b_5+b_9+b_{12}+b_{15}\}.$$
Define a map
$$\eta_{15}(d+b_5+b_9+b_{12}+k)=d'+b_6+b_8+b_{12}+b_{13}+k\quad\mbox{for }k=1,\ldots,b_{15}.$$
If $b_{15}=0$, then we define $V_{(15)}^\varepsilon=\{0\}$. Suppose $b_{15}>0$. Then we put
\begin{align*}
c & =d+b_5+b_9+b_{12}+\left[\frac{b_{15}+1}{2}\right] \mand I_{(15)}^+ =\{d+b_5+b_9+b_{12}+1,\ldots,c\}.
\end{align*}
For $i\in I_{(15)}^+-\{c\}$, define
$$V_{<i>}=\bbf (e_i+e_{\eta_{15}(i)})\oplus \bbf (e_{\overline{i}}-e_{\overline{\eta_{15}(i)}}).$$
Define
$$V_{<c>}^\varepsilon=\begin{cases} \bbf (e_c+e_{\eta_{15}(c)})\oplus \bbf (e_{\overline{c}}-e_{\overline{\eta_{15}(c)}}) & \text{if $b_{15}$ is even and $\varepsilon=0$}, \\
\bbf (e_c+e_{\eta_{15}(c)})\oplus \bbf (e_{\overline{c}}-e_{\overline{\eta_{15}(c)}}-\frac{1}{2} e_c+e_{n+1}) & \text{if $b_{15}$ is even and $\varepsilon=1$}, \\
\bbf (e_{\overline{c}}-\frac{1}{2} e_c+e_{n+1}) & \text{if $b_{15}$ is odd ($\varepsilon=1$)}.
\end{cases}$$
Then we define maximally isotropic subspaces $V_{(15)}^\varepsilon$ of $U_{(15)}$ by
$$V_{(15)}^\varepsilon=(\bigoplus_{i\in I_{(15)}^+-\{c\}} V_{<i>})\oplus V_{<c>}^\varepsilon.$$

\begin{theorem} Let $V$ be a maximally isotropic subspace of $\bbf^{2n+1}$. Define the numbers $b_1,\ldots,b_{15}$ and $\varepsilon$ as in Section \ref{sec3.3}. Then the $R$-orbit of $V$ contains the representative
$$V(b_1,\ldots,b_{15},\varepsilon).$$
\label{th3.10}
\end{theorem}

\begin{remark} If $V=V(b_1,\ldots,b_{15},\varepsilon)$, then we can write for $S=V,V^\perp,U_+,U_+^\perp,U_-$ and $U_-^\perp$,
$$S=\left(\bigoplus_{i\in I_{(1)}\sqcup\cdots\sqcup I_{(14)}} ((W_{<i>}\cap S)\oplus (\overline{W}_{<i>}\cap S))\right) \oplus (U_{(15)}\cap S)$$
where
$W_{<i>}=\bigoplus_{k\in I_{<i>}} \bbf e_k$ and $\overline{W}_{<i>}=\bigoplus_{k\in I_{<i>}} \bbf e_{\overline{k}}$
with
$$I_{<i>}=\begin{cases} \{i\} & \text{if $i\in I_{(j)}$ with $j\in J_1$}, \\
\{i,\eta_j(i)\} & \text{if $i\in I_{(j)}$ with $j\in J_2$}, \\
\{i,\kappa(i),\lambda(i)\} & \text{if $i\in I_{(12)}$}.
\end{cases}$$
So the subspaces $W_{<i>}\oplus \overline{W}_{<i>}$ are ``indecomposable summands'' in the sense of \cite{MWZ2}. They are characterized by
$${\bf d}(W)=((d_1^V(W),d_2^V(W),d_3^V(W)), (d_1^+(W),d_2^+(W),d_3^+(W)), (d_1^-(W),d_2^-(W),d_3^-(W)))$$
for $W=W_{<i>}$ and $W=\overline{W}_{<i>}$ where
\begin{align*}
d_1^V(W) & =\dim(W\cap V), \\
d_2^V(W) & =\dim(W\cap V^\perp)-\dim(W\cap V), \\
d_3^V(W) & =\dim W-\dim(W\cap V^\perp), \\
d_1^+(W) & =\dim(W\cap U_+), \\
d_2^+(W) & =\dim(W\cap U_+^\perp)-\dim(W\cap U_+), \\
d_3^+(W) & =\dim W-\dim(W\cap U_+^\perp), \\
d_1^-(W) & =\dim(W\cap U_-), \\
d_2^-(W) & =\dim(W\cap U_-^\perp)-\dim(W\cap U_-), \\
\mand d_3^-(W) & =\dim W-\dim(W\cap U_-^\perp).
\end{align*}
For $i\in I_{(j)}$ with $j=1,\ldots,14$, ${\bf d}(W_{<i>})$ and ${\bf d}(\overline{W}_{<i>})$ are as in Table \ref{table3.1}. Note that
$$d_k^*(W_{<i>})=d_{4-k}^*(\overline{W}_{<i>})$$
for $*=V,+,-$ and $k=1,2,3$. We write them in the ``compressed form'' as in \cite{MWZ2}: We omit $d_k^*(W)$ and $d_{4-k}^*(W)$ if both are $0$. We also omit commas.

\begin{table}
\centering
\setlength{\extrarowheight}{4pt}
\begin{center}
\begin{tabular}{ccc} \hline
$i\in I_{(j)}$ & ${\bf d}(W_{<i>})$ & ${\bf d}(\overline{W}_{<i>})$ \\ \hline
$I_{(1)}$ & ((10)(10)(10)) & ((01)(01)(01)) \\
$I_{(2)}$ & ((01)(10)(10)) & ((10)(01)(01)) \\
$I_{(3)}$ & (10)(10)(1)) & ((01)(01)(1)) \\
$I_{(4)}$ & ((10)(1)(10)) & ((01)(1)(01)) \\
$I_{(5)}$ & ((10)(10)(01)) & ((01)(01)(10)) \\
$I_{(6)}$ & ((10)(01)(10)) & ((01)(10)(01)) \\
$I_{(7)}$ & ((11)(110)(110)) & ((11)(011)(011)) \\
$I_{(8)}$ & ((11)(11)(110)) & ((11)(11)(011)) \\
$I_{(9)}$ & ((11)(110)(11)) & ((11)(011)(11)) \\
$I_{(10)}$ & (01)(10)(1)) & ((10)(01)(1)) \\
$I_{(11)}$ & ((01)(1)(10)) & ((10)(1)(01)) \\
$I_{(12)}$ & ((21)(111)(111)) & ((12)(111)(111)) \\
$I_{(13)}$ & ((11)(110)(011)) & ((01)(011)(110)) \\
$I_{(14)}$ & ((10)(1)(1)) & ((01)(1)(1)) \\
\hline
\end{tabular}
\end{center}
\setlength{\extrarowheight}{0pt}
\caption{${\bf d}(W_{<i>})$ and ${\bf d}(\overline{W}_{<i>})$}
\label{table3.1}
\end{table}

\bigskip
We can also decompose $U_{(15)}$ into indecomposable summands.

First suppose that $\varepsilon=0$. (Then $b_{15}$ is even.) Put
$W_{<i>}^0=\bbf e_i\oplus \bbf e_{\eta_{15}(i)},\ \overline{W}_{<i>}^0=\bbf e_{\overline{i}}\oplus \bbf e_{\overline{\eta_{15}(i)}}$
for $i\in I_{(15)}^+$ and $W_{<n+1>}=\bbf e_{n+1}$. Then
$$U_{(15)}=(\bigoplus_{i\in I_{(15)}^+} (W_{<i>}^0\oplus \overline{W}_{<i>}^0)) \oplus W_{<n+1>},$$
${\bf d}(W_{<i>}^0)={\bf d}(\overline{W}_{<i>}^0)=((11)(11)(11))$ and ${\bf d}(W_{<n+1>})=((1)(1)(1))$.

Next suppose that $b_{15}$ is even and $\varepsilon=1$. Put
$W_{<c>}^{\rm even}=\bbf e_c\oplus \bbf e_{\eta_{15}(c)}\oplus \bbf e_{\overline{c}}\oplus \bbf e_{\overline{\eta_{15}(c)}}\oplus \bbf e_{n+1}$.
Then
$$U_{(15)}=(\bigoplus_{i\in I_{(15)}^+-\{c\}} (W_{<i>}^0\oplus \overline{W}_{<i>}^0)) \oplus W_{<c>}^{\rm even}$$
and ${\bf d}(W_{<c>}^{\rm even})=((212)(212)(212))$.

Finally suppose that $b_{15}$ is odd. (Then $\varepsilon=1$.) Put
$W_{<c>}^{\rm odd}=\bbf e_c\oplus \bbf e_{\overline{c}}\oplus \bbf e_{n+1}$.
Then
$$U_{(15)}=(\bigoplus_{i\in I_{(15)}^+-\{c\}} (W_{<i>}^0\oplus \overline{W}_{<i>}^0)) \oplus W_{<c>}^{\rm odd}$$
and
${\bf d}(W_{<c>}^{\rm odd})=((111)(111)(111))$.
\label{rem3.11}
\end{remark}

\subsection{Proof of Theorem \ref{th3.10}}

(i) $V_{(1)}$-part:
There exists an element $g\in R$ of the form
$$g=\ell\left(\bp A & 0 \\ 0 & I_{a_++a_-} \ep\right)$$
with $A\in {\rm GL}_{a_0}(\bbf)$ such that
$$g(W_0\cap V)=V_{(1)}=\bbf e_1\oplus\cdots\oplus \bbf e_{b_1}.$$
Put $V'_1=gV$ and let $U_1$ be the orthogonal complement of $U_{(1)}=V_{(1)}\oplus \overline{V}_{(1)}$ in $\bbf^{2n+1}$ where $\overline{V}_{(1)}=\bigoplus_{i\in I_{(1)}} \bbf e_{\overline{i}}$. Then
$$V'_1=gV=V_{(1)}\oplus V_1.$$
with $V_1=U_1\cap V'_1$. So we have only to consider the $({\rm O}(U_1)\cap R)$-orbit of $V_1$ in the following. Since $gV\cap W_0=g(V\cap W_0)=V_{(1)}$, we have
$$V_1\cap W_0=\{0\}.$$

\bigskip
(ii) $V_{(2)}$-part:
Let $U_2$ denote the orthogonal complement of $U_{(2)}=\overline{V}_{(2)}\oplus V_{(2)}$ in $U_1$ where $\overline{V}_{(2)}=\bigoplus_{i\in I_{(2)}} \bbf e_i$. Then $U_2=\bbf e_{a_0+1}\oplus\cdots\oplus \bbf e_{\overline{a_0+1}}$ and
$$U_1=\overline{V}_{(2)}\oplus U_2\oplus V_{(2)}.$$
Let $p$ denote the projection of $U_1$ onto $V_{(2)}$ with respect to this direct sum decomposition. Since $V_1\cap \overline{V}_{(2)}=\{0\}$, we have $p(V_1)=V_{(2)}$ by Lemma \ref{lem3.4} (i). So we can write
$$V_1=(V_1\cap (\overline{V}_{(2)}\oplus U_2))\oplus \bigoplus_{j\in \overline{I_{(2)}}} \bbf v_j$$
with vectors
$v_j\in e_j+(\overline{V}_{(2)}\oplus U_2)$
for $j\in \overline{I_{(2)}}$. By Lemma \ref{lem3.5}, we can take a $g_1\in {\rm O}(U_1)$ such that
$$g_1e_j\begin{cases} =e_j & \text{if $j\in I_{(2)}$}, \\
=v_j & \text{if $j\in \overline{I_{(2)}}$}, \\
\in e_j+\overline{V}_{(2)} & \text{if $j\in I-I_{(2)}-\overline{I_{(2)}}$}.
\end{cases}$$
Since $\overline{V}_{(2)}=\bigoplus_{i\in I_{(2)}} \bbf e_i \subset W_0$, we have $g_1U_\pm=U_\pm$ and hence $g_1\in R$.
Since $g_1^{-1}V_1\supset V_{(2)}$, we have
$g_1^{-1}V_1\subset U_2\oplus V_{(2)}$
and hence
$$g_1^{-1}V_1=(g_1^{-1}V_1\cap U_2)\oplus V_{(2)}=V_2\oplus V_{(2)}$$
where $V_2=g_1^{-1}V_1\cap U_2$. So we have only to consider $({\rm O}(U_2)\cap R)$-orbit of $V_2$ in the following since $V'_2=g_1^{-1}V'_1$ is written as $V'_2=V_{(1)}\oplus V_{(2)}\oplus V_2$.

\bigskip
(iii) $V_{(3)}$ and $V_{(4)}$-parts:
We can take an element
$$g_2=\ell\left(\bp I_{a_0} && 0 \\ & A & \\ 0 && B \ep\right)\in {\rm O}(U_2)\cap R\cap L_W$$
with some $A\in {\rm GL}_{a_+}(\bbf)$ and $B\in {\rm GL}_{a_-}(\bbf)$ such that
$$g_2(W_+\cap V_2)=V_{(3)}\mand g_2(W_-\cap V_2)=V_{(4)}.$$
Let $U_4$ denote the orthogonal complement of $U_{(3)}\oplus  U_{(4)}$ in $U_2$. Then we have
$$g_2V_2=V_{(3)}\oplus V_{(4)}\oplus V_4$$
with $V_4=g_2V_2\cap U_4$. So we have only to consider the $({\rm O}(U_4)\cap R)$-orbit of $V_4$ in the following since $V'_4=g_2V'_2$ is written as $V'_4=V_{(1)}\oplus V_{(2)}\oplus V_{(3)}\oplus V_{(4)}\oplus V_4$. Note that
\begin{equation}
V_4\cap W_+=V_4\cap W_-=\{0\}.
\label{eq3.2'}
\end{equation}

\bigskip
(iv) $V_{(5)}$ and $V_{(6)}$-parts:
Let $p: W\oplus U\to U$ denote the canonical projection map. (Note that $p$ is the composition of $\pi:W^\perp=W\oplus U\to W^\perp/W$ and the identification $W^\perp/W\cong U$.)
It follows from (\ref{eq3.2'}) that $\dim p(V_4\cap U_+)=b_5$ and that $\dim p(V_4\cap U_-)=b_6$. Since $(p(V_4\cap U_+),p(V_4\cap U_-))=\{0\}$, we have
$$p(V_4\cap U_+)\subset U_{(+)}^{\perp p(V_4\cap U_-)}.$$
Hence we can take an element
$$g_3=\ell_{00}(A,I_{m-2a_1})\in R\cap L_U$$
with some $A\in {\rm GL}_{a_1}(\bbf)$ such that
\begin{align*}
g_3 p(V_4\cap U_+) & =V_{(5)} \\
\mbox{and that}\quad g_3U_{(+)}^{\perp p(V_4\cap U_-)} & =U_{(+)}^{\perp V_{(6)}}\quad(\Longleftrightarrow g_3 p(V_4\cap U_-)=V_{(6)}).
\end{align*}

We can write $V_4\cap U_+=\bbf v_1\oplus\cdots\oplus \bbf v_{b_5}$ and $V_4\cap U_-=\bbf v'_1\oplus\cdots\oplus \bbf v'_{b_6}$ with vectors
$$v_j=e_{d+j}+\sum_{i=a_0+b_3+1}^{a_0+a_+} y_{i,j}e_i$$
for $j=1,\ldots,b_5$ and
$$v'_j=e_{d'+j}+\sum_{i=a_0+a_++b_4+1}^{d} y'_{i,j}e_i$$
for $j=1,\ldots,b_6$ with some matrices $\{y_{i,j}\}$ and $\{y'_{i,j}\}$. Define a matrix $X=\{x_{i,j}\}^{i=1,\ldots,d}_{j=1,\ldots,m}$ by
$$x_{i,j}=\begin{cases} y_{i,j} & \text{if $(i,j)\in\{a_0+b_3+1,\ldots,a_0+a_+\}\times \{1,\ldots,b_5\}$}, \\
y'_{i,d'-d+j} & \text{if $(i,j)\in\{a_0+a_++b_4+1,\ldots,d\}\times \{d'-d+1,\ldots,d'-d+b_6\}$}, \\
0 & \text{otherwise}
\end{cases}$$
and put $g_4=g(X,0)\in {\rm O}(U_4)\cap R\cap N_W$ (Proposition \ref{prop3.2} (iv)). Then we have
\begin{align*}
g_4^{-1}v_j & =e_{d+j}\quad\mbox{for }j=1,\ldots,b_5 
\mand g_4^{-1}v'_j =e_{d'+j}\quad\mbox{for }j=1,\ldots,b_6.
\end{align*}
Thus we have
$$g_4^{-1}g_3(V_4\cap U_+)=V_{(5)}\mand g_4^{-1}g_3(V_4\cap U_-)=V_{(6)}.$$

Take the orthogonal complement $U_6$ of $U_{(5)}\oplus U_{(6)}$ in $U_4$. Then we have
$$g_4^{-1}g_3V_4=V_{(5)}\oplus V_{(6)}\oplus V_6$$
with $V_6=g_4^{-1}g_3V_4\cap U_6$. So we have only to consider the $({\rm O}(U_6)\cap R)$-orbit of $V_6$ in the following since $V'_6=g_4^{-1}g_3V'_4$ is written as $V'_6=(\bigoplus_{j=1}^6 V_{(j)})\oplus V_6$. Note that
\begin{equation}
V_6\cap U_+=V_6\cap U_-=\{0\}.
\label{eq3.3'}
\end{equation}
Let
$p_\pm: U_6\cap(U_++U_-)=(U_6\cap U_+)\oplus (U_6\cap U_-)\to U_6\cap U_\pm$
denote the projections with respect to the direct sum decomposition. By (\ref{eq3.3'}), we have:

\begin{proposition} The projections give injections
$$p_\pm:V_6\cap(U_++U_-)\to U_6\cap U_\pm.$$
\label{prop3.3}
\end{proposition}

\bigskip
(v) $V_{(7)}$-part:
Let $U_7$ denote the orthogonal complement of $U_{(7)}$ in $U_6$. Then we can write $U_7=\bigoplus_{i\in I_7} \bbf e_i$ with $I_7=I_6-\widetilde{I_{(7)}}=I-(\bigsqcup_{j=1}^7 \widetilde{I_{(j)}})$. 

By Proposition \ref{prop3.3}, we can take an element
$$g_5=\ell\left(\bp I_{a_0+b_3} &&& 0 \\ & A && \\ && I_{b_4} & \\ 0 &&& B \ep\right)\in {\rm O}(U_6)\cap R\cap L_W$$
with some $A\in {\rm GL}_{a_+-b_3}(\bbf)$ and $B\in {\rm GL}_{a_--b_4}(\bbf)$ such that
$$g_5V_6\cap W=V_{(7)}^1.$$
Since $\overline{W}\cap U_6=\bigoplus_{i\in I_{(7)}\sqcup I_{(8)}\sqcup I_{(9)}\sqcup I_{(10)}\sqcup I_{(11)}} \bbf e_{\overline{i}}$, it follows that
\begin{equation}
p_{\overline{W}}(g_5V_6)=V_{(7)}^2\oplus \bigoplus_{i\in I_{(8)}\sqcup I_{(9)}\sqcup I_{(10)}\sqcup I_{(11)}} \bbf e_{\overline{i}}
\label{eq3.11}
\end{equation}
by Lemma \ref{lem3.4} (i).

By (\ref{eq3.11}) we can write
$$g_5V_6=((V_{(7)}^1\oplus U_7)\cap g_5V_6)\oplus \bigoplus_{j\in I_{(7)}} \bbf v_j$$
with vectors
$$v_j=e_{\overline{j}}-e_{\overline{\eta_7(j)}}+\sum_{i\in I_6\cap I_{W^\perp}} c_{i,j}e_i$$
for $j\in I_{(7)}$ with $c_{i,j}\in\bbf$ where $I_{W^\perp}=\{i\in I\mid e_i\in W^\perp\}$.

Put
$$u_j=-e_{\overline{\eta_7(j)}}+\sum_{i\in I_{U_{(+)}}\cap I_6} c_{i,j}e_i
\mand
v_j^+=u_j -\sum_{i\in I_{(7)}} (u_j,v_i)e_i$$
for $j\in I_{(7)}$ where $I_{U_{(+)}}=\{i\in I\mid e_i\in U_{(+)}\}=\{d+1,\ldots,d+a_1\}$. Then $(v_j^+,v_k^+)=0$ and
$$(v_j^+,v_k)=(u_j,v_k)-\sum_{i\in I_{(7)}} (u_j,v_i)(e_i,v_k) =(u_j,v_k)-(u_j,v_k)=0$$
for $j,k\in I_{(7)}$. Write $v_j^-=v_j-v_j^+$. Then we can take an element $g_6\in {\rm O}(U_6)$ such that
$$g_6e_j=e_j,\quad g_6e_{\eta_7(j)}=e_{\eta_7(j)},\quad g_6e_{\overline{j}}=v_j^-\quad\mbox{and that}\quad -g_6e_{\overline{\eta_7(j)}}=v_j^+$$
for $j\in I_{(7)}$ since $(v_j^+,v_k^+)=(v_j^-,v_k^-)=(v_j^+,v_k^-)=0$ for $j,k\in I_{(7)}$ by Lemma \ref{lem3.5}. The element $g_6$ is also contained in $R\cap N_W$ by its construction. Hence we have
$$g_6^{-1}g_5V_6=V_{(7)}^1\oplus (U_7\cap g_6^{-1}g_5V_6)\oplus V_{(7)}^2.$$
So we have only to consider the $({\rm O}(U_7)\cap R)$-orbit of $V_7=U_7\cap g_6^{-1}g_5V_6$ in the following since $V'_7=g_6^{-1}g_5V'_6$ is written as $V'_7=(\bigoplus_{j=1}^7 V_{(j)}) \oplus V_7$.

\bigskip
(vi) $V_{(8)}$-part:
Let $U_8$ denote the orthogonal complement of $U_{(8)}$ in $U_7$. Then we can write $U_8=\bigoplus_{i\in I_8} \bbf e_i$ with $I_8=I_7-\widetilde{I_{(8)}}$. 

By Proposition \ref{prop3.3}, $\dim p_+((W_++U_-)\cap V_7)=\dim p_-((W_++U_-)\cap V_7)=b_8$. Furthermore $p_-((W_++U_-)\cap V_7)\cap W_-=\{0\}$ by the definition of $V_7$. So we can take elements
\begin{align*}
g_7 & =\ell\left(\bp I_{a_0+b_3+b_7} && 0 \\ & A & \\ 0 && I_{a_-} \ep\right)\in {\rm O}(U_7)\cap R\cap L_W \\
\mand g_8 & =\ell_{00}\left(\bp I_{b_5} && 0 \\ & B & \\ 0 && I_{b_6} \ep,I_{m-2a_1}\right) \in {\rm O}(U_7)\cap R\cap L_U
\end{align*}
with $A\in {\rm GL}_{a_+-b_3-b_7}(\bbf)$ and $B\in {\rm GL}_{a_1-b_5-b_6}(\bbf)$ such that
$$g_8g_7((W_++U_-)\cap V_7)=\bigoplus_{j\in I_{(8)}} \bbf(e_j+w_j)$$
where $w_j\in e_{\eta_8(j)}+(W_-\cap V_7)$. By Proposition \ref{prop3.2} (iv), we can take an element $g_9\in {\rm SO}(U_7)\cap R\cap N_W$ such that
$g_9e_{\eta_8(j)}=w_j$
for $j=1,\ldots,b_8$. Thus we have
$$g_9^{-1}g_8g_7((W_++U_-)\cap V_7)=V_{(8)}^1.$$

Since $g_9^{-1}g_8g_7V_7\cap W=\{0\}$, it follows from Lemma \ref{lem3.4} that $p_{\overline{W}}(g_9^{-1}g_8g_7V_7)=U_7\cap \overline{W}= \bigoplus_{i\in I_{(8)}\sqcup I_{(9)}\sqcup I_{(10)}\sqcup I_{(11)}} \bbf e_{\overline{i}}$. So we can take vectors $v_j$ in $g_9^{-1}g_8g_7V_7$ for $j\in I_{(8)}$ of the form
$$v_j=e_{\overline{j}}+\sum_{i\in I_7\cap I_{W^\perp}} c_{i,j} e_i.$$
Since $g_9^{-1}g_8g_7V_7$ contains $V_{(8)}^1$, they are of the form
$$v_j=e_{\overline{j}}-e_{\overline{\eta_8(j)}}+\sum_{i\in (I_7\cap I_{W^\perp})-\overline{\eta_8(I_{(8)})}} c_{i,j} e_i.$$

Define vectors
$$u_j=-e_{\overline{\eta_8(j)}}+\sum_{i\in (I_7\cap I_{U_{(+)}})-\overline{\eta_8(I_{(8)})}} c_{i,j}e_i\mand v_j^+=u_j-\sum_{i\in I_{(8)}} (u_j,v_i)e_i$$
for $j\in I_{(8)}$. Then we have $(v_j^+,v_k^+)=0$ and
$$(v_j^+,v_k)=(u_j,v_k)-\sum_{i\in I_{(8)}} (u_j,v_i)(e_i,v_k) =(u_j,v_k)-(u_j,v_k)=0$$
for $j,k\in I_{(8)}$. Put $v_j^-=v_j-v_j^+$.  Since $(v_j^+,v_k^+)=(v_j^-,v_k^-)=(v_j^+,v_k^-)=0$ for $j,k\in I_{(8)}$, we can take an element $g_{10}\in {\rm O}(U_7)$ such that
$$g_{10}e_j=e_j,\quad g_{10}e_{\eta_8(j)}=e_{\eta_8(j)},\quad g_{10}e_{\overline{\eta_8(j)}}=-v_j^+\quad\mbox{and that}\quad g_{10}e_{\overline{j}}=v_j^-$$
by Lemma \ref{lem3.5}. The element $g_{10}$ is also contained in $R\cap N_W$ by its construction. Hence we have
$$g_{10}^{-1}g_9^{-1}g_8g_7V_7=V_{(8)}^1\oplus (U_8\cap g_{10}^{-1}g_9^{-1}g_8g_7V_7) \oplus V_{(8)}^2.$$
So we have only to consider the $({\rm O}(U_8)\cap R)$-orbit of $V_8=U_8\cap g_{10}^{-1}g_9^{-1}g_8g_7V_7$ in the following since $V'_8=g_{10}^{-1}g_9^{-1}g_8g_7V'_7$ is written as $V'_8=(\bigoplus_{j=1}^8 V_{(j)})\oplus V_8$.

\bigskip
(vii) $V_{(9)}$-part:
Let $U_9$ denote the orthogonal complement of $U_{(9)}$ in $U_8$. Then we can write $U_9=\bigoplus_{i\in I_9} \bbf e_i$ with $I_9=I_8-\widetilde{I_{(9)}}$. 
We can take an element $g'\in{\rm O}(U_8)\cap R$ such that
$$g'V_8=V_{(9)}^1\oplus (g'V_8\cap U_9)\oplus V_{(9)}^2$$
in the same way as in (vi).
So we have only to consider the $({\rm O}(U_9)\cap R)$-orbit of $V_9=U_9\cap g'V_8$ since $V'_9=g'V'_8$ is written as $V'_9=(\bigoplus_{j=1}^9 V_{(j)})\oplus V_9$.

\bigskip
(viii) $V_{(10)}$ and $V_{(11)}$-parts:
Considering $V'_9$ instead of $V$, we may define
$$X=X(V'_9)=(U_++U_-)\cap V'_9,\quad X'=X'(V'_9)=(U_++U_-)\cap (V'_9)^\perp$$
$$\mand X_0=X_0(V'_9)=((W_++U_-)\cap V'_9)+((U_++W_-)\cap V'_9).$$
Since $V'_9=V_{(1)}\oplus\cdots\oplus V_{(9)}\oplus V_9$, we have
$$(V'_9)^\perp=V_{(1)}\oplus\cdots\oplus V_{(9)}\oplus (U_9\cap V_9^\perp).$$
We see that
$$X_0=V_{(1)}\oplus V_{(3)}\oplus V_{(4)}\oplus V_{(5)}\oplus V_{(6)}\oplus V_{(7)}^1\oplus V_{(8)}^1\oplus V_{(9)}^1.$$
So it follows from Lemma \ref{lem3.6} that the maps
\begin{align*}
\pi_+|_{U_9\cap X'} & :U_9\cap X'\to \pi_+(U_9\cap X')\subset U_{(+)} \\
\mand 
\pi_-|_{U_9\cap X'} & :U_9\cap X'\to \pi_-(U_9\cap X')\subset U_{(-)}
\end{align*}
are bijective. Here we identify $U_\pm/W_\pm$ with $U_{(\pm)}$ and consider the image of $\pi_\pm |_{U_9\cap X'}$ as subspaces of $U_{(\pm)}$. There exist linear maps
$$\varphi_+: U_9\cap X'\to W_+\cap U_9=\overline{V}_{(10)} \mand \varphi_-: U_9\cap X'\to W_-\cap U_9=\overline{V}_{(11)}$$
such that
$$v=\varphi_+(v)+\varphi_-(v)+\pi_+(v)+\pi_-(v)\quad\mbox{for }v\in U_9\cap X'.$$
Consider the maps
$$\widetilde{\varphi_\pm}=\varphi_\pm\circ (\pi_\pm|_{U_9\cap X'})^{-1}: \pi_\pm(U_9\cap X')\to U_9\cap W_{(\pm)}$$
and extend them linearly to $U_9\cap U_{(\pm)}$. Then we can take an element $h\in {\rm O}(U_9)\cap R\cap N_W$ such that
$$h(u)=u+\widetilde{\varphi_\pm}(u)$$
for $u\in U_9\cap U_{(\pm)}$ by Proposition \ref{prop3.2} (iv). For $v\in U_9\cap X'$, we have
\begin{align*}
h^{-1}v & =h^{-1}(\varphi_+(v)+\pi_+(v)+\varphi_-(v)+\pi_-(v)) \\
& =h^{-1}(\widetilde{\varphi_+}\pi_+(v)+\pi_+(v)+\widetilde{\varphi_-}\pi_-(v)+\pi_-(v)) \\
& =h^{-1}(h\pi_+(v)+h\pi_-(v))=\pi_+(v)+\pi_-(v) \in U_9\cap U.
\end{align*}
Hence $h^{-1}(U_9\cap X')\subset U$. This implies
$$(e_i,h^{-1}(U_9\cap X'))=\{0\}\quad\mbox{for }i\in \overline{I_{(10)}}\sqcup \overline{I_{(11)}}.$$

Since the orthogonal space of $h^{-1}(U_9\cap X')=U_9\cap h^{-1}X'=U_9\cap (U_++U_-)\cap h^{-1}V_9^\perp$ in $U_9$ is $(U_9\cap Z)+(U_9\cap h^{-1}V_9)$, we can write
$$e_i=v_i-z_i$$
with $v_i\in U_9\cap h^{-1}V_9$ and $z_i\in U_9\cap Z$ for $i\in \overline{I_{(10)}}\sqcup \overline{I_{(11)}}$. By Lemma \ref{lem3.5}, we can take an element $h_2\in {\rm SO}(U_9)\cap R\cap N_W$ such that
$$h_2e_i=e_i+z_i=v_i$$
for $i\in \overline{I_{(10)}}\sqcup \overline{I_{(11)}}$. So we have
$$h_2^{-1}h^{-1}V_9\supset V_{(10)}\oplus V_{(11)}.$$
Let $U_{11}$ be the orthogonal complement of $\overline{V}_{(10)}\oplus \overline{V}_{(11)}\oplus V_{(10)}\oplus V_{(11)}=U_9\cap (W\oplus \overline{W})$ in $U_9$. Then
$$h_2^{-1}h^{-1}V_9=(h_2^{-1}h^{-1}V_9\cap U_{11})\oplus V_{(10)}\oplus V_{(11)}.$$
Put $V_{11}=h_2^{-1}h^{-1}V_9\cap U_{11}$ and $V'_{11}=h_2^{-1}h^{-1}V'_9$. Then $V'_{11}=(\bigoplus_{j=1}^{11} V_{(j)})\oplus V_{11}$.

\bigskip
(ix) $V_{(12)},V_{(13)},V_{(14)}$ and $V_{(15)}^\varepsilon$-parts:
Put $I_{11}=I_9-(\widetilde{I_{(10)}}\sqcup \widetilde{I_{(11)}})=I-\bigsqcup_{j=1}^{11} \widetilde{I_{(j)}}$. Then
\begin{align*}
\begin{split}
I_{11} & =\{d+b_5+b_9+1,\ldots,d+a_1-b_6-b_8,d+a_1+1,\ldots,\overline{d+a_1+1}, \\
& \qquad\qquad \overline{d+a_1-b_6-b_8},\ldots,\overline{d+b_5+b_9+1}\}.
\end{split}
\end{align*}
With respect to the basis $\{e_i\mid i\in I_{11}\}$ of $U_{11}$, every element of ${\rm O}(U_{11})\cap R$ is represented by a matrix of the form
$$\ell(A,B)=\bp A && 0 \\ & B & \\ 0 && J_{a'_1}{}^tA^{-1}J_{a'_1} \ep$$
with $A\in {\rm GL}_{a'_1}(\bbf)$ and $B\in {\rm O}_{2a_2+1}(\bbf)$ where $a'_1=a_1-b_5-b_6-b_8-b_9$.

Write $Y=U_{11}\cap X(V'_{11}),\ Y'=U_{11}\cap X'(V'_{11})$ and $Y_1=U_{11}\cap X^1(V'_{11})$. Then
$$\dim Y'=b_{12}+b_{15},\quad \dim Y_1=b_{12}\mand \dim Y'-\dim Y=\varepsilon.$$
As is shown in (viii), the maps $\pi_\pm|_{Y'}: Y'\to \pi_{\pm}(Y')\subset U_{11}\cap U_\pm$ are bijective.

Put
\begin{align*}
U_{(12)}^{+1} & =\bigoplus_{i\in I_{(12)}} \bbf e_i,\quad U_{(12)}^{-2}=\bigoplus_{i\in I_{(12)}} \bbf e_{\kappa(i)},\quad U_{(13)}^+=\bigoplus_{i\in I_{(13)}} \bbf e_i, \\
U_{(15)}^+ & =\bigoplus_{i\in I_{(15)}} \bbf e_i \mand U_{(15)}^-=\bigoplus_{i\in I_{(15)}} \bbf e_{\overline{i}}.
\end{align*}
Note that the space $\pi_-(Y')$ is determined by its orthogonal space
$(U_{11}\cap U_+)^{\perp \pi_-(Y')}$
in $U_{11}\cap U_+$.
Since $\dim \pi_+(Y')=b_{12}+b_{15},\ \dim(U_{11}\cap U_+)^{\perp \pi_-(Y')}=b_{12}+b_{13}$ and since the dimension of $\pi_+(Y')\cap (U_{11}\cap U_+)^{\perp \pi_-(Y')}=\pi_+(Y_1)$ is $b_{12}$, we can take an element $\ell_1=\ell(A,I_{2a_2+1})$ with some $A\in {\rm GL}_{a'_1}(\bbf)$ such that
\begin{align*}
\pi_+(\ell_1 Y') & =\ell_1\pi_+(Y')=U_{(12)}^{+1}\oplus U_{(15)}^+, \\
(U_{11}\cap U_+)^{\perp \pi_-(\ell_1 Y')} & =\ell_1 (U_{11}\cap U_+)^{\perp \pi_-(Y')} =U_{(12)}^{+1}\oplus U_{(13)}^+ \\
\mbox{and that}\quad \pi_+(\ell_1 Y_1) & =\ell_1\pi_+(Y_1) =U_{(12)}^{+1}
\end{align*}
by Lemma \ref{lem6.1'} in the appendix. The second formula implies
$\pi_-(\ell_1 Y')=U_{(12)}^{-2}\oplus U_{(15)}^-$.

Consider the bijective linear map
\begin{align*}
f & =\pi_+|_{\ell_1 Y'}\circ (\pi_-|_{\ell_1 Y'})^{-1}:  U_{(12)}^{-2}\oplus U_{(15)}^-\to U_{(12)}^{+1}\oplus U_{(15)}^+.
\end{align*}
Since $\pi_-(\ell_1 Y_1)=U_{(12)}^{-2}$, we have
$f(U_{(12)}^{-2})=U_{(12)}^{+1}$.
Write $f(e_{\kappa(i)})=\sum_{j\in I_{(12)}} a_{i,j}e_j$ for $i\in I_{(12)}$. Define $\ell_2=\ell(A,I_{2a_2+1})\in {\rm O}(U_{11})\cap R$ with the matrix
$$A=\bp \{a_{i,j}\} & 0 \\ 0 & I_{a'_1-b_{12}} \ep.$$
Then
$\ell_2^{-1}\ell_1Y_1=V_{(12)}^1=\bigoplus_{i\in I_{(12)}} \bbf(e_i+e_{\kappa(i)})$.
We can write
$f(U_{(15)}^-)=\bigoplus_{i\in I_{(15)}} \bbf v_i$
with vectors
$v_i=e_i+\sum_{j\in I_{(12)}} c_{i,j}e_j$
for $i\in I_{(15)}$. Take an element $\ell_3=\ell(C,I_{2a_2+1})\in {\rm O}(U_{11})\cap R$ with the matrix
$$C=\bp I_{b_{12}} & \{c_{i,j}\} & 0 \\ 0 & I_{b_{15}} & 0 \\ 0 & 0 & I_{b_{13}+b_{12}} \ep.$$
Then we have
$$\ell_3^{-1}\ell_2^{-1}\ell_1Y' \subset V_{(12)}^1\oplus U_{(15)}^0$$
where $U_{(15)}^0=U_{(15)}^+\oplus U_{(15)}^-=U_{(15)}\cap (U_++U_-)$.
Since $\dim(Z\cap \ell_3^{-1}\ell_2^{-1}\ell_1V_{11})=b_{14}$ by definition, we can take an element $\ell_4=\ell(I_{a'_1},B)$ with some $B\in {\rm O}_{2a_2+1}(\bbf)$ such that
$$Z\cap \ell_4\ell_3^{-1}\ell_2^{-1}\ell_1V_{11}=V_{(14)}=\bbf e_{d+a_1+1}\oplus\cdots\oplus \bbf e_{d+a_1+b_{14}}.$$
We have only to consider $V_\#=\ell_4\ell_3^{-1}\ell_2^{-1}\ell_1V_{11}$ and $V'_\#=\ell_4\ell_3^{-1}\ell_2^{-1}\ell_1V'_{11}=(\bigoplus_{j=1}^{11} V_{(j)})\oplus V_\#$ in the following.

By the above arguments we have
$$V_\#\supset V_{(12)}^1\oplus (U_{(15)}^0\cap V_\#)\oplus V_{(14)}\mand V_\#^\perp \supset V_{(12)}^1\oplus (U_{(15)}^0\cap V_\#^\perp)\oplus V_{(14)}$$
with $\dim(U_{(15)}^0\cap V_\#^\perp)=b_{15}$. Let $V_\#^{(15)}$ denote the orthogonal space of $U_{(15)}^0\cap V_\#^\perp$ in $U_{(15)}^0$. Then
$$V_\#^{(15)}=\begin{cases} U_{(15)}^0\cap V_\#^\perp=U_{(15)}^0\cap V_\# & \text{if $\varepsilon=0$}, \\
(U_{(15)}^0\cap V_\#) \oplus \bbf u_0\ & \text{if $\varepsilon=1$} \end{cases}$$
where $u_0$ is an element of $V_\#^{(15)}$ such that $(u_0,u_0)=-1$.

Consider the projections $\pi_{+-}: U_{11}\to U_{11}\cap (U_++U_-)=U_{11}\cap (U_{(+)}\oplus U_{(-)})$ and $\pi_Z: U_{11}\to Z_0$ with respect to the direct sum decomposition
$$U_{11}=(U_{11}\cap (U_++U_-))\oplus Z_0$$
where $Z_0=Z\cap U_{11}=\bbf e_{d+a_1+1}\oplus\cdots\oplus \bbf e_{\overline{d+a_1+1}}$.
Since $\pi_{+-}(V_\#)$ is the orthogonal space of $Y'(V_\#)=U_{11}\cap X'(V'_\#)$ in $U_{11}\cap (U_++U_-)$, we have
$$\pi_{+-}(V_\#) =(\bigoplus_{i\in I_{(12)}\sqcup \kappa(I_{(12)}) \sqcup I_{(13)}\sqcup \overline{I_{(13)}}} \bbf e_i)\oplus (\bigoplus_{i\in I_{(12)}} \bbf (e_{\overline{i}}-e_{\overline{\kappa(i)}}))\oplus V_\#^{(15)}.$$

Put $u_i=e_i$ for $i\in I_{(12)}\sqcup I_{(13)}\sqcup \overline{I_{(13)}}$ and $u_{\overline{i}}=e_{\overline{i}}-e_{\overline{\kappa(i)}}$ for $i\in I_{(12)}$.

First suppose $\varepsilon=0$. Then $\{u_i \mid i\in I_{(12)}\sqcup I_{(13)}\sqcup \overline{I_{(13)}}\sqcup \overline{I_{(12)}}\}$ is a basis
of a complementary subspace of $Y(V_\#)=X(V'_\#)=V_{(12)}^1\oplus V_\#^{(15)}$ in $\pi_{+-}(V_\#)$. The space $V_\#$ uniquely defines vectors $v_i\ (i\in I_{(12)}\sqcup I_{(13)}\sqcup \overline{I_{(13)}}\sqcup \overline{I_{(12)}})$ contained in $\bbf e_{d+a_1+b_{14}+1}\oplus\cdots\oplus \bbf e_{\overline{d+a_1+b_{14}+1}}$ (the orthogonal complement of $U_{(14)}=V_{(14)}\oplus \overline{V}_{(14)}$ in $Z_0$) such that
$u_i+v_i\in V_\#$.
Since
$$(u_i,u_j)=\begin{cases} 1 & \text{if $j=\overline{i}$}, \\
0 & \text{otherwise}, \end{cases}$$
we have
$$(v_i,v_j)=\begin{cases} -1 & \text{if $j=\overline{i}$}, \\
0 & \text{otherwise}. \end{cases}$$
So we can take an element $\ell_5=\ell(I_{a'_1},B)$ with some $B\in {\rm O}_{2a_2+1}(\bbf)$ such that
\begin{align}
\begin{split}
\ell_5 v_i & =e_{\lambda(i)},\ \ell_5 v_{\overline{i}} =-e_{\overline{\lambda(i)}}\quad\mbox{for }i\in I_{(12)} \\
\mbox{and that}\quad \ell_5 v_i & =e_{\eta_{13}(i)},\ \ell_5 v_{\overline{i}} =-e_{\overline{\eta_{13}(i)}}\quad\mbox{for }i\in I_{(13)}.
\label{eq3.4''}
\end{split}
\end{align}
Thus we have
\begin{align*}
\ell_5V_\# & =V_{(12)}^1\oplus V_{(12)}^2\oplus V_{(12)}^3\oplus V_{(13)}\oplus V_{(14)}\oplus V_\#^{(15)} 
 =V_{(12)}\oplus V_{(13)}\oplus V_{(14)}\oplus V_\#^{(15)}.
\end{align*}

Next suppose $\varepsilon=1$. Then we can take a basis $u_i\ (i\in \{0\}\sqcup I_{(12)}\sqcup I_{(13)}\sqcup \overline{I_{(13)}}\sqcup \overline{I_{(12)}})$ of a complementary subspace of $Y(V_\#)=V_{(12)}^1\oplus (U_{(15)}^0\cap V_\#)$ in $\pi_{+-}(V_\#)$. The space $U_{11}\cap V_\#$ uniquely defines vectors $v_i\ (i\in \{0\}\sqcup I_{(12)}\sqcup I_{(13)}\sqcup \overline{I_{(13)}}\sqcup \overline{I_{(12)}})$ contained in $\bbf e_{d+a_1+b_{14}+1}\oplus\cdots\oplus \bbf e_{\overline{d+a_1+b_{14}+1}}$ such that
$u_i+v_i\in V_\#$.
Since
$$(u_i,u_j)=\begin{cases} -1 & \text{if $i=j=0$}, \\
1 & \text{if $j=\overline{i}$}, \\
0 & \text{otherwise}, \end{cases}$$
we have
$$(v_i,v_j)=\begin{cases} 1 & \text{if $i=j=0$}, \\
-1 & \text{if $j=\overline{i}$}, \\
0 & \text{otherwise}. \end{cases}$$
So we can take an element $\ell_5=\ell(I_{a'_1},B)$ with some $B\in {\rm O}_{2a_2+1}(\bbf)$ satisfying (\ref{eq3.4''}) and
$\ell_5 v_0= e_{n+1}$.
Thus we have
$$\ell_5V_\# =V_{(12)}\oplus V_{(13)}\oplus V_{(14)}\oplus (U_{(15)}^0\cap V_\#)\oplus \bbf(u_0+ e_{n+1}).$$

Finally we have only to consider $({\rm O}(U_{(15)})\cap R)$-orbit of
$$U_{(15)}\cap \ell_5V_\# =\begin{cases} U_{(15)}^0\cap V_\# & \text{if $\varepsilon=0$}, \\
(U_{(15)}^0\cap V_\#)\oplus \bbf(u_0+ e_{n+1}) & \text{if $\varepsilon=1$}. \end{cases}$$
But this problem is already solved in \cite{M2}. There exists an $\ell_6\in {\rm O}(U_{(15)})\cap R\cong {\rm GL}_{b_{15}}(\bbf)$ such that
$\ell_6 (U_{(15)}\cap \ell_5V_\#)=V_{(15)}^\varepsilon$.
Thus we have
$$\ell_6\ell_5V_\#=V_{(12)}\oplus V_{(13)}\oplus V_{(14)}\oplus V_{(15)}^\varepsilon$$
and hence $\ell_6\ell_5V'_\#=(\bigoplus_{j=1}^{14} V_{(j)})\oplus V_{(15)}^\varepsilon$,
proving Theorem \ref{th3.10}. \hfill$\square$

\subsection{Construction of elements in $R_V|_{U_+}$}

Assume $V=V(b_1,\ldots,b_{15},\varepsilon)$. We will construct elements in $R_V|_{U_+}$ where $R_V=\{g\in R \mid gV=V\}$.

\begin{lemma} For $j=1,\ldots,14$, let $A=\{a_{i,k}\}|_{i,k\in I_{(j)}}$ be an invertible matrix with the inverse matrix $A^{-1}=\{b_{i,k}\}$. Then we can construct an element $h=h_{(j)}(A)$ of $R_V$ such that $:$

{\rm (i)} If $j\in J_1=\{1,2,3,4,5,6,10,11,14\}$, then
$$he_k=\sum_{i\in I_{(j)}} a_{i,k}e_i,\quad he_{\overline{k}}=\sum_{i\in I_{(j)}} b_{k,i}e_{\overline{i}}$$
for $k\in I_{(j)}$ and $he_\ell=e_\ell$ for $\ell\in I-\widetilde{I_{(j)}}$.

{\rm (ii)} If $j\in J_2=\{7,8,9,13\}$, then
$$he_k=\sum_{i\in I_{(j)}} a_{i,k}e_i,\ he_{\eta_j(k)}=\sum_{i\in I_{(j)}} a_{i,k}e_{\eta_j(i)},\ he_{\overline{k}}=\sum_{i\in I_{(j)}} b_{k,i}e_{\overline{i}},\ he_{\overline{\eta_j(k)}}=\sum_{i\in I_{(j)}} b_{k,i}e_{\overline{\eta_j(i)}}$$
for $k\in I_{(j)}$ and $he_\ell=e_\ell$ for $\ell\in I-\widetilde{I_{(j)}}$.

{\rm (iii)} If $j=12$, then
\begin{align*}
he_k & =\sum_{i\in I_{(12)}} a_{i,k}e_i,\quad he_{\kappa(k)}=\sum_{i\in I_{(12)}} a_{i,k}e_{\kappa(i)},\quad he_{\lambda(k)}=\sum_{i\in I_{(12)}} a_{i,k}e_{\lambda(i)}, \\
he_{\overline{k}} & =\sum_{i\in I_{(12)}} b_{k,i}e_{\overline{i}},\quad he_{\overline{\kappa(k)}}=\sum_{i\in I_{(12)}} b_{k,i}e_{\overline{\kappa(i)}},\quad he_{\overline{\lambda(k)}}=\sum_{i\in I_{(12)}} b_{k,i}e_{\overline{\lambda(i)}}
\end{align*}
for $k\in I_{(12)}$ and $he_\ell=e_\ell$ for $\ell\in I-\widetilde{I_{(12)}}$.
\label{lem3.12}
\end{lemma}

\begin{proof} Clearly $hU_+=U_+,\ hU_-=U_-$ and $hV=V$. So we have only to prove that $h$ preserves the bilinear form $(\ ,\ )$. Namely
$$(he_k,e_\ell)=(e_k,h^{-1}e_\ell)$$
for $k,\ell\in I$.

Suppose $j\in J_1$. Then the equality is nontrivial only when $(k,\ell)\in (\overline{I_{(j)}}\times I_{(j)}) \sqcup (I_{(j)}\times \overline{I_{(j)}})$. If $k,\ell\in I_{(j)}$, then
$$(he_{\overline{k}},e_\ell)=(\sum_{i\in I_{(j)}} b_{k,i}e_{\overline{i}},\ e_\ell)=b_{k,\ell}=(e_{\overline{k}},\sum_{i\in I_{(j)}} b_{i,\ell}e_i)=(e_{\overline{k}},h^{-1}e_\ell)$$
and
$$(he_k,e_{\overline{\ell}})=(\sum_{i\in I_{(j)}} a_{i,k}e_i,e_{\overline{\ell}}) =a_{\ell,k}=(e_k,\sum_{i\in I_{(j)}} a_{\ell,i}e_{\overline{i}})=(e_k,h^{-1}e_{\overline{\ell}}).$$
So the assertion is proved. We can also prove the assertion for  $j\in J_2\sqcup \{12\}$ in the same way. 
\end{proof}

The index set $I_+=\{i\in I\mid e_i\in U_+\}$ is decomposed as
$$I_+=\bigsqcup_{j\in \mathcal{I}_+} I_{(j)}$$
where
$\mathcal{I}_+=\{1,2,3,7,8,10,5,9,12,15,13,\overline{12},\overline{8},\overline{6}\}$
and $I_{(\overline{6})}=\overline{I_{(6)}},\ I_{(\overline{8})}=\overline{\eta_8(I_{(8)})}$ and $I_{(\overline{12})}=\overline{\kappa(I_{(12)})}$. Correspondingly $U_+$ is decomposed as
$$U_+=\bigoplus_{U\in\mathcal{U}} U$$
where 
\begin{align*}
\mathcal{U} &=\{U_{(1)}^+, U_{(2)}^+, U_{(3)}^+, U_{(7)}^+, U_{(8)}^{+1},  U_{(10)}^+, U_{(5)}^+, U_{(9)}^+, 
 U_{(12)}^{+1},  U_{(15)}^+, U_{(13)}^+, U_{(12)}^{+2}, U_{(8)}^{+2}, U_{(6)}^+\} \\
U_{(j)}^+ & =U_{(j)}\cap U_+\ (j=1,\ldots,15),\quad U_{(8)}^{+1}=\bigoplus_{i\in I_{(8)}} \bbf e_i,\quad 
U_{(8)}^{+2}=\bigoplus_{i\in I_{(\overline{8})}} \bbf e_i, \\
U_{(12)}^{+1} & =\bigoplus_{i\in I_{(12)}} \bbf e_i\mand U_{(12)}^{+2}=\bigoplus_{i\in I_{(\overline{12})}} \bbf e_i.
\end{align*}
Consider the following diagram of $\{I_{(j)}\mid j\in\mathcal{I}_+\}$.

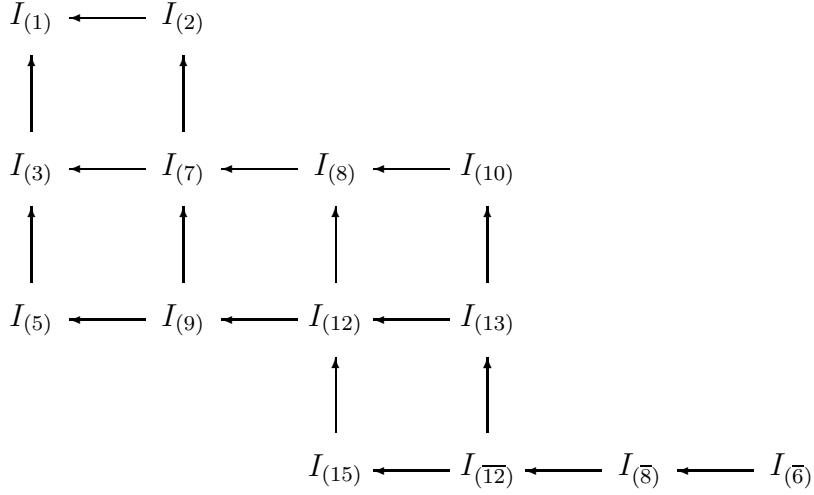
\setlength{\unitlength}{1mm}
\begin{picture}(140,80)(10,20)
\put(30,90){\makebox(0,0){$I_{(1)}$}}
\put(50,90){\makebox(0,0){$I_{(2)}$}}
\put(30,70){\makebox(0,0){$I_{(3)}$}}
\put(50,70){\makebox(0,0){$I_{(7)}$}}
\put(70,70){\makebox(0,0){$I_{(8)}$}}
\put(90,70){\makebox(0,0){$I_{(10)}$}}
\put(30,50){\makebox(0,0){$I_{(5)}$}}
\put(50,50){\makebox(0,0){$I_{(9)}$}}
\put(70,50){\makebox(0,0){$I_{(12)}$}}
\put(90,50){\makebox(0,0){$I_{(13)}$}}
\put(110,30){\makebox(0,0){$I_{(\overline{8})}$}}
\put(70,30){\makebox(0,0){$I_{(15)}$}}
\put(90,30){\makebox(0,0){$I_{(\overline{12})}$}}
\put(130,30){\makebox(0,0){$I_{(\overline{6})}$}}
\put(45,90){\vector(-1,0){10}}
\put(45,70){\vector(-1,0){10}}
\put(65,70){\vector(-1,0){10}}
\put(85,70){\vector(-1,0){10}}
\put(45,50){\vector(-1,0){10}}
\put(65,50){\vector(-1,0){10}}
\put(85,50){\vector(-1,0){10}}
\put(30,75){\vector(0,1){10}}
\put(50,75){\vector(0,1){10}}
\put(30,55){\vector(0,1){10}}
\put(50,55){\vector(0,1){10}}
\put(70,55){\vector(0,1){10}}
\put(90,55){\vector(0,1){10}}
\put(90,35){\vector(0,1){10}}
\put(90,35){\vector(0,1){10}}
\put(70,35){\vector(0,1){10}}
\put(85,30){\vector(-1,0){10}}
\put(105,30){\vector(-1,0){10}}
\put(125,30){\vector(-1,0){10}}
\end{picture}

We define a partial order $I_{(j)}<I_{(j')}$ for $j,j'\in\mathcal{I}_+$ if there exists a sequence $j_0,j_1,\ldots,j_k$ in $\mathcal{I}_+$ such that
$$I_{(j)}=I_{(j_0)}\longleftarrow I_{(j_1)}\longleftarrow \cdots\longleftarrow I_{(j_k)}=I_{(j')}.$$
For example, $I_{(1)}<I_{(j)}$ for all $j\in\mathcal{I}_+-\{1\}$ and $I_{(j)}<I_{(\overline{6})}$ for all $j\in\mathcal{I}_+-\{\overline{6}\}$.

For $i\in I_{(1)}\sqcup\cdots\sqcup I_{(14)}$, define $I_{<i>},W_{<i>}$ and $\overline{W}_{<i>}$ as in Remark \ref{rem3.11}. For $i\in I_{(6)}$, define
$$I_{<\overline{i}>}=\{\overline{i}\},\quad W_{<\overline{i}>}=\bbf e_{\overline{i}}\mand \overline{W}_{<\overline{i}>}=\bbf e_i.$$
For $i\in I_{(8)}$, define
$$I_{<\overline{\eta_8(i)}>}=\{\overline{i},\overline{\eta_8(i)}\},\quad
W_{<\overline{\eta_8(i)}>}=\overline{W}_{<i>} \mand \overline{W}_{<\overline{\eta_8(i)}>}=W_{<i>}.$$
For $i\in I_{(12)}$, define
$$I_{<\overline{\kappa(i)}>}=\{\overline{i},\overline{\kappa(i)},\overline{\lambda(i)}\},\quad
W_{<\overline{\kappa(i)}>}=\overline{W}_{<i>} \mand \overline{W}_{<\overline{\kappa(i)}>}=W_{<i>}.$$

\begin{lemma} Suppose that $I_{(j)}<I_{(j')}$ for $j,j'\in\mathcal{I}_+-\{15\}$.

{\rm (i)} For $i\in I_{(j)}$ and $k\in I_{(j')}$, there exists a linear map $\varphi=\varphi_{i,k}: W_{<k>}\to W_{<i>}$ satisfying the following three properties.

{\rm (P1)} $\varphi (e_k) =e_i,\ \varphi(W_{<k>}\cap U_+^\perp) \subset U_+^\perp,\ \varphi(W_{<k>}\cap U_-)\subset U_-, \ 
 \varphi(W_{<k>}\cap U_-^\perp) \subset U_-^\perp$ and \ $\varphi(W_{<k>}\cap V)\subset V$.

{\rm (P2)} If $i\in I_{(8)}$ and $k\in I_{(12)}$, then $\varphi^*(e_{\overline{\eta_8(i)}})=e_{\overline{\kappa(k)}}$.

{\rm (P3)} If $i\in I_{(8)}\sqcup I_{(12)}$ and $k\in I_{(\overline{8})}\sqcup I_{(\overline{12})}$, then $\varphi^*(\overline{W}_{<i>}\cap U_+)=\{0\}$.

\noindent Here $\varphi^*: \overline{W}_{<i>}\to \overline{W}_{<k>}$ is the adjoint map of $\varphi=\varphi_{i,k}$ with respect to the bilinear form $(\ ,\ )$.

{\rm (ii)} If $W_{<i>}\ne \overline{W}_{<k>}$, then the map $g=g_{i,k}(\mu):\bbf^{2n+1}\to \bbf^{2n+1}\ (\mu\in\bbf)$ defined by
$$ge_\ell=\begin{cases} e_\ell+\mu\varphi(e_\ell) & \text{if $\ell\in I_{<k>}$,} \\
e_\ell-\mu\varphi^*(e_\ell) & \text{if $\ell\in \overline{I_{<i>}}$,} \\
e_\ell & \text{if $\ell\in I-I_{<k>}-\overline{I_{<i>}}$}
\end{cases}$$ 
is an element of $R_V$.

{\rm (iii)} For $i\in I_{(12)}$ and $\mu\in\bbf$, the map $g=g_{i,\overline{\kappa(i)}}(\mu)$ defined by
$$ge_{\overline{\kappa(i)}} =e_{\overline{\kappa(i)}}+\mu e_i,\quad ge_{\overline{i}}=e_{\overline{i}}-\mu e_{\kappa(i)}\mand ge_\ell=e_\ell\ \mbox{for }\ell\in I-\{\overline{i},\overline{\kappa(i)}\}$$
is an element of $R_V$.

{\rm (iv)} For $i\in I_{(8)}$ and $\mu\in\bbf$, the map $g=g_{i,\overline{\eta_8(i)}}(\mu)$ defined by
$$ge_{\overline{\eta_8(i)}} =e_{\overline{\eta_8(i)}}+\mu e_i,\quad ge_{\overline{i}}=e_{\overline{i}}-\mu e_{\eta_8(i)}\mand ge_\ell=e_\ell\ \mbox{for }\ell\in I-\{\overline{i},\overline{\eta_8(i)}\}$$
is an element of $R_V$.
\label{lem3.13'}
\end{lemma}

\begin{proof} (i) (A) First consider the case of $I_{(j)}\longleftarrow I_{(j')}$.

When $\dim W_{<k>}=1$, the pair $(i,k)$ is contained in one of
$$I_{(1)}\times I_{(2)},\quad I_{(1)}\times I_{(3)},\quad I_{(3)}\times I_{(5)},\quad I_{(8)}\times I_{(10)} \quad\mbox{or}\quad I_{(\overline{8})}\times I_{(\overline{6})}.$$
Then Table \ref{table2} shows that the maps
$$\varphi_{i,k}: W_{<k>}\ni e_k\mapsto e_i\in W_{<i>}$$
satisfy the property (P1).

\begin{table}
\setlength{\extrarowheight}{4pt}
\begin{center}
\begin{tabular}{cccc} \hline
$i\in I_{(j)}$ & $e_i\in U_-$ & $e_i\in U_-^\perp$ & $e_i\in V$ \\ \hline
$I_{(1)}$ & Yes & Yes & Yes  \\
$I_{(2)}$ & Yes & Yes & No  \\
$I_{(3)}$ & No & Yes & Yes  \\
$I_{(5)}$ & No & No & Yes  \\
$I_{(8)}$ & No & Yes & No  \\
$I_{(10)}$ & No & Yes & No  \\
$I_{(\overline{8})}$ & No & No & No  \\
$I_{(\overline{6})}$ & No & No & Yes  \\
\hline
\end{tabular}
\end{center}
\setlength{\extrarowheight}{0pt}
\caption{Properties on $e_i$ for $i\in I(j)$}
\label{table2}
\end{table}

Next consider the case where $\dim W_{<k>}=2$ and $\dim W_{<i>}=1$. Then $(i,k)$ is contained in $I_{(j)}\times I_{(j')}$ with
$$(j,j')=(2,7),(3,7),(5,9)\mbox{ or } (10,13).$$
Define $\varphi(e_{\eta_{j'}(k)})$ as follows. Then the maps $\varphi=\varphi_{i,k}$ satisfy (P1).

\setlength{\extrarowheight}{4pt}
\begin{center}
\begin{tabular}{cc} \hline
$I_{(j)}\longleftarrow I_{(j')}$ & $\varphi(e_{\eta_{j'}(k)})$ \\ \hline
$I_{(2)}\longleftarrow I_{(7)}$ & $-e_i$  \\
$I_{(3)}\longleftarrow I_{(7)}$ & $0$  \\
$I_{(5)}\longleftarrow I_{(9)}$ & $0$  \\
$I_{(10)}\longleftarrow I_{(13)}$ & $-e_i$ \\
\hline
\end{tabular}
\end{center}
\setlength{\extrarowheight}{0pt}

\bigskip
Consider the case of $\dim W_{<k>}=\dim W_{<i>}=2$. There are two cases
$$(i,k)\in I_{(7)}\times I_{(8)}\mand (i,k)\in I_{(7)}\times I_{(9)}.$$
In these cases, we put $\varphi(e_{\eta_j(k)})=e_{\eta_7(i)}$ for $j=8$ and $9$. Then the maps $\varphi=\varphi_{i,k}$ satisfy (P1).

The remaining cases are when $i$ or $k$ is in $I_{(12)}\sqcup I_{(\overline{12})}$. The pair $(i,k)$ is contained in one of 
$$I_{(8)}\times I_{(12)},\quad I_{(9)}\times I_{(12)},\quad I_{(12)}\times I_{(13)},\quad I_{(13)}\times I_{(\overline{12})} \quad\mbox{or}\quad I_{(\overline{12})}\times I_{(\overline{8})}.$$

If $(i,k)\in I_{(8)}\times I_{(12)}$, then we put
$$\varphi(e_{\kappa(k)})=e_{\eta_8(i)}\mand \varphi(e_{\lambda(k)})=-e_i.$$
Hence the property (P2) $: \varphi^*(e_{\overline{\eta_8(i)}})=e_{\overline{\kappa(k)}}$ is satisfied.

If $(i,k)\in I_{(9)}\times I_{(12)}$, then we put
$$\varphi(e_{\kappa(k)})=e_{\eta_9(i)}\mand \varphi(e_{\lambda(k)})=e_{\eta_9(i)}.$$

If $(i,k)\in I_{(12)}\times I_{(13)}$, then we put
$$\varphi(e_{\eta_{13}(k)})=e_{\kappa(i)}.$$

For $\overline{\kappa(k)}\in I_{(\overline{12})}$ and $i\in I_{(13)}$, the map $\varphi=\varphi_{i,\overline{\kappa(k)}}: W_{<\overline{\kappa(k)}>}=\overline{W}_{<k>}\to W_{<i>}$ is defined by
$$\varphi(e_{\overline{\lambda(k)}})=e_{\eta_{13}(i)}\mand \varphi(e_{\overline{k}})=0\quad (\varphi(e_{\overline{\kappa(k)}})=e_i\ \mbox{by definition}).$$

For $\overline{\eta_8(k)}\in I_{(\overline{8})}$ and $\overline{\kappa(i)}\in I_{(\overline{12})}$, the map $\varphi=\varphi_{\overline{\kappa(i)},\overline{\eta_{8}(k)}}: W_{<\overline{\eta_{8}(k)}>}=\overline{W}_{<k>}\to W_{<\overline{\kappa(i)}>}=\overline{W}_{<i>}$ is defined by
$$\varphi(e_{\overline{k}})=e_{\overline{i}}-e_{\overline{\lambda(i)}}\quad(\varphi(e_{\overline{\eta_{8}(k)}})=e_{\overline{\kappa(i)}}\ \mbox{by definition}).$$
(Remark: The map $\varphi_{\overline{\kappa(i)},\overline{\eta_{8}(k)}}: W_{<\overline{\eta_8(k)}>}=\overline{W}_{<k>}\to W_{<\overline{\kappa(i)}>}=\overline{W}_{<i>}$ is the adjoint map of $\varphi_{k,i}: W_{<i>}\to W_{<k>}$ for $I_{(12)}\to I_{(8)}$ with respect to the bilinear form $(\ ,\ )$.)

(B) Suppose we have constructed $\varphi_{i,\ell}: W_{<\ell>}\to W_{<i>}$ and $\varphi_{\ell,k}: W_{<k>}\to W_{<\ell>}$ for $(i,\ell,k)\in I_{(j)}\times I_{(j'')}\times I_{(j')}$ with $I_{(j)}< I_{(j'')}< I_{(j')}$. Then the composition
$$\varphi=\varphi_{i,\ell}\circ\varphi_{\ell,k}: W_{<k>}\to W_{<i>}$$
also satisfies (P1). If the numbers $b_1,\ldots, b_{13}$ are all nonzero, then we can construct all $\varphi$ in this way.  If some of $b_j$'s are $0$, then we have only to embed $\bbf^{2n+1}$ into a larger space $\bbf^{2n'+1}$ with nonzero $b_j$'s. 

\bigskip
Remark: This construction of $\varphi=\varphi_{i,\ell}\circ\varphi_{\ell,k}$ depends on the choice of intermediate $j''\in\mathcal{I}_+$. For example, consider $i\in I_{(1)}$ and $k\in I_{(7)}$. If $\ell\in I_{(2)}$, then
$$\varphi_{i,\ell}\circ\varphi_{\ell,k}(e_{\eta_7(k)})=\varphi_{i,\ell}(-e_\ell)=-e_i.$$
But if $\ell\in I_{(3)}$, then
$$\varphi_{i,\ell}\circ\varphi_{\ell,k}(e_{\eta_7(k)})=\varphi_{i,\ell}(0)=0.$$

\bigskip
Suppose $i\in I_{(8)}\sqcup I_{(12)}$ and $k\in I_{(\overline{8})}\sqcup I_{(\overline{12})}$. Then the map $\varphi=\varphi_{i,k}$ is constructed as the composition of $\varphi_{i,\ell}: W_{<\ell>}\to W_{<i>}$ and $\varphi_{\ell,k}: W_{<k>}\to W_{<\ell>}$ with $\ell\in I_{(13)}$. Since $\overline{W}_{<\ell>}\cap U_+=\{0\}$ and since $\varphi_{i,\ell}^*(\overline{W}_{<i>}\cap U_+)\subset \overline{W}_{<\ell>}\cap U_+$ by (P1), we have
$$\varphi^*(\overline{W}_{<i>}\cap U_+)=\varphi_{\ell,k}^*\varphi_{i,\ell}^*(\overline{W}_{<i>}\cap U_+)\subset \varphi_{\ell,k}^*(\overline{W}_{<\ell>}\cap U_+)=\{0\}.$$
So the map $\varphi$ satisfies the property (P3).

(ii) It follows from $\varphi(W_{<k>}\cap U_+^\perp)\subset U_+^\perp,\ \varphi(W_{<k>}\cap U_-^\perp)\subset U_-^\perp$ and $\varphi(W_{<k>}\cap V)\subset V$ that
\begin{equation}
\varphi^*(\overline{W}_{<i>}\cap U_+)\subset U_+,\quad \varphi^*(\overline{W}_{<i>}\cap U_-)\subset U_-\mand \varphi^*(\overline{W}_{<i>}\cap V)\subset V.
\label{eq3.6'}
\end{equation}
Put $W_{<0>}=\bigoplus_{\ell\in I-I_{<k>}-\overline{I}_{<i>}} \bbf e_\ell$. Then we have
\begin{align*}
U_+ & =(U_+\cap W_{<k>})\oplus (U_+\cap \overline{W}_{<i>})\oplus (U_+\cap W_{<0>}), \\
U_- & =(U_-\cap W_{<k>})\oplus (U_-\cap \overline{W}_{<i>})\oplus (U_-\cap W_{<0>}) \\
\mand V & =(V\cap W_{<k>})\oplus (V\cap \overline{W}_{<i>})\oplus (V\cap W_{<0>}).
\end{align*}
Hence we have
$$gU_+=U_+,\quad gU_-=U_-\mand gV=V$$
by (P1) and (\ref{eq3.6'}). So we have only to show that
$$(ge_\ell,e_{\ell'})=(e_\ell,g^{-1}e_{\ell'})$$
for all $\ell,\ell'\in I$. This equality is trivial unless $(\ell,\ell')\in (I_{<k>}\times \overline{I}_{<i>})\sqcup (\overline{I}_{<i>}\times I_{<k>})$. If $(\ell,\ell')\in I_{<k>}\times \overline{I}_{<i>}$, then
\begin{align*}
(ge_\ell,e_{\ell'}) & =(e_\ell+\mu\varphi (e_\ell),e_{\ell'})=\mu(\varphi (e_\ell),e_{\ell'}) \\
& =\mu(e_\ell,\varphi^*(e_{\ell'}))=(e_\ell,e_{\ell'}+\mu\varphi^*(e_{\ell'}))=(e_\ell,g^{-1}e_{\ell'}).
\end{align*}
For $(\ell,\ell')\in \overline{I}_{<i>}\times I_{<k>}$, we also have the equality in the same way.

(iii) and (iv) are clear.
\end{proof}

For $k\in I_{(15)}$, define a vector $f_k\in U_{(15)}^+$ by
$$f_k=\begin{cases} \delta(k)e_{\overline{\eta_{15}(k)}} & \text{if $k\ne c,\overline{\eta_{15}(c)}$}, \\
e_{\overline{\eta_{15}(c)}}+\frac{1}{2}\varepsilon e_c & \text{if $b_{15}$ is even and $k=c$}, \\
-e_{\overline{\eta_{15}(k)}} & \text{if $b_{15}$ is even and $k=\overline{\eta_{15}(c)}$}, \\
\frac{1}{2} e_c & \text{if $b_{15}$ is odd and $k=c$}
\end{cases}$$
where
$$\delta(k)=\begin{cases} 1 & \text{if $k\in I_{(15)}^+$}, \\ -1 & \text{otherwise}. \end{cases}$$
For $k\in I_{(15)}$, define $I_{<k>}$ by
$$I_{<k>}=\begin{cases} \{k,\eta_{15}(k),\overline{k},\overline{\eta_{15}(k)}\} & \text{if $k\ne c,\overline{\eta_{15}(c)}$}, \\
\{c,\eta_{15}(c),\overline{c},\overline{\eta_{15}(c)},n+1\} & \text{if $b_{15}$ is even and $k=c,\overline{\eta_{15}(c)}$,} \\
\{c,\overline{c},n+1\} & \text{if $b_{15}$ is odd and $k=c\ (=\overline{\eta_{15}(c)})$}
\end{cases}$$
and put $W_{<k>}=\bigoplus_{\ell\in I_{<k>}} \bbf e_\ell$ (c.f. Remark \ref{rem3.11}).

\begin{lemma} Let $i\in I_{(j)}$ with $j=1,2,3,5,7,8,9,12$ and $k\in I_{(15)}$.

{\rm (i)} There exists a linear map $\varphi=\varphi_{i,k}: W_{<k>}\to W_{<i>}$ such that
\begin{align}
\begin{split}
\varphi (e_k) & =e_i,\quad \varphi(W_{<k>}\cap U_+^\perp) \subset U_+^\perp, \quad
\varphi(W_{<k>}\cap U_-) \subset U_-, \\
& \varphi(W_{<k>}\cap U_-^\perp)\subset U_-^\perp \quad
\mbox{and that}\quad \varphi(W_{<k>}\cap V^\perp)\subset V^\perp. \label{eq3.7'}
\end{split}
\end{align}
$($Note that $\varphi(W_{<k>}\cap V^\perp)\subset V^\perp$ implies $\varphi(W_{<k>}\cap V^\perp)\subset W_{<i>}\cap V^\perp=W_{<i>}\cap V$ and hence $\varphi(W_{<k>}\cap V)\subset V$.$)$

{\rm (ii)} The map $g=g'_{i,k}(\mu):\bbf^{2n+1}\to \bbf^{2n+1}\ (\mu\in\bbf)$ defined by
$$ge_\ell=\begin{cases} e_\ell+\mu\varphi(e_\ell) & \text{if $\ell\in I_{<k>}$,} \\
e_\ell-\mu\varphi^*(e_\ell)-\frac{\mu^2}{2} \varphi\varphi^*(e_\ell) & \text{if $\ell\in \overline{I_{<i>}}$,} \\
e_\ell & \text{if $\ell\in I-I_{<k>}-\overline{I_{<i>}}$}
\end{cases}$$ 
is an element of $R_V$.

{\rm (iii)} If $i\in I_{(12)}$, then $\displaystyle{g'_{i,k}(\mu)e_{\overline{\kappa(i)}}=e_{\overline{\kappa(i)}}-\mu f_k-\frac{\mu^2}{4} \varepsilon\delta_{k,c}e_i}$.

{\rm (iv)} If $i\in I_{(8)}$, then $\displaystyle{g'_{i,k}(\mu)e_{\overline{\eta_8(i)}}=e_{\overline{\eta_8(i)}}-\mu f_k-\frac{\mu^2}{4} \varepsilon\delta_{k,c}e_i}$.
\label{lem3.14'}
\end{lemma}

\begin{proof} (i) (A) First suppose that $i\in I_{(12)}$. If $k\ne c,\overline{\eta_{15}(c)}$, then $\varphi=\varphi_{i,k}: W_{<k>}\to W_{<i>}$ is defined by
$$\varphi(e_k)=e_i, \quad \varphi (e_{\eta_{15}(k)})=\delta(k)e_{\kappa(i)}\mand \varphi(e_{\overline{k}})=\varphi(e_{\overline{\eta_{15}(k)}})=0.$$

If $b_{15}$ is even and $\varepsilon=0$, then $\varphi=\varphi_{i,k}: W_{<c>}\to W_{<i>}$ for $k=c,\overline{\eta_{15}(c)}$ are defined by
$$\varphi(e_k)=e_i, \quad \varphi (e_{\eta_{15}(k)})=\delta(k)e_{\kappa(i)}\mand \varphi(e_{\overline{k}})=\varphi(e_{\overline{\eta_{15}(k)}})=\varphi(e_{n+1})=0.$$

Suppose $b_{15}$ is even and $\varepsilon=1$. Then $\varphi=\varphi_{i,c}: W_{<c>}\to W_{<i>}$ is defined by
\begin{align*}
\varphi(e_c) & =e_i,\quad \varphi (e_{\eta_{15}(c)})=e_{\kappa(i)}, \quad
\varphi(e_{\overline{c}})=\frac{1}{2} e_{\kappa(i)}, \\
\varphi (e_{\overline{\eta_{15}(c)}}) & =0\mand \varphi(e_{n+1})=-e_{\lambda(i)}.
\end{align*}
(Since $W_{<c>}\cap V^\perp=\bbf (e_{\overline{c}}-e_{\overline{\eta_{15}(c)}}+\frac{1}{2} e_c)\oplus \bbf (-e_c+e_{n+1})$, we can verify the condition $\varphi(W_{<c>}\cap V^\perp)\subset V^\perp$.)
$\varphi=\varphi_{i,\overline{\eta_{15}(c)}}: W_{<c>}\to W_{<i>}$ is defined by
$$\varphi (e_{\overline{\eta_{15}(c)}})=e_i,\quad \varphi(e_{\overline{c}})=-e_{\kappa(i)} \mand \varphi(e_c)=\varphi (e_{\eta_{15}(c)})=\varphi(e_{n+1})=0.$$

If $b_{15}$ is odd, then $\varphi=\varphi_{i,c}: W_{<c>}\to W_{<i>}$ is defined by
$$\varphi(e_c)=e_i,\quad \varphi(e_{\overline{c}})=\frac{1}{2} e_{\kappa(i)} \mand \varphi(e_{n+1})=-e_{\lambda(i)}.$$
(Since $W_{<c>}\cap V^\perp=\bbf (e_{\overline{c}}+\frac{1}{2} e_c)\oplus \bbf (-e_c+e_{n+1})$, we can verify the condition $\varphi(W_{<c>}\cap V^\perp)\subset V^\perp$.)

(B) When $i\in I_{(j)}$ with $j=1,2,3,5,7,8,9$ and $b_{12}\ne 0$, we take the composition $\varphi=\varphi_{i,\ell}\circ \varphi_{\ell,k}$ where $\varphi_{i,\ell}$ is given in Lemma \ref{lem3.13'} and $\varphi_{\ell,k}$ is given in (A) for $\ell\in I_{(12)}$. This construction is also valid for the case of $b_{12}=0$ as is explained in the proof of Lemma \ref{lem3.13'}.

(ii) It follows from (\ref{eq3.7'}) that
$$\varphi^*(\overline{W}_{<i>}\cap U_+)\subset U_+,\quad \varphi^*(\overline{W}_{<i>}\cap U_-)\subset U_-\mand \varphi^*(\overline{W}_{<i>}\cap V)\subset V.$$
So we have
$gU_+=U_+,\ gU_-=U_-$ and $gV=V$
as in the proof of Lemma \ref{lem3.13'}. We have only to show that
$$(ge_\ell,e_{\ell'})=(e_\ell,g^{-1}e_{\ell'})$$
for all $\ell,\ell'\in I$. In view of the proof of Lemma \ref{lem3.13'}, this equality is clear unless $(\ell,\ell')\in I_{(j)}\times I_{(j)}$. If $(\ell,\ell')\in I_{(j)}\times I_{(j)}$, then
$$(ge_\ell,e_{\ell'})=-\frac{\mu^2}{2} (\varphi\varphi^*(e_\ell),e_{\ell'}) =-\frac{\mu^2}{2} (e_\ell,\varphi\varphi^*(e_{\ell'})) =(e_\ell,g^{-1}e_{\ell'})$$
since
$$g^{-1}e_{\ell'}=e_{\ell'}+\mu\varphi^*(e_{\ell'})-\frac{\mu^2}{2} \varphi\varphi^*(e_{\ell'}).$$

(iii) It follows from the construction of $\varphi$ in (i) that
$\varphi^*(e_{\overline{\kappa(i)}})=f_k$.
Hence
\begin{align*}
ge_{\overline{\kappa(i)}} & =e_{\overline{\kappa(i)}}-\mu \varphi^*(e_{\overline{\kappa(i)}})-\frac{\mu^2}{2} \varphi\varphi^*(e_{\overline{\kappa(i)}})=e_{\overline{\kappa(i)}}-\mu f_k-\frac{\mu^2}{2} \varphi(f_k) \\
& =e_{\overline{\kappa(i)}}-\mu f_k-\frac{\mu^2}{4} \varepsilon\delta_{k,c}e_i.
\end{align*}

(iv) The proof is similar to (iii).
\end{proof}

\begin{lemma} {\rm (i)} Suppose that $I_{(j)}<I_{(j')}$ for $j,j'\in\mathcal{I}_+$ and that
$$(j,j') \notin\{(8,12),\ (8,15),\ (12,15),\  (15,\overline{12}),\ (15,\overline{8}),\ (\overline{12},\overline{8})\}.$$
Then for $i\in I_{(j)},\ k\in I_{(j')}$ and $\mu\in\bbf$, there exists an element $g=g_{i,k}(\mu)\in R_V$ such that
$$ge_k=e_k+\mu e_i\quad\mbox{and that}\quad ge_\ell=e_\ell\mbox{ for }\ell\in I_+-\{k\}.$$

{\rm (ii)} For $i\in I_{(8)},\ k\in I_{(12)}$ and $\mu\in\bbf$, there exists an element $g=g_{i,k}(\mu)\in R_V$ such that
$$ge_k=e_k+\mu e_i,\quad ge_{\overline{\eta_8(i)}}=e_{\overline{\eta_8(i)}}-\mu e_{\overline{\kappa(k)}}$$
and that $ge_\ell=e_\ell$ for $\ell\in I_+-\{k,\overline{\eta_8(i)}\}$.

{\rm (iii)} For $i\in I_{(12)},\ k\in I_{(15)}$ and $\mu\in\bbf$, there exists an element $g=g_{i,k}(\mu)\in R_V$ such that
$$ge_k=e_k+\mu e_i,\quad ge_{\overline{\kappa(i)}}= e_{\overline{\kappa(i)}}-\mu f_k$$
and that $ge_\ell=e_\ell$ for $\ell\in I_+-\{k,\overline{\kappa(i)}\}$.

{\rm (iv)} For $i\in I_{(8)},\ k\in I_{(15)}$ and $\mu\in\bbf$, there exists an element $g=g_{i,k}(\mu)\in R_V$ such that
$$ge_k=e_k+\mu e_i,\quad ge_{\overline{\eta_8(i)}}= e_{\overline{\eta_8(i)}}-\mu f_k$$
and that $ge_\ell=e_\ell$ for $\ell\in I_+-\{k,\overline{\eta_8(i)}\}$.

\label{lem3.14}
\end{lemma}

\begin{proof} (i) and (ii). If $j,j'\ne 15$, then $g=g_{i,k}(\mu)$ is constructed in Lemma \ref{lem3.13'}. In particular, if $i\in I_{(8)}$ and $k\in I_{(12)}$, then the property (P2) implies the assertion (ii).

Suppose $j'=15$ and $j\ne 8,12$. Then $g=g_{i,k}(\mu)=g'_{i,k}(\mu)$ is constructed in Lemma \ref{lem3.14'}.

Finally consider the case of $(j,j')=(15,\overline{6})$. For $i\in I_{(15)},\ k\in \overline{I}_{(6)}$ and $\mu\in\bbf$, we can define $g=g_{i,k}(\mu)\in R_V$ by
$$ge_k=e_k+\mu e_i,\quad ge_{\overline{i}}=e_{\overline{i}}-\mu e_{\overline{k}}$$
and $ge_\ell=e_\ell$ for $\ell\in I-\{k,\overline{i}\}$.

(iii) In Lemma \ref{lem3.14'}, we constructed for $i\in I_{(12)},\ k\in I_{(15)}$ and $\mu\in\bbf$, $g'=g'_{i,k}(\mu)\in R_V$ such that
$$g'e_k=e_k+\mu e_i \quad\mbox{and that}\quad  g'e_{\overline{\kappa(i)}}=e_{\overline{\kappa(i)}}-\mu f_k-\frac{\mu^2}{4} \varepsilon\delta_{k,c} e_i.$$
So the element
$$g=g_{i,\overline{\kappa(i)}}\left(\frac{\mu^2}{4} \varepsilon\delta_{k,c}\right)g'$$
satisfies the desired property in view of Lemma \ref{lem3.13'} (iii).

The proof of (iv) is similar to (iii).
\end{proof}

\begin{lemma} Let $k$ be an index in $I_{(\overline{6})}$. Then for any element $u$ in $\bigoplus_{i\in I-I_{(\overline{6})}} \bbf e_i$,
there exists a $g\in R_V$ such that
$g(e_k+u)=e_k$.
\label{lem3.15}
\end{lemma}

\begin{proof} For $u=\sum_{i\in I_+-I_{(\overline{6})}} \mu_ie_i$, put
$g=\prod_{i\in I_+-I_{(\overline{6})}} g_{i,k}(-\mu_i)$.
(The elements $g_{i,k}(-\mu_i)$ are commutative for $i\in I_+-I_{(\overline{6})}$.) Then we have $g(e_k+u)=e_k$ by Lemma \ref{lem3.12} (i).
\end{proof}

\begin{lemma} Let $k$ be an index in $I_{(8)}$. Then for any element $u$ in $\bigoplus_{i\in I_+-I_{(\overline{8})}-I_{(\overline{6})}} \bbf e_i$,
there exists a $g\in R_V$ such that
$g(e_{\overline{\eta_8(k)}}+u)=e_{\overline{\eta_8(k)}}$
and that $g$ acts trivially on $U_{(6)}^+$.
\label{lem3.16}
\end{lemma}

\begin{proof} Write $u=u_1+\sum_{i\in I_{(12)}} \mu_ie_{\overline{\kappa(i)}}$ with an element
$u_1\in \bigoplus_{i\in I_+-I_{(\overline{12})}-I_{(\overline{8})}-I_{(\overline{6})}} \bbf e_i$
and put
$g_1=\prod_{i\in I_{(12)}} g_{k,i}(\mu_i)$.
Then we have
\begin{align*}
g_1(e_{\overline{\eta_8(k)}}+u) & =(e_{\overline{\eta_8(k)}}-\sum_{i\in I_{(12)}} \mu_ie_{\overline{\kappa(i)}})+(g_1(u_1)+\sum_{i\in I_{(12)}} \mu_ie_{\overline{\kappa(i)}}) \\
& =e_{\overline{\eta_8(k)}}+g_1(u_1)\in e_{\overline{\eta_8(k)}}+\bigoplus_{i\in I_+-I_{(\overline{12})}-I_{(\overline{8})}-I_{(\overline{6})}} \bbf e_i
\end{align*}
by Lemma \ref{lem3.14} (ii).

Since $\{f_i\mid i\in I_{(15)}\}$ is a basis of $U_{(15)}^+$, we can write $g_1(u_1)=u_2+\sum_{i\in I_{(15)}} \nu_if_i$ with an element
$u_2\in \bigoplus_{i\in I_+-I_{(15)}-I_{(\overline{12})}-I_{(\overline{8})}-I_{(\overline{6})}} \bbf e_i$.
Put $g_2=\prod_{i\in I_{(15)}} g_{k,i}(-\nu_i)$. Then we can write
$$g_2e_{\overline{\eta_8(k)}}=e_{\overline{\eta_8(k)}}+(\sum_{i\in I_{(15)}} \nu_if_i)+v$$
with some $v\in U_{(8)}^{+1}$ by Lemma \ref{lem3.14} (iv). So we have
\begin{align*}
g_2^{-1}g_1(e_{\overline{\eta_8(k)}}+u) & =
g_2^{-1}(e_{\overline{\eta_8(k)}}+g_1(u_1)) =g_2^{-1}(e_{\overline{\eta_8(k)}}+(\sum_{i\in I_{(15)}} \nu_if_i)+v+u_2-v)\\
& =e_{\overline{\eta_8(k)}}+u_2-v \in e_{\overline{\eta_8(k)}}+\bigoplus_{i\in I_+-I_{(15)}-I_{(\overline{12})}-I_{(\overline{8})}-I_{(\overline{6})}} \bbf e_i.
\end{align*}

Write $u_2-v=\sum_{i\in I'} \rho_ie_i$ where $I'= I_{(1)}\sqcup I_{(2)}\sqcup I_{(3)}\sqcup I_{(7)}\sqcup I_{(8)}\sqcup I_{(10)}\sqcup I_{(5)} \sqcup I_{(9)}\sqcup I_{(12)}\sqcup I_{(13)}$ and put $g_3=\prod_{i\in I'} g_{i,\overline{\eta_8(k)}}(-\rho_i)$. Then we have
$$g_3g_2^{-1}g_1(e_{\overline{\eta_8(k)}}+u)=g_3(e_{\overline{\eta_8(k)}}+u_2-v)=e_{\overline{\eta_8(k)}}$$
by Lemma \ref{lem3.14} (i). It follows from the definitions of $g_1,g_2$ and $g_3$ that $g_3g_2^{-1}g_1$ acts trivially on $U_{(6)}^+$.
\end{proof}

\section{Finiteness of $\mct_{(\alpha_1,\alpha_2),(\beta),(n)}$}

Put $\alpha=\alpha_1+\alpha_2$. As in Section 3, we may fix $\alpha$ and $\beta$-dimensional isotropic subspaces $U_+$ and $U_-$, respectively. By Theorem \ref{th3.10}, we may also fix a maximally isotropic subspace $V=V(b_1,\ldots,b_{15},\varepsilon)$. We have only to show that there are a finite number of $R_V$-orbits on the Grassmann variety of $\alpha_1$-dimensional subspaces of $U_+$.

For any subset $J$ of $I_+$, let $p_J$ denote the canonical projection of $U_+$ onto $\bigoplus_{i\in J} \bbf e_i$ defined by
$$p_J(\sum_{i\in I_+} a_ie_i)=\sum_{i\in J} a_ie_i.$$

\subsection{First reduction}

Let $S$ be an $\alpha_1$-dimensional subspace of $U_+$. Put $s_1=\dim p_{I_{(\overline{6})}}S$. Then there exists an $h_1\in R_V$ such that
$$p_{I_{(\overline{6})}}h_1S=h_1p_{I_{(\overline{6})}}S=U_{(6)}^{+,s_1}=\bigoplus_{i\in I_{(\overline{6},s_1)}} \bbf e_i$$
where $I_{(\overline{6},s_1)}=\{d+a_1-b_6+1,\ldots,d+a_1-b_6+s_1\}$ by Lemma \ref{lem3.12}. Noting that $\bigoplus_{i\in I_+-I_{(\overline{6})}} \bbf e_i=U_+^{\perp (U_-\cap V)}$, we can write
$$h_1S=(h_1S\cap U_+^{\perp (U_-\cap V)})\oplus (\bigoplus_{i\in I_{(\overline{6},s_1)}} \bbf v_i)$$
with vectors $v_i\in e_i+U_+^{\perp (U_-\cap V)}$ for $i\in I_{(\overline{6},s_1)}$. By Lemma \ref{lem3.15}, we can take an element $g_1\in R_V$ such that
$$g_1v_i=e_i\quad\mbox{for all }i\in I_{(\overline{6},s_1)}$$
and that $g_1$ acts trivially on $U_+^{\perp (U_-\cap V)}$. So we have
$$g_1h_1S=S_1\oplus U_{(6)}^{+,s_1}$$
where $S_1=g_1h_1S\cap U_+^{\perp (U_-\cap V)}=S\cap U_+^{\perp (U_-\cap V)}$.

\subsection{Second reduction}

Put $s_2=\dim p_{I_{(\overline{8})}}S_1$. Then there exists an $h_2\in R_V$ such that
$$p_{I_{(\overline{8})}}h_2S_1=h_2p_{I_{(\overline{8})}}S_1=U_{(8)}^{+2,s_2}=\bigoplus_{i\in I_{(8,2)}} \bbf e_{\overline{\eta_8(i)}}$$
where $I_{(8,2)}=\{a_0+b_3+b_7+b_8-s_2+1,\ldots,a_0+b_3+b_7+b_8\}$ by Lemma \ref{lem3.12}. Since $\bigoplus_{i\in I_+-I_{(\overline{8})}-I_{(\overline{6})}} \bbf e_i=U_+^{\perp ((U_-+W_+)\cap V)}$, we can write
$$h_2S_1=(h_2S_1\cap U_+^{\perp ((U_-+W_+)\cap V)})\oplus (\bigoplus_{i\in I_{(8,2)}} \bbf v_i)$$
with vectors $v_i\in e_{\overline{\eta_8(i)}}+U_+^{\perp ((U_-+W_+)\cap V)}$ for $i\in  I_{(8,2)}$. By Lemma \ref{lem3.16}, we can take an element $g_2\in R_V$ such that
$$g_2v_i=e_{\overline{\eta_8(i)}}\quad\mbox{for all }i\in I_{(8,2)}.$$
(Note that $g_2$ does not act trivially on $U_+^{\perp ((U_-+W_+)\cap V)}$ in general.) So we have
$$g_2h_2S_1=S_2\oplus U_{(8)}^{+2,s_2}$$
where $S_2=g_2h_2S_1\cap U_+^{\perp ((U_-+W_+)\cap V)}=g_2S_1\cap U_+^{\perp ((U_-+W_+)\cap V)}$. Since $g_2$ and $h_2$ acts trivially on $U_{(6)}^+$, we have
$$g_2h_2g_1h_1S=S_2\oplus U_{(8)}^{+2,s_2}\oplus U_{(6)}^{+,s_1}.$$

\subsection{Third reduction}

The index set $I_{(8)}$ is decomposed as $I_{(8)}=I_{(8,1)}\sqcup I_{(8,2)}$ where
$$I_{(8,1)} =\{a_0+b_3+b_7+1,\ldots,a_0+b_3+b_7+b_8-s_2\}.$$
Let $U_{(8)}^{+1}=U_{(8,1)}^{+1}\oplus U_{(8,2)}^{+1}$ be the corresponding direct sum decomposition of $U_{(8)}^{+1}$ into
$$U_{(8,1)}^{+1}=\bigoplus_{i\in I_{(8,1)}} \bbf e_i \mand U_{(8,2)}^{+1}=\bigoplus_{i\in I_{(8,2)}} \bbf e_i.$$
Then we have easily the following lemma.

\begin{lemma} Let $h=h_8(A)$ be the element of $R_V$ defined in Lemma \ref{lem3.12} for $A\in {\rm GL}(U_{(8)}^{+1})$. Then
$$hU_{(8)}^{+2,s_2}=U_{(8)}^{+2,s_2}\Longleftrightarrow AU_{(8,1)}^{+1}=U_{(8,1)}^{+1}.$$
\label{lem4.1}
\end{lemma}

Write $p=p_{I_{(8,2)}},\ q=p_{I_{(10)}}$ and $r=p_{I_{(12)}}$. Put $U_0=U_{(1)}^+\oplus U_{(2)}^+\oplus U_{(3)}^+\oplus U_{(7)}^+\oplus U_{(5)}^+\oplus U_{(9)}^+\oplus U_{(8)}^{+1}$.
Put $s_3=\dim p(S_2\cap U_0)$ and $s_4=\dim p(S_2\cap (U_0\oplus U_{(12)}^{+1}))-s_3$. Define a decomposition $I_{(8,2)}=I_{(8,2,1)}\sqcup I_{(8,2,2)}\sqcup I_{(8,2,3)}$ where
\begin{align*}
I_{(8,2,1)} & =\{t_0+1,\ldots,t_0+s_3\},\quad
I_{(8,2,2)} =\{t_0+s_3+1,\ldots,t_0+s_3+s_4\} \\
\mand I_{(8,2,3)} & =\{t_0+s_3+s_4+1,\ldots,t_0+s_2\}
\end{align*}
($t_0=a_0+b_3+b_7+b_8-s_2$). Let $U_{(8,2)}^{+1}=U_{(8,2,1)}^{+1}\oplus U_{(8,2,2)}^{+1}\oplus U_{(8,2,3)}^{+1}$ denote the corresponding direct sum decomposition. Then there exists an $h_3\in R_V$ such that
$$h_3p(S_2\cap U_0)=U_{(8,2,1)}^{+1},\quad h_3p(S_2\cap (U_0\oplus U_{(12)}^{+1}))=U_{(8,2,1)}^{+1}\oplus U_{(8,2,2)}^{+1},\quad h_3p=ph_3$$
and that $h_3e_i=e_i$ for all $i\in I_+-I_{(8,2)}-\overline{\eta_8(I_{(8,2)})}$ by Lemma \ref{lem3.12}.

Put $s_5=\dim q(S_2\cap (U_0\oplus U_{(10)}^+))$ and $s_6=\dim q(S_2\cap (U_0\oplus U_{(10)}^+\oplus U_{(12)}^{+1}))-s_5$. Define a decomposition $I_{(10)}=I_{(10,1)}\sqcup I_{(10,2)}\sqcup I_{(10,3)}$ where
\begin{align*}
I_{(10,1)} & =\{t_0+s_2+1,\ldots,t_0+s_2+s_5\}, \\
I_{(10,2)} & =\{t_0+s_2+s_5+1,\ldots,t_0+s_2+s_5+s_6\} \\
\mand I_{(10,3)} & =\{t_0+s_2+s_5+s_6+1,\ldots,t_0+s_2+b_{10}\}.
\end{align*}
Let $U_{(10)}^+=U_{(10,1)}^+\oplus U_{(10,2)}^+\oplus U_{(10,3)}^+$ denote the corresponding direct sum decomposition. Then there exists an $h_4\in R_V$ such that
$$h_4q(S_2\cap (U_0\oplus U_{(10)}^+))=U_{(10,1)}^+,\quad h_4q(S_2\cap (U_0\oplus U_{(10)}^+\oplus U_{(12)}^{+1}))=U_{(10,1)}^+\oplus U_{(10,2)}^+,$$
$h_4q=qh_4$ and that $h_4e_i=e_i$ for all $i\in I_+-I_{(10)}$ by Lemma \ref{lem3.12}. Put $S'_2=h_4h_3S_2$. Then we have
$$p(S'_2\cap (U_0\oplus U_{(12)}^{+1}))=U_{(8,2,1)}^{+1}\oplus U_{(8,2,2)}^{+1}$$
$$\mand q(S'_2\cap (U_0\oplus U_{(10)}^+\oplus U_{(12)}^{+1}))=U_{(10,1)}^+\oplus U_{(10,2)}^+.$$
We can write
$$S'_2\cap (U_0\oplus U_{(10)}^+\oplus U_{(12)}^{+1})=(S'_2\cap (U_0\oplus U_{(12)}^{+1}))\oplus \bigoplus_{i\in I_{(10,1)}\sqcup I_{(10,2)}} \bbf v_i$$
with vectors $v_i\in e_i+U_0+U_{(12)}^{+1}$. Define a subspace
$$U'_0=U_{(1)}^+\oplus U_{(2)}^+\oplus U_{(3)}^+\oplus U_{(7)}^+\oplus U_{(5)}^+\oplus U_{(9)}^+\oplus U_{(8,1)}^{+1}$$
of $U_0$. Then $U_0=U'_0\oplus U_{(8,2)}^{+1}$. By Lemma \ref{lem3.14} (i), we can take an element $g_3\in R_V$ such that
$$g_3v_i\in e_i+U'_0+U_{(12)}^{+1}$$
for $i\in I_{(10,1)}\sqcup I_{(10,2)}$ and that $g_3e_i=e_i$ for $i\in I_+-(I_{(10,1)}\sqcup I_{(10,2)})$. Put $S''_2=g_3S'_2$. Then
$$p_{I_{(8,2)}\sqcup I_{(10)}} (S''_2\cap (U_0\oplus U_{(10)}^+\oplus U_{(12)}^{+1}))=U_{(8,2,1)}^{+1}\oplus U_{(8,2,2)}^{+1}\oplus U_{(10,1)}^+\oplus U_{(10,2)}^+.$$

Put $S_3=S''_2\cap (U_0\oplus U_{(10)}^+)$. Since $p_{I_{(8,2)}\sqcup I_{(10)}} S_3=U_{(8,2,1)}^{+1}\oplus U_{(10,1)}^+$, we can take a complementary subspace $S_4$ of $S_3$ in $S''_2\cap (U_0\oplus U_{(10)}^+\oplus U_{(12)}^{+1})$ such that
\begin{equation}
p_{I_{(8,2)}\sqcup I_{(10)}} S_4=U_{(8,2,2)}^{+1}\oplus U_{(10,2)}^+.
\label{eq4.1}
\end{equation}
In the same way, we can take a complementary subspace $S_5$ of $S_3\oplus S_4=S''_2\cap (U_0\oplus U_{(10)}^+\oplus U_{(12)}^{+1})$ in $S''_2\cap (U_0\oplus U_{(10)}^+\oplus U_{(12)}^{+1}\oplus U_{(15)}^+)$ and a complementary subspace $S_6$ of $S_3\oplus S_4\oplus S_5=S''_2\cap (U_0\oplus U_{(10)}^+\oplus U_{(12)}^{+1}\oplus U_{(15)}^+)$ in $S''_2$ such that
$$p_{I_{(8,2)}\sqcup I_{(10)}} (S_5\oplus S_6)\subset U_{(8,2,3)}^{+1}\oplus U_{(10,3)}^+.$$
Hence $S''_2$ is decomposed as
$$S''_2=S_3\oplus S_4\oplus S_5\oplus S_6.$$

\subsection{Finiteness of orbits for $S_3$-part}

Put $\gamma=\dim S_3$. We can show that there are a finite number of $R_V$-orbits on the Grassmann variety consisting of $\gamma$-dimensional subspaces of $U_{(S_3)}=U'_0\oplus U_{(8,2,1)}^{+1}\oplus U_{(10,1)}^+$ as follows. The space $U_{(S_3)}$ can be decomposed as
$$U_{(S_3)}=U_{(S_3,1)}\oplus U_{(S_3,2)}$$
where
$$U_{(S_3,1)}=U_{(1)}^+\oplus U_{(2)}^+\oplus U_{(7)}^+\oplus U_{(8,1)}^{+1}\oplus U_{(8,2,1)}^{+1}\oplus U_{(10,1)}^+\mand U_{(S_3,2)}=U_{(3)}^+\oplus U_{(5)}^+\oplus U_{(9)}^+.$$
By Lemma \ref{lem3.12} and Lemma \ref{lem3.14} (i), $R_V|_{U_+}$ contains parabolic subgroups of ${\rm GL}(U_{(S_3,1)})$ and ${\rm GL}(U_{(S_3,2)})$ stabilizing the flags
\begin{align*}
U_{(1)}^+ & \subset U_{(1)}^+\oplus U_{(2)}^+\subset U_{(1)}^+\oplus U_{(2)}^+\oplus U_{(7)}^+\subset U_{(1)}^+\oplus U_{(2)}^+\oplus U_{(7)}^+\oplus U_{(8,1)}^{+1} \\
& \qquad \subset U_{(1)}^+\oplus U_{(2)}^+\oplus U_{(7)}^+\oplus U_{(8,1)}^{+1}\oplus U_{(8,2,1)}^{+1} \\
\mand U_{(3)}^+ & \subset U_{(3)}^+\oplus U_{(5)}^+,
\end{align*}
respectively. (The first parabolic subgroup stabilizes $U_{(8)}^{+2,s_2}$ by Lemma \ref{lem4.1}.) So it contains the product $B_1\times B_2$ of Borel subgroups. By Proposition \ref{lem8.1} in the appendix, there are a finite number of $B_1\times B_2$-orbits on the Grassmann variety. (Remark: We may also directly get explicit but complicated representatives.)

\subsection{A standard form of $S_4$}

The projection $r=p_{I_{(12)}}$ is injective on $S_4$. So we have a bijection
$$r: \widetilde{S_4}\stackrel{\sim}{\to} r(S_4)$$
where $\widetilde{S_4}=p_{I_{(8,2)}\sqcup I_{(10)}\sqcup I_{(12)}} (S_4)$. By (\ref{eq4.1}), we have a surjection
$$f=p_{I_{(8,2)}\sqcup I_{(10)}} \circ (r|_{\widetilde{S_4}})^{-1}: r(S_4)\to U_{(8,2,2)}^{+1}\oplus U_{(10,2)}^+$$
with the kernel $r(S_4\cap (U'_0\oplus U_{(12)}^{+1}))$. Put $s_7=\dim r(S_4\cap (U'_0\oplus U_{(12)}^{+1}))$. Then $\dim r(S_4)=s_4+s_6+s_7$. We can take a basis $v_1,\ldots, v_{s_4+s_6+s_7}$ of $r(S_4)$ such that
\begin{align*}
r(S_4\cap (U'_0\oplus U_{(12)}^{+1})) & =\bbf v_1\oplus\cdots\oplus \bbf v_{s_7}, \\
f(v_{s_7+i}) & =e_{t_0+s_3+i}\mbox{ for }i=1,\ldots,s_4 \\
\mbox{and that}\quad f(v_{s_7+s_4+i}) & =e_{t_0+s_2+s_5+i}\mbox{ for }i=1,\ldots,s_6.
\end{align*}
Take an element $A\in {\rm GL}(U_{(12)}^{+1})$ such that
$$Av_i=e_{d+b_5+b_9+i}\quad\mbox{for }i=1,\ldots,s_4+s_6+s_7$$
and put $h_5=h_{(12)}(A)\in R_V$. Then
$$p_{I_{(8,2)}\sqcup I_{(10)}\sqcup I_{(12)}} (h_5S_4)=h_5\widetilde{S_4}=U(s_7,s_4,s_6)$$
where
\begin{align*}
U(s_7,s_4,s_6) & =(\bigoplus_{i=1}^{s_7} \bbf e_{d+b_5+b_9+i})\oplus (\bigoplus_{i=1}^{s_4} \bbf (e_{d+b_5+b_9+s_7+i}+e_{t_0+s_3+i})) \\
& \qquad \oplus (\bigoplus_{i=1}^{s_6} \bbf (e_{d+b_5+b_9+s_7+s_4+i}+e_{t_0+s_2+s_5+i})).
\end{align*}
Define a decomposition
$I_{(12)}=I_{(12,1)}\sqcup I_{(12,2)}\sqcup I_{(12,3)}\sqcup I_{(12,4)}$
of $I_{(12)}$ into subspaces
\begin{align*}
I_{(12,1)} & =\{d+b_5+b_9+1,\ldots,d+b_5+b_9+s_7\}, \\
I_{(12,2)} & =\{d+b_5+b_9+s_7+1,\ldots,d+b_5+b_9+s_7+s_4\}, \\
I_{(12,3)} & =\{d+b_5+b_9+s_7+s_4+1,\ldots,d+b_5+b_9+s_7+s_4+s_6\} \\
\mand I_{(12,4)} & =\{d+b_5+b_9+s_7+s_4+s_6+1,\ldots,d+b_5+b_9+b_{12}\}.
\end{align*}
Define bijections $\sigma: I_{(12,2)}\to I_{(8,2,2)}$ and $\tau: I_{(12,3)}\to I_{(10,2)}$ by
$$\sigma (d+b_5+b_9+s_7+i)=t_0+s_3+i\mand \tau (d+b_5+b_9+s_7+s_4+i)=t_0+s_2+s_5+i,$$
respectively. Then we can write
$$U(s_7,s_4,s_6)=(\bigoplus_{i\in I_{(12,1)}} \bbf e_i)\oplus (\bigoplus_{i\in I_{(12,2)}} \bbf(e_i+e_{\sigma(i)}))\oplus (\bigoplus_{i\in I_{(12,3)}} \bbf(e_i+e_{\tau(i)})).$$

\begin{lemma} For $i\in I_{(12)}$ and $u\in U'_0$, there exists a $g\in R_V$ such that $g(e_i+u)=e_i$ and that $g$ acts trivially on $U_{(S_3)}\oplus U_{(8)}^{+2,s_2}\oplus U_{(6)}^{+,s_1}$.
\label{lem4.2}
\end{lemma}

\begin{proof} Write $u=u_1+u_2$ with $u_1\in U_{(1)}^+\oplus U_{(2)}^+\oplus U_{(3)}^+\oplus U_{(7)}^+\oplus U_{(5)}^+\oplus U_{(9)}^+$ and $u_2\in U_{(8,1)}^{+1}$. By Lemma \ref{lem3.14} (i), we can take an element $g_1\in R_V$ such that $g_1(e_i+u_1)=e_i$ and that $g_1(e_k)=e_k$ for $k\in I_+-\{i\}$.

Write $u_2=\sum_{k\in I_{(8,1)}} \mu_k e_k$ and put $g_2=\prod_{k\in I_{(8,1)}} g_{i,k}(-\mu_k)$. Then by Lemma \ref{lem3.14} (ii), we have $g_2(e_i+u_2)=e_i$ and $g_2$ acts trivially on $U_{(8)}^{+2,s_2}$ because $\overline{\eta_8(i)}\notin \overline{\eta_8(I_{(8,2)})}$. The element $g=g_2g_1$ satisfies the desired property.
\end{proof}

We can write
\begin{align*}
h_5S_4 & =(\bigoplus_{i\in I_{(12,1)}} \bbf (e_i+u_i))\oplus (\bigoplus_{i\in I_{(12,2)}} \bbf (e_i+e_{\sigma(i)}+u_i)) 
 \oplus (\bigoplus_{i\in I_{(12,3)}} \bbf (e_i+e_{\tau(i)}+u_i)).
\end{align*}
with some $u_i\in U'_0$ for $i\in I_{(12,1)}\sqcup I_{(12,2)}\sqcup I_{(12,3)}$. By Lemma \ref{lem4.2}, there exist $g^{(i)}\in R_V$ for $i\in I_{(12,1)}\sqcup I_{(12,2)}\sqcup I_{(12,3)}$ such that
$$g^{(i)}(e_i+u_i)=e_i$$
and that $g^{(i)}$ acts trivially on $U_{(S_3)}\oplus U_{(8)}^{+2,s_2}\oplus U_{(6)}^{+,s_1}$. Put $g_4=\prod_{i\in I_{(12,1)}\sqcup I_{(12,2)}\sqcup I_{(12,3)}} g^{(i)}$. Then
$$g_4h_5S_4=U(s_7,s_4,s_6).$$

\subsection{Normalization of $g_4h_5(S_5\oplus S_6)$}

Since $g_4h_5S''_2=S_3\oplus g_4h_5S_4\oplus g_4h_5(S_5\oplus S_6)=S_3\oplus U(s_7,s_4,s_6)\oplus g_4h_5(S_5\oplus S_6)$, we have only to consider the orbit of $g_4h_5(S_5\oplus S_6)$ by the action of the subgroup
\begin{align*}
R'_V & =\{g\in R_V\mid g(S_3\oplus U(s_7,s_4,s_6)\oplus U_{(8)}^{+2,s_2}\oplus U_{(6)}^{+,s_1}) \\
& \qquad =S_3\oplus U(s_7,s_4,s_6)\oplus U_{(8)}^{+2,s_2}\oplus U_{(6)}^{+,s_1}\}
\end{align*}
of $R_V$.

Since $p_{I_{(15)}}$ is injective on $g_4h_5S_5$, there exit linear maps
$$f_1: p_{I_{(15)}} (g_4h_5S_5)\to U'_0\oplus U_{(12)}^{+1} \mand f_2: p_{I_{(15)}} (g_4h_5S_5)\to U_{(8,2,3)}^{+1}\oplus U_{(10,3)}^+$$
such that
$$g_4h_5S_5= \{u+f_1(u)+f_2(u)\mid u\in p_{I_{(15)}} (g_4h_5S_5)\}.$$
Extend $f_1$ and $f_2$ linearly on $U_{(15)}^+$. Then by Lemma \ref{lem3.14} (i), (iii) and (iv), there exists an element $g_5\in R_V$ such that
$$g_5(u+f_1(u))=u\quad\mbox{for }u\in U_{(15)}^+.$$
The element $g_5$ is a product of elements of the form $g_{i,k}(\mu)$ with $k\in I_{(15)}, i\in I_{(1)}\sqcup I_{(2)}\sqcup I_{(3)}\sqcup I_{(7)}\sqcup I_{(5)}\sqcup I_{(9)}\sqcup I_{(8,1)}\sqcup I_{(12)}$ and $\mu\in\bbf$. When $i\in I_{(8,1)}$, $e_{\overline{\eta_8(i)}}$ is not contained in $U_{(8)}^{+2,s_2}$ and hence $g_{i,k}(\mu)$ acts trivially on $U_{(8)}^{+2,s_2}$. Similarly we see that each $g_{i,k}(\mu)$ acts trivially on $S_3\oplus U(s_7,s_4,s_6)\oplus U_{(8)}^{+2,s_2}\oplus U_{(6)}^{+,s_1}$. Hence $g_5\in R'_V$. (Note that $g_{i,k}(\mu)$ for $k\in I_{(12)}$ does not stabilize $g_4h_5S_6$ in general.) Thus we have $g_5\in R'_V$ such that
$$g_5g_4h_5S_5\subset U_{(8,2,3)}^{+1}\oplus U_{(10,3)}^+\oplus U_{(15)}^+.$$

Since $p_{I_{(13)}\sqcup I_{(\overline{12})}}$ is injective on $g_5g_4h_5S_6$, there exit linear maps
\begin{align*}
f_3 & : p_{I_{(13)}\sqcup I_{(\overline{12})}} (g_5g_4h_5S_5)\to U'_0\oplus U_{(12)}^{+1} \\
\mand f_4 & : p_{I_{(13)}\sqcup I_{(\overline{12})}} (g_5g_4h_5S_5)\to U_{(8,2,3)}^{+1}\oplus U_{(10,3)}^+\oplus U_{(15)}^+
\end{align*}
such that
$$g_5g_4h_5S_5= \{u+f_3(u)+f_4(u)\mid u\in p_{I_{(13)}\sqcup I_{(\overline{12})}} (g_5g_4h_5S_5)\}.$$
Extend $f_3$ and $f_4$ linearly on $U_{(13)}^+\oplus U_{(12)}^{+2}$. Then by Lemma \ref{lem3.14} (i), there exists an element $g_6\in R_V$ such that
$$g_6(u+f_3(u))=u\quad\mbox{for }u\in U_{(13)}^+\oplus U_{(12)}^{+2}.$$
Since $g_6e_i=e_i$ for $i\in I_+-I_{(13)}-I_{(\overline{12})}$, we have
$g_6\in R'_V$ and $g_6$ acts trivially on $g_5g_4h_5S_5$. Thus we have
$$g_6g_5g_4h_5S_6\subset U_{(8,2,3)}^{+1}\oplus U_{(10,3)}^+\oplus U_{(13)}^+\oplus U_{(12)}^{+2}\oplus U_{(15)}^+$$
and
$$g_6g_5g_4h_5(S_5\oplus S_6)\subset U_{(8,2,3)}^{+1}\oplus U_{(10,3)}^+\oplus U_{(13)}^+\oplus U_{(12)}^{+2}\oplus U_{(15)}^+.$$

\subsection{Construction of a subgroup of $R'_V$}

Put $U_{(12,j)}^{+1}=\bigoplus_{i\in I_{(12,j)}} \bbf e_i$ for $j=1,2,3,4$.
Define linear isomorphisms $\sigma: U_{(12,2)}^{+1}\stackrel{\sim}{\to} U_{(8,2,2)}^{+1}$ and $\tau: U_{(12,3)}^{+1}\stackrel{\sim}{\to} U_{(10,2)}^+$ by
$$\sigma (e_i)=e_{\sigma(i)}\mand \tau (e_i)=e_{\tau(i)},$$
respectively. For $A_j\in {\rm GL}(U_{(12,j)}^{+1})$, define $\widetilde{A_j}\in {\rm GL}(U_{(12)}^{+1})$ by
$$\widetilde{A_j}v=\begin{cases} A_jv & \text{for $v\in U_{(12,j)}^{+1}$}, \\
v & \text{for $v\in U_{(12,k)}^{+1}$ with $k\ne j$.}
\end{cases}$$
We also define $\widetilde{B}\in {\rm GL}(U_{(8)}^{+1})$ for $B\in {\rm GL}(U_{(8,2,2)}^{+1})$ and $\widetilde{C}\in {\rm GL}(U_{(10)}^+)$ for $C\in {\rm GL}(U_{(10,2)}^+)$ in the same way. Then the following lemma is clear from the definition of $U(s_7,s_4,s_6)$.

\begin{lemma} {\rm (i)} For $A_1\in {\rm GL}(U_{(12,1)}^{+1})$ and $A_4\in {\rm GL}(U_{(12,4)}^{+1})$, $\widetilde{h}(A_1)=h_{(12)}(\widetilde{A_1})$ and $\widetilde{h}(A_4)=h_{(12)}(\widetilde{A_4})$ are elements of $R'_V$.

{\rm (ii)} For $A_2\in {\rm GL}(U_{(12,2)}^{+1})$, $\widetilde{h}(A_2)=h_{(8)}(\widetilde{\sigma A_2\sigma^{-1}})h_{(12)}(\widetilde{A_2})$ is an element of $R'_V$.

{\rm (iii)} For $A_3\in {\rm GL}(U_{(12,3)}^{+1})$, $\widetilde{h}(A_3)=h_{(10)}(\widetilde{\tau A_3\tau^{-1}})h_{(12)}(\widetilde{A_3})$ is an element of $R'_V$.
\label{lem4.3}
\end{lemma}

For $(i,k)\in I_{(12,j)}\times I_{(12,j')}$ with $j<j'$ and $\mu\in\bbf$, define
$$\widetilde{g}_{i,k} (\mu)=\begin{cases} h_{(12)}(D) & \text{if $(j,j')\ne (2,3)$,} \\
g_{\sigma(i),\tau(k)}(\mu)h_{(12)}(D) & \text{if $(j,j')=(2,3)$.}
\end{cases}$$
where $D\in {\rm GL}(U_{(12)}^{+1})$ is defined by $De_k=e_k+\mu e_i$ and $De_\ell=e_\ell$ for $\ell\in I_{(12)}-\{k\}$. Then the following lemma is also clear from the definition of $U(s_7,s_4,s_6)$ and $h_{(12)}(D)$.

\begin{lemma} {\rm (i)} $\widetilde{g}_{i,k} (\mu)\in R'_V$.

{\rm (ii)} $\widetilde{g}_{i,k} (\mu)e_{\overline{\kappa(i)}}=e_{\overline{\kappa(i)}}-\mu e_{\overline{\kappa(k)}}$ and $\widetilde{g}_{i,k} (\mu)e_\ell=e_\ell$ for $\ell\in I_{(\overline{12})}-\{\overline{\kappa(i)}\}$.
\label{lem4.4}
\end{lemma}

Put $U_\#=U_{(8,2,3)}^{+1}\oplus U_{(10,3)}^+\oplus U_{(13)}^+\oplus U_{(12)}^{+2}$. The space $U_{(12)}^{+2}$ is decomposed as
$$U_{(12)}^{+2}=U_{(12,4)}^{+2}\oplus U_{(12,3)}^{+2}\oplus U_{(12,2)}^{+2}\oplus U_{(12,1)}^{+2}$$
where $U_{(12,j)}^{+2}=\bigoplus_{i\in I_{(12,j)}} \bbf e_{\overline{\kappa(i)}}$ for $j=1,2,3,4$.
Let $P$ denote the parabolic subgroup of ${\rm GL}(U_\#)$ stabilizing the flag
\begin{align*}
U_{(8,2,3)}^{+1} & \subset U_{(8,2,3)}^{+1}\oplus U_{(10,3)}^+\subset U_{(8,2,3)}^{+1}\oplus U_{(10,3)}^+\oplus U_{(13)}^+\subset U_{(8,2,3)}^{+1}\oplus U_{(10,3)}^+\oplus U_{(13)}^+\oplus U_{(12,4)}^{+2} \\
& \subset U_{(8,2,3)}^{+1}\oplus U_{(10,3)}^+\oplus U_{(13)}^+\oplus U_{(12,4)}^{+2}\oplus U_{(12,3)}^{+2} \\
& \subset U_{(8,2,3)}^{+1}\oplus U_{(10,3)}^+\oplus U_{(13)}^+\oplus U_{(12,4)}^{+2}\oplus U_{(12,3)}^{+2}\oplus U_{(12,2)}^{+2}.
\end{align*}

\begin{lemma} For any element $p\in P$, there exists a $g\in R'_V$ such that $g|_{U_\#}=p$ and that $g$ acts trivially on $U_{(15)}^+$.
\label{lem4.6}
\end{lemma}

\begin{proof} We have only to construct $g\in R'_V$ for the elements of a set of generators in $P$.

First consider the standard Levi subgroup
$${\rm GL}(U_{(8,2,3)}^{+1})\times {\rm GL}(U_{(10,3)}^+)\times {\rm GL}(U_{(13)}^+)\times {\rm GL}(U_{(12,4)}^{+2})\times {\rm GL}(U_{(12,3)}^{+2})\times {\rm GL}(U_{(12,2)}^{+2})\times {\rm GL}(U_{(12,1)}^{+2})$$
of $P$. For $A\in {\rm GL}(U_{(8,2,3)}^{+1})$, $h_{(8)}(\widetilde{A})$ is a desired element of $R'_V$ where $\widetilde{A}\in {\rm GL}(U_{(8)}^{+1})$ is the canonical extension of $A$. Similarly for $B\in {\rm GL}(U_{(10,3)}^+)$, $h_{(10)}(\widetilde{B})$ is a desired element of $R'_V$. For $C\in {\rm GL}(U_{(13)}^+)$, $h_{(13)}(C)$ is a desired element of $R'_V$.

For $A_j\in {\rm GL}(U_{(12,j)}^{+2})$ with $j=1,2,3,4$, we can take $B_j\in {\rm GL}(U_{(12,j)}^{+1})$ such that $h_{(12)}(\widetilde{B_j})|_{U_{(12,j)}^{+2}}=A_j$. Then $\widetilde{h}(B_j)$ are desired elements of $R'_V$ by Lemma \ref{lem4.3}.

So we have only to consider the generators of the unipotent radical of $P$. For
\begin{align*}
(i,k) & \in (I_{(8,3)}\times I_{(10,3)})\sqcup (I_{(8,3)}\times I_{(13)})\sqcup (I_{(10,3)}\times I_{(13)}) \\
& \quad \sqcup (I_{(8,3)}\times I_{(\overline{12})})\sqcup (I_{(10,3)}\times I_{(\overline{12})})\sqcup (I_{(13)}\times I_{(\overline{12})}),
\end{align*}
we have $g=g_{i,k}(\mu)\in R'_V\ (\mu\in\bbf)$ such that $ge_k=e_k+\mu e_i$ and that $ge_\ell=e_\ell$ for $\ell\in I_+-\{k\}$ by Lemma \ref{lem3.14} (i). For $(i,k)\in I_{(12,j)}\times I_{(12,j')}$ with $1\le j<j'\le 4$ and $\mu\in\bbf$, we have $\widetilde{g}_{i,k} (-\mu)\in R'_V$ such that
$$\widetilde{g}_{i,k} (-\mu)e_{\overline{\kappa(i)}}=e_{\overline{\kappa(i)}}+\mu e_{\overline{\kappa(k)}}\quad\mbox{and that}\quad \widetilde{g}_{i,k} (-\mu)e_\ell=e_\ell\mbox{ for }\ell\in I_{(\overline{12})}-\{\overline{\kappa(i)}\}$$
by Lemma \ref{lem4.4}. $\widetilde{g}_{i,k} (-\mu)$ acts trivially on $U_{(15)}^+$ by its construction. Thus we have constructed desired elements of $R'_V$ for all the generators of the unipotent radical of $P$.
\end{proof}

\subsection{Finiteness of $S_5\oplus S_6$-part}

Put $H_{(15)}=({\rm SO}(U_{(15)})\cap R_V)|_{U_{(15)}^+}$. Then we showed in \cite{M2} that
$$H_{(15)}\cong \begin{cases} {\rm Sp}_{b_{15}}(\bbf) & \text{if $b_{15}$ is even and $\varepsilon=0$}, \\
1\times {\rm Sp}_{b_{15}-1}(\bbf) & \text{if $b_{15}$ is odd}, \\
Q_{b_{15}} & \text{if $b_{15}$ is even and $\varepsilon=1$}
\end{cases}$$
where $Q_{b_{15}}=\{g\in {\rm Sp}_{b_{15}}(\bbf)\mid gv=v\}$ with some $0\ne v\in U_{(15)}^+$. We showed in \cite{M2} that there are a finite number of $H_{(15)}$-orbits on the full flag variety of ${\rm GL}(U_{(15)}^+)$.

It is clear that $H_{(15)}\subset R'_V$. Hence the restriction of $R'_V$ to $U_\#\oplus U_{(15)}^+$ contains the group $P\times H_{(15)}$ by Lemma \ref{lem4.6}. So it contains a group of the form $B\times H_{(15)}$ where $B$ is a Borel subgroup of ${\rm GL}(U_\#)$ contained in $P$. Applying Proposition \ref{lem8.1}, we have a finite number of $R'_V$-orbits on the Grassmann variety of  $U_\#\oplus U_{(15)}^+$.

Thus we have proved that the triple flag variety $\mct_{(\alpha_1,\alpha_2),(\beta),(n)}$ is of finite type.

\section{Finiteness of $\mct_{(\alpha),(1),(1^n)}$}

Let $U_+$ and $U_-$ be $\alpha$-dimensional and one-dimensional isotropic subspaces of $\bbf^{2n+1}$, respectively. Put $R=P_{U_+}\cap P_{U_-}$. Then we have only to consider $R$-orbit decomposition of the full flag variety $M_0$ of $G$.

Since $a_0+a_-+a_1=\beta=1$, there are three cases:
$${\rm (i)} \ a_0=1,\qquad {\rm (ii)} \ a_-=1\mand {\rm (iii)} \ a_1=1.$$
If $a_0=1$, then $U_-\subset U_+$. So the decomposition is reduced to the Bruhat decomposition and hence the number of orbits is finite. Thus we have only to consider the remaining two cases.

\subsection{Case of $a_-=1$}

As in Section 3, we put
$$U_+=\bbf e_1\oplus\cdots\oplus \bbf e_\alpha\mand U_-=\bbf e_{\alpha+1}.$$
Let $V_1\subset\cdots\subset V_n$ be a full flag in $\bbf^{2n+1}$. Then by Theorem \ref{th3.10}, we may assume $V_n=V=V(b_1,\ldots,b_{15},\varepsilon)$. For each $V$, we have only to show that there are a finite number of $R_V$-orbits on the full flag variety $M_0(V)$ of ${\rm GL}(V)$.

Since $a_0=a_1=0$ and since
$$a_0=b_1+b_2,\quad a_1=b_5+b_9+2b_{12}+b_{15}+b_{13}+b_8+b_6,$$
we have $b_i=0$ for $i=1,2,5,6,8,9,12,13$ and $15$. Since
$$b_4+b_7+b_9+b_{11}=a_-=1,$$
there are three cases:
$${\rm (a)} \ b_4=1,\qquad {\rm (b)} \ b_7=1\mand {\rm (c)} \ b_{11}=1.$$

\subsubsection{Case of $b_4=1$}

$V$ is written as $V=V_{(3)}\oplus V_{(4)}\oplus V_{(14)}\oplus V_{(10)}$ with
\begin{align*}
V_{(3)} & =\bbf e_1\oplus\cdots\oplus \bbf e_{b_3},\quad V_{(4)}=\bbf e_{\alpha+1},\quad V_{(14)}=\bbf e_{\alpha+2}\oplus\cdots\oplus \bbf e_n \\
\mand V_{(10)} & =\bbf e_{\overline{\alpha}}\oplus\cdots\oplus \bbf e_{\overline{b_3+1}}.
\end{align*}

\begin{lemma} With respect to the basis
$$e_1,\ldots,e_{b_3},e_{\alpha+1},e_{\alpha+2},\ldots,e_n,e_{\overline{\alpha}},\ldots,e_{\overline{b_3+1}}$$
of $V$, the group $R_V|_V$ is represented as
$$R_V|_V=Q_{(4)}=\left\{\bp A & 0 & * & * \\ 0 & \lambda & * & * \\ 0 & 0 & B & * \\ 0 & 0 & 0 & C \ep\right\}$$
where $A\in {\rm GL}_{b_3}(\bbf),\ B\in {\rm GL}_{n-\alpha-1}(\bbf),\ C\in {\rm GL}_{\alpha-b_3}(\bbf)$ and $\lambda\in\bbf^\times$.
\label{lem5.1}
\end{lemma}

\begin{proof} It is clear from the conditions on $R_V$ that $R_V|_V\subset Q_{(4)}$.

Take elements $h_{(3)}(A),\ h_{(4)}(\lambda),\ h_{(14)}(B)$ and $h_{(10)}(C^*)$ constructed in Lemma \ref{lem3.12} where $C^*=J_{\alpha-b_3}{}^tC^{-1}J_{\alpha-b_3}$. Then
$$h_{(3)}(A)h_{(4)}(\lambda)h_{(14)}(B)h_{(10)}(C^*)|_V=
\bp A & 0 & 0 & 0 \\ 0 & \lambda & 0 & 0 \\ 0 & 0 & B & 0 \\ 0 & 0 & 0 & C \ep.$$
So we have only to construct generators of the unipotent part in the same way as in Lemma \ref{lem3.14}.

For $i\in I_{(3)}\sqcup I_{(4)},k\in I_{(14)}$ and $\mu\in\bbf$, define $g=g_{i,k}(\mu)\in G$ by
$$ge_k=e_k+\mu e_i,\quad ge_{\overline{i}}=e_{\overline{i}}-\mu e_{\overline{k}}$$
and $ge_\ell=e_\ell$ for $\ell\in I-\{k,\overline{i}\}$. Then we see that $g$ stabilizes $U_+,U_-$ and $V$. So $g\in R_V$.

Similarly for $i\in I_{(14)},\ k\in I_{(10)}$ and $\mu\in\bbf$, we can define $g=g_{i,\overline{k}}(\mu)\in R_V$ by
$$ge_{\overline{k}}=e_{\overline{k}}+\mu e_i,\quad ge_{\overline{i}}=e_{\overline{i}}-\mu e_k$$
and $ge_\ell=e_\ell$ for $\ell\in I-\{\overline{k},\overline{i}\}$. We can construct the remaining generators as in Lemma \ref{lem3.13'}.
\end{proof}

By Proposition \ref{prop9.2} in the appendix, $Q_{(4)}$-orbit decomposition on $M(V)$ is reduced to the orbit decomposition of the full flag variety of ${\rm GL}_{b_3+1}(\bbf)$ by the subgroup consisting of matrices of the form
$$\bp A & 0 \\
0 & \lambda \ep.$$
There are $(b_3+2)(b_3+1)/2$ orbits in this decomposition (\cite{M2}).

\subsubsection{Case of $b_{11}=1$}

$V$ is written as $V=V_{(3)}\oplus V_{(14)}\oplus V_{(11)}\oplus V_{(10)}$ with
\begin{align*}
V_{(3)} & =\bbf e_1\oplus\cdots\oplus \bbf e_{b_3},\quad V_{(14)}=\bbf e_{\alpha+2}\oplus\cdots\oplus \bbf e_n,\quad V_{(11)}=\bbf e_{\overline{\alpha+1}} \\
\mand V_{(10)} & =\bbf e_{\overline{\alpha}}\oplus\cdots\oplus \bbf e_{\overline{b_3+1}}.
\end{align*}
We can prove the following lemma in the same way as Lemma \ref{lem5.1}.

\begin{lemma} With respect to the basis
$$e_1,\ldots,e_{b_3},e_{\alpha+2},\ldots,e_n,e_{\overline{\alpha+1}},e_{\overline{\alpha}},\ldots,e_{\overline{b_3+1}}$$
of $V$, the group $R_V|_V$ is represented as
$$R_V|_V=Q_{(11)}=\left\{\bp A & * & * & * \\ 0 & B & * & * \\ 0 & 0 & \lambda & 0 \\ 0 & 0 & 0 & C \ep\right\}$$
where $A\in {\rm GL}_{b_3}(\bbf),\ B\in {\rm GL}_{n-\alpha-1}(\bbf),\ C\in {\rm GL}_{\alpha-b_3}(\bbf)$ and $\lambda\in\bbf^\times$.
\end{lemma}

By Proposition \ref{prop9.2}, $Q_{(11)}$-orbit decomposition on $M(V)$ is reduced to the orbit decomposition of the full flag variety of ${\rm GL}_{\alpha-b_3+1}(\bbf)$ by the subgroup consisting of matrices of the form
$\bp \lambda & 0 \\
0 & C \ep$
with $(\alpha-b_3+2)(\alpha-b_3+1)/2$ orbits.

\subsubsection{Case of $b_{7}=1$}

$V$ is written as $V=V_{(3)}\oplus V_{(7)}^1\oplus V_{(14)}\oplus V_{(10)}\oplus V_{(7)}^2$ with
\begin{align*}
V_{(3)} & =\bbf e_1\oplus\cdots\oplus \bbf e_{b_3},\quad V_{(7)}^1=\bbf (e_{b_3+1}+e_{\alpha+1}),\quad V_{(14)}=\bbf e_{\alpha+2}\oplus\cdots\oplus \bbf e_n, \\
V_{(10)} & =\bbf e_{\overline{\alpha}}\oplus\cdots\oplus \bbf e_{\overline{b_3+2}}\mand V_{(7)}^2=\bbf (e_{\overline{b_3+1}}-e_{\overline{\alpha+1}}).
\end{align*}

\begin{lemma} With respect to the basis
$$e_1,\ldots,e_{b_3},e_{b_3+1}+e_{\alpha+1},e_{\alpha+2}, \ldots,e_n,e_{\overline{\alpha}},\ldots,e_{\overline{b_3+2}},e_{\overline{b_3+1}}-e_{\overline{\alpha+1}}$$
of $V$, the group $R_V|_V$ is represented as
$$R_V|_V=Q_{(7)}=\left\{\bp A & * & * & * & * \\ 0 & \lambda & * & * & * \\ 0 & 0 & B & * & * \\ 0 & 0 & 0 & C & * \\ 0 & 0 & 0 & 0 &  \lambda^{-1} \ep\right\}$$
where $A\in {\rm GL}_{b_3}(\bbf),\ B\in {\rm GL}_{n-\alpha-1}(\bbf),\ C\in {\rm GL}_{\alpha-b_3-1}(\bbf)$ and $\lambda\in\bbf^\times$.
\label{lem5.4}
\end{lemma}

\begin{proof} Since $R_V$ stabilizes subspaces
\begin{align*}
V\cap U_+ & =V_{(3)},\quad V\cap (U_+\oplus U_-)=V_{(3)}\oplus V_{(7)}^1, \\
V\cap (U_+\oplus U_-)^\perp & =V_{(3)}\oplus V_{(7)}^1\oplus V_{(14)} \\
\mand V\cap U_-^\perp & =V_{(3)}\oplus V_{(7)}^1\oplus V_{(14)}\oplus V_{(10)}
\end{align*}
of $V$, $R_V|_V$ is contained in the parabolic subgroup
$$\left\{\bp A & * & * & * & * \\ 0 & \lambda & * & * & * \\ 0 & 0 & B & * & * \\ 0 & 0 & 0 & C & * \\ 0 & 0 & 0 & 0 &  \mu \ep\right\}$$
where $A\in {\rm GL}_{b_3}(\bbf),\ B\in {\rm GL}_{n-\alpha-1}(\bbf),\ C\in {\rm GL}_{\alpha-b_3-1}(\bbf)$ and $\lambda,\mu\in\bbf^\times$. Let $g$ be an element of $R_V$. Since $gU_-=U_-$, we have $ge_{\alpha+1}=\lambda e_{\alpha+1}$ with some $\lambda\in\bbf^\times$. Hence
$$g(e_{b_3+1}+e_{\alpha+1})\in \lambda(e_{b_3+1}+e_{\alpha+1})+V_{(3)}.$$
On the other hand, we have
$$ge_{\overline{\alpha+1}}=\lambda^{-1}e_{\overline{\alpha+1}}+\bigoplus_{i\ne \overline{\alpha+1}} \bbf e_i$$
since $g$ preserves the bilinear form $(\ ,\ )$. Hence
$$g(e_{\overline{b_3+1}}-e_{\overline{\alpha+1}})\in \lambda^{-1}(e_{\overline{b_3+1}}-e_{\overline{\alpha+1}})+(V_{(3)}\oplus V_{(7)}^1\oplus V_{(14)}\oplus V_{(10)}).$$
Thus we have $R_V|_V\subset Q_{(7)}$.

As in the proof of Lemma \ref{lem5.1}, the elements
$$h_{(3)}(A)h_{(7)}(\lambda)h_{(14)}(B)h_{(10)}(C^*)|_V=
\bp A & 0 & 0 & 0 & 0 \\ 0 & \lambda & 0 & 0 & 0 \\ 0 & 0 & B & 0 & 0 \\ 0 & 0 & 0 & C & 0 \\ 0 & 0 & 0 & 0 & \lambda^{-1} \ep$$
are contained in $R_V|_V$ for $A\in {\rm GL}_{b_3}(\bbf),\ B\in {\rm GL}_{n-\alpha-1}(\bbf),\ C\in {\rm GL}_{\alpha-b_3-1}(\bbf)$ and $\lambda\in\bbf^\times$. So we have only to consider the generators of the unipotent part.

For $i\in I_{(3)}$ and $\mu\in\bbf$, define $g=g_{i,b_3+1}(\mu)\in G$ by
$$ge_{b_3+1}=e_{b_3+1}+\mu e_i,\quad ge_{\overline{i}}=e_{\overline{i}}-\mu e_{\overline{b_3+1}}$$
and $ge_\ell=e_\ell$ for $\ell\in I-\{\alpha+1,\overline{i}\}$.
Then we see that $g$ stabilizes $U_+,U_-$ and $V$. So $g\in R_V$.

For $k\in I_{(14)}$ and $\mu\in\bbf$, define $g=g_{b_3+1,k}(\mu)\in G$ by
$$ge_k=e_k+\mu (e_{b_3+1}+e_{\alpha+1}),\quad ge_{\overline{b_3+1}}=e_{\overline{b_3+1}}-\mu e_{\overline{k}},\quad ge_{\overline{\alpha+1}}=e_{\overline{\alpha+1}}-\mu e_{\overline{k}}$$
and $ge_\ell=e_\ell$ for $\ell\in I-\{k,\overline{b_3+1},\overline{\alpha+1}\}$. Then we see that $g\in R_V$.

For $i\in I_{(14)},\ k\in I_{(10)}$ and $\mu\in\bbf$, define $g=g_{i,\overline{k}}(\mu)\in R_V$ by
$$ge_{\overline{k}}=e_{\overline{k}}+\mu e_i,\quad ge_{\overline{i}}=e_{\overline{i}}-\mu e_k$$
and $ge_\ell=e_\ell$ for $\ell\in I-\{\overline{k},\overline{i}\}$.

For $i\in I_{(10)}$ and $\mu\in\bbf$, define $g=g_{\overline{i},\overline{b_3+1}}(\mu)\in R_V$ by
$$ge_{\overline{b_3+1}}=e_{\overline{b_3+1}}+\mu e_{\overline{i}},\quad ge_i=e_i-\mu e_{b_3+1}$$
and $ge_\ell=e_\ell$ for $\ell\in I-\{i,\overline{b_3+1}\}$. We can construct the remaining generators as in Lemma \ref{lem3.13'}.
\end{proof}

By this lemma, $R_V$-orbit decomposition of $M_0(V)$ is reduced to the Bruhat decomposition. So there are a finite number of $R_V$-orbits on $M_0(V)$.

\subsection{Case of $a_1=1$}

As in Section 3, we put
$$U_+=\bbf e_1\oplus\cdots\oplus \bbf e_\alpha\mand U_-=\bbf e_{\overline{\alpha}}=\bbf e_{2n+2-\alpha}.$$
Since $a_0=a_-=0$ and since
$$a_0=b_1+b_2,\quad a_-=b_4+b_7+b_9+b_{11},$$
we have $b_1=b_2=b_4=b_7=b_9=b_{11}=0$. Since
$$a_1=b_5+b_9+2b_{12}+b_{15}+b_{13}+b_8+b_6=1,$$
there are five cases:
$${\rm (a)}\ b_5=1,\quad {\rm (b)}\ b_{15}=1,\quad {\rm (c)}\ b_{13}=1,\quad {\rm (d)}\ b_8=1\mand {\rm (e)}\ b_6=1.$$
Since $a_2=n-\alpha=b_{12}+b_{13}+b_{14}$, the condition ${\rm (C')}$ implies
$${\rm (C'')}\qquad b_{13}=1\Longrightarrow |\bbf^\times/(\bbf^\times)^2|<\infty.$$

\subsubsection{Case of $b_5=1$}

$V$ is written as $V=V_{(3)}\oplus V_{(5)}\oplus V_{(14)}\oplus V_{(10)}$ with
\begin{align*}
V_{(3)} & =\bbf e_1\oplus\cdots\oplus \bbf e_{b_3},\quad V_{(5)}=\bbf e_\alpha,\quad V_{(14)}=\bbf e_{\alpha+1}\oplus\cdots\oplus \bbf e_n \\
\mand V_{(10)} & =\bbf e_{\overline{\alpha-1}}\oplus\cdots\oplus \bbf e_{\overline{b_3+1}}.
\end{align*}
We can easily prove:

\begin{lemma} With respect to the basis
$$e_1,\ldots,e_{b_3},e_\alpha,e_{\alpha+1},\ldots,e_n,e_{\overline{\alpha-1}},\ldots,e_{\overline{b_3+1}}$$
of $V$, $R_V|_V$ is represented as
$$R_V|_V=Q_{(5)}=\left\{\bp A & * & * & * \\ 0 & \lambda & 0 & 0 \\ 0 & 0 & B & * \\ 0 & 0 & 0 & C \ep\right\}$$
where $A\in {\rm GL}_{b_3}(\bbf),\ B\in {\rm GL}_{n-\alpha}(\bbf),\ C\in {\rm GL}_{\alpha-b_3-1}(\bbf)$ and $\lambda\in\bbf^\times$.
\end{lemma}

Applying Proposition \ref{prop5.2} in the appendix, we see that there are a finite number of $R_V$-orbits on $M_0(V)$.

\subsubsection{Case of $b_6=1$}

$V$ is written as $V=V_{(3)}\oplus V_{(14)}\oplus V_{(6)}\oplus V_{(10)}$ with
\begin{align*}
V_{(3)} & =\bbf e_1\oplus\cdots\oplus \bbf e_{b_3},\quad V_{(14)}=\bbf e_{\alpha+1}\oplus\cdots\oplus \bbf e_n,\quad V_{(6)}=\bbf e_{\overline{\alpha}} \\
\mand V_{(10)} & =\bbf e_{\overline{\alpha-1}}\oplus\cdots\oplus \bbf e_{\overline{b_3+1}}.
\end{align*}
We can easily prove:

\begin{lemma} With respect to the basis
$$e_1,\ldots,e_{b_3},e_{\alpha+1},\ldots,e_n, e_{\overline{\alpha}},e_{\overline{\alpha-1}},\ldots,e_{\overline{b_3+1}}$$
of $V$, $R_V|_V$ is represented as
$$R_V|_V=Q_{(6)}=\left\{\bp A & * & 0 & * \\ 0 & B & 0 & * \\ 0 & 0 & \lambda & * \\ 0 & 0 & 0 & C \ep\right\}$$
where $A\in {\rm GL}_{b_3}(\bbf),\ B\in {\rm GL}_{n-\alpha}(\bbf),\ C\in {\rm GL}_{\alpha-b_3-1}(\bbf)$ and $\lambda\in\bbf^\times$.
\end{lemma}

Applying Proposition \ref{prop5.2}, we see that there are a finite number of $R_V$-orbits on $M_0(V)$.

\subsubsection{Case of $b_{15}=1$}

$V$ is written as $V=V_{(3)}\oplus V_{(14)}\oplus V_{(15)}\oplus V_{(10)}$ with
\begin{align*}
V_{(3)} & =\bbf e_1\oplus\cdots\oplus \bbf e_{b_3},\quad V_{(14)}=\bbf e_{\alpha+1}\oplus\cdots\oplus \bbf e_n, \\
V_{(15)} & =\bbf (e_{\overline{\alpha}}-\frac{1}{2}e_\alpha+e_{n+1}) 
\mand V_{(10)} =\bbf e_{\overline{\alpha-1}}\oplus\cdots\oplus \bbf e_{\overline{b_3+1}}.
\end{align*}
We can easily prove:

\begin{lemma} With respect to the basis
$$e_1,\ldots,e_{b_3},e_{\alpha+1},\ldots,e_n, e_{\overline{\alpha}}-\frac{1}{2}e_\alpha+e_{n+1},e_{\overline{\alpha-1}},\ldots,e_{\overline{b_3+1}}$$
of $V$, $R_V|_V$ is represented as
$$R_V|_V=Q_{(15)}=\left\{\bp A & * & * & * \\ 0 & B & * & * \\ 0 & 0 & 1 & 0 \\ 0 & 0 & 0 & C \ep\right\}$$
where $A\in {\rm GL}_{b_3}(\bbf),\ B\in {\rm GL}_{n-\alpha}(\bbf)$ and $C\in {\rm GL}_{\alpha-b_3-1}(\bbf)$.
\end{lemma}

By Proposition \ref{prop9.2}, $Q_{(15)}$-orbit decomposition on $M(V)$ is reduced to the orbit decomposition of the full flag variety of ${\rm GL}_{\alpha-b_3}(\bbf)$ by the subgroup consisting of matrices of the form
$\bp 1 & 0 \\
0 & C \ep$
with $(\alpha-b_3+1)(\alpha-b_3)/2$ orbits.

\subsubsection{Case of $b_{8}=1$}

$V$ is written as $V=V_{(3)}\oplus V_{(8)}^1\oplus V_{(14)}\oplus V_{(10)}\oplus V_{(8)}^2$ with
\begin{align*}
V_{(3)} & =\bbf e_1\oplus\cdots\oplus \bbf e_{b_3},\quad V_{(8)}^1=\bbf (e_{b_3+1}+e_{\overline{\alpha}}),\quad V_{(14)}=\bbf e_{\alpha+1}\oplus\cdots\oplus \bbf e_n, \\
V_{(10)} & =\bbf e_{\overline{\alpha-1}}\oplus\cdots\oplus \bbf e_{\overline{b_3+2}} \mand V_{(8)}^2=\bbf (e_{\overline{b_3+1}}-e_\alpha).
\end{align*}

\begin{lemma} With respect to the basis
$$e_1,\ldots,e_{b_3},e_{b_3+1}+e_{\overline{\alpha}},e_{\alpha+1},\ldots,e_n, e_{\overline{\alpha-1}},\ldots,e_{\overline{b_3+2}},e_{\overline{b_3+1}}-e_\alpha$$
of $V$, $R_V|_V$ is represented as
$$R_V|_V=Q_{(8)}=\left\{\bp A & * & * & * & * \\ 0 & \lambda & 0 & * & * \\ 0 & 0 & B & * & * \\ 0 & 0 & 0 & C & * \\ 0 & 0 & 0 & 0 & \lambda^{-1} \ep\right\}$$
where $A\in {\rm GL}_{b_3}(\bbf),\ B\in {\rm GL}_{n-\alpha}(\bbf),\ C\in {\rm GL}_{\alpha-b_3-2}(\bbf)$ and $\lambda\in\bbf^\times$.
\end{lemma}

\begin{proof} Since $R_V$ stabilizes subspaces
\begin{align*}
V\cap U_+ & =V_{(3)},\quad V\cap (U_+\oplus U_-)=V_{(3)}\oplus V_{(8)}^1, \\
V\cap (U_+\oplus U_-)^\perp & =V_{(3)}\oplus V_{(14)} 
\mand V\cap U_-^\perp =V_{(3)}\oplus V_{(8)}^1\oplus V_{(14)}\oplus V_{(10)}
\end{align*}
of $V$, $R_V|_V$ is contained in the group
$$\left\{\bp A & * & * & * & * \\ 0 & \lambda & 0 & * & * \\ 0 & 0 & B & * & * \\ 0 & 0 & 0 & C & * \\ 0 & 0 & 0 & 0 & \mu \ep\right\}$$
where $A\in {\rm GL}_{b_3}(\bbf),\ B\in {\rm GL}_{n-\alpha}(\bbf),\ C\in {\rm GL}_{\alpha-b_3-2}(\bbf)$ and $\lambda,\mu\in\bbf^\times$. Let $g$ be an element of $R_V$. Since $gU_-=U_-$, we have $ge_{\overline{\alpha}}=\lambda e_{\overline{\alpha}}$ with some $\lambda\in\bbf^\times$. Hence
$$g(e_{b_3+1}+e_{\overline{\alpha}})\in \lambda(e_{b_3+1}+e_{\overline{\alpha}})+V_{(3)}.$$
On the other hand, we have
$$ge_\alpha=\lambda^{-1}e_\alpha+\bigoplus_{i\ne \alpha} \bbf e_i$$
since $g$ preserves the bilinear form $(\ ,\ )$. Hence
$$g(e_{\overline{b_3+1}}-e_\alpha)\in \lambda^{-1}(e_{\overline{b_3+1}}-e_\alpha)+(V_{(3)}\oplus V_{(8)}^1\oplus V_{(14)}\oplus V_{(10)}).$$
Thus we have $R_V|_V\subset Q_{(8)}$.

The proof of $Q_{(8)}\subset R_V|_V$ is similar to Lemma \ref{lem5.4}.
\end{proof}

By Proposition \ref{prop9.2}, $Q_{(8)}$-orbit decomposition on $M(V)$ is reduced to the orbit decomposition of the full flag variety of ${\rm GL}_{n-\alpha+1}(\bbf)$ by the subgroup consisting of matrices of the form
$\bp \lambda & 0 \\
0 & B \ep$
with $(n-\alpha+2)(n-\alpha+1)/2$ orbits.

\subsubsection{Case of $b_{13}=1$}

$V$ is written as $V=V_{(3)}\oplus V_{(14)}\oplus V_{(13)}^1\oplus V_{(13)}^2\oplus V_{(10)}$ with
\begin{align*}
V_{(3)} & =\bbf e_1\oplus\cdots\oplus \bbf e_{b_3},\quad V_{(14)}=\bbf e_{\alpha+1}\oplus\cdots\oplus \bbf e_{n-1},\quad V_{(13)}^1=\bbf (e_\alpha+e_n), \\
V_{(13)}^2 & =\bbf (e_{\overline{\alpha}}-e_{n+2})\mand V_{(10)} =\bbf e_{\overline{\alpha-1}}\oplus\cdots\oplus \bbf e_{\overline{b_3+1}}.
\end{align*}

\begin{lemma} With respect to the basis
$$e_1,\ldots,e_{b_3},e_{\alpha+1},\ldots,e_n, e_\alpha+e_n, e_{\overline{\alpha}}-e_{n+2}, e_{\overline{\alpha-1}},\ldots,e_{\overline{b_3+1}}$$
of $V$, $R_V|_V$ is represented as
$$R_V|_V=Q_{(13)}=\left\{\bp A & * & * & * & * \\
0 & B & * & * & * \\
0 & 0 & \lambda & 0 & 0 \\
0 & 0 & 0 & \lambda^{-1} & * \\
0 & 0 & 0 & 0 & C \ep\right\}$$
where $A\in {\rm GL}_{b_3}(\bbf),\ B\in {\rm GL}_{n-\alpha-1}(\bbf),\ C\in {\rm GL}_{\alpha-b_3-1}(\bbf)$ and $\lambda\in\bbf^\times$.
\end{lemma}

\begin{proof} Since $R_V$ stabilizes subspaces
\begin{align*}
V\cap U_+ & =V_{(3)},\quad V\cap (U_+\oplus U_-)^\perp=V_{(3)}\oplus V_{(14)}, \\
V\cap U_+^\perp & =V_{(3)}\oplus V_{(14)}\oplus V_{(13)}^1, \quad
V\cap (W_+\oplus U_-)^\perp =V_{(3)}\oplus V_{(14)}\oplus V_{(13)}^2 \\
\mand V\cap U_-^\perp & =V_{(3)}\oplus V_{(14)}\oplus V_{(13)}^2\oplus V_{(10)}
\end{align*}
of $V$, $R_V|_V$ is contained in the group
$$\left\{\bp A & * & * & * & * \\
0 & B & * & * & * \\
0 & 0 & \lambda & 0 & 0 \\
0 & 0 & 0 & \mu & * \\
0 & 0 & 0 & 0 & C \ep\right\}$$
where $A\in {\rm GL}_{b_3}(\bbf),\ B\in {\rm GL}_{n-\alpha-1}(\bbf),\ C\in {\rm GL}_{\alpha-b_3-1}(\bbf)$ and $\lambda,\mu\in\bbf^\times$. Let $g$ be an element of $R_V$. Since $gU_-=U_-$, we have $ge_{\overline{\alpha}}=\lambda^{-1} e_{\overline{\alpha}}$ with some $\lambda\in\bbf^\times$. Hence
$$g(e_{\overline{\alpha}}-e_{n+2})\in \lambda^{-1}(e_{\overline{\alpha}}-e_{n+2})+(V_{(3)}\oplus V_{(14)}).$$
On the other hand, we have
$$ge_\alpha=\lambda e_\alpha+\bigoplus_{i\ne \alpha} \bbf e_i$$
since $g$ preserves the bilinear form $(\ ,\ )$. Hence
$$g(e_\alpha+e_n)\in \lambda(e_\alpha+e_n)+(V_{(3)}\oplus V_{(14)}).$$
Thus we have $R_V|_V\subset Q_{(13)}$.

The proof of $Q_{(13)}\subset R_V|_V$ is similar to Lemma \ref{lem5.4}.
\end{proof}

By Proposition \ref{prop9.2}, $Q_{(13)}$-orbit decomposition on $M(V)$ is reduced to the orbit decomposition of the full flag variety $M'$ of ${\rm GL}_{\alpha-b_3+1}(\bbf)$ by the subgroup $H$ consisting of matrices of the form
$$\bp \lambda & 0 & 0 \\
0 & \lambda^{-1} & * \\
0 & 0 & C \ep.$$
Let $\widetilde{H}$ denote the subgroup of ${\rm GL}_{\alpha-b_3+1}(\bbf)$ consisting of matrices of the form
$$\bp \lambda & 0 & 0 \\
0 & \mu & * \\
0 & 0 & C \ep$$
where $\lambda,\mu\in\bbf^\times$ and $C\in{\rm GL}_{\alpha-b_3-1}(\bbf)$. Then there are a finite number of $\widetilde{H}$-orbits on $M'$ by Proposition \ref{prop5.2}. Since the center $Z=\{\lambda I_{\alpha-b_3+1}\mid \lambda\in\bbf^\times \}$ acts trivially on $M'$ and since
$$\widetilde{H}/ZH\cong \bbf^\times/(\bbf^\times)^2$$
there are a finite number of $H$-orbits on $M'$ by the condition ${\rm (C'')}$.

Thus we have completed the proof of finiteness of the triple flag variety $\mct_{(\alpha),(1),(1^n)}$.

\section{Appendix}

\subsection{An elementary lemma for ${\rm GL}_n(\bbf)$}

\begin{lemma} Let $U,U',V$ and $V'$ be subspaces of $\bbf^n$ over an arbitrary field $\bbf$ such that
$\dim U =\dim U',\ \dim V=\dim V'$ and that $\dim(U\cap V) =\dim(U'\cap V')$.
Then there exists an element $g$ of $G={\rm GL}_n(\bbf)$ such that
$$gU=U'\quad\mbox{and that}\quad gV=V'.$$
\label{lem6.1'}
\end{lemma}

\begin{proof} Put $\dim U =\dim U'=k+m,\ \dim V=\dim V'=\ell+m$ and $\dim(U\cap V) =\dim(U'\cap V')=m$. Let $w_1,\ldots,w_m$ and $w'_1,\ldots,w'_m$ be bases of $W=U\cap V$ and $W'=U'\cap V'$, respectively. Then we can take vectors $u_1,\ldots,u_k,u'_1,\ldots,u'_k,v_1,\ldots,v_\ell,v'_1,\ldots,v'_\ell$ such that
\begin{align*}
U & =\bbf u_1\oplus\cdots\oplus \bbf u_k\oplus W,\quad U'=\bbf u'_1\oplus\cdots\oplus \bbf u'_k\oplus W', \\
V & =\bbf v_1\oplus\cdots\oplus \bbf v_\ell \oplus W \quad\mbox{and that}\quad V'=\bbf v'_1\oplus\cdots\oplus \bbf v'_\ell \oplus W'.
\end{align*}
Since the vectors $u_1,\ldots,u_k,v_1,\ldots,v_\ell,w_1,\ldots,w_m$ and $u'_1,\ldots,u'_k,v'_1,\ldots,v'_\ell,w'_1,\ldots,w'_m$ are linearly independent, we can take an element $g\in G$ such that
$$gu_1=u'_1,\ldots,gu_k=u'_k,\ gv_1=v'_1,\ldots,gv_\ell=v'_\ell,\ gw_1=w'_1\ldots,gw_m=w'_m.$$
This $g$ is a desired element of $G$.
\end{proof}

\subsection{Finiteness of some orbit decomposition on the Grassmann variety}

Consider the direct sum decomposition $\bbf^{m+n}=U_1\oplus U_2$ with $U_1=\bbf e_1\oplus\cdots\oplus \bbf e_m$ and $U_2=\bbf e_{m+1}\oplus\cdots\oplus \bbf e_{m+n}$ over an arbitrary field $\bbf$. Write $G_1={\rm GL}(U_1)$ and $G_2={\rm GL}(U_2)$. For $0<s<m+n$, let ${\rm Gr}$ denote the Grassmann variety
$${\rm Gr} =\{s\mbox{-dimensional subspaces in }U_1\oplus U_2\}.$$

Let $\pi_i$ denote the projection $U_1\oplus U_2\to U_i$ for $i=1,2$. Let $S$ be an $s$-dimensional subspace of $U_1\oplus U_2$. Then the $G_1\times G_2$-orbit of $S$ is determined by the four invariants
$$\dim(U_1\cap S),\quad \dim\pi_1(S),\quad \dim(U_2\cap S)\mand \dim\pi_2(S)$$
with $\dim(U_1\cap S)+\dim\pi_2(S)=\dim(U_2\cap S)+\dim\pi_1(S)=s$.
Write $p=\dim(U_1\cap S),\ q=\dim(U_1\cap S)$ and $r=s-p-q$. Then 
$$\dim\pi_1(S)-p=\dim\pi_2(S)-q=r.$$

Consider the canonical full flag
$$U_{2,1}\subset\cdots\subset U_{2,n-1}$$
in $U_2$ where $U_{2,i}=\bbf e_{m+1}\oplus\cdots\oplus \bbf e_{m+i}$. Let $B_2$ denote the Borel subgroup of $G_2$ stabilizing this full flag. Define subsets
\begin{align*}
J =\{i_1,\ldots,i_q\} & =\{i\in I_2\mid \dim(U_{2,i}\cap S)=\dim(U_{2,i-1}\cap S)+1\} \\
\mand K =\{j_1,\ldots,j_r\} & =\{i\in I_2 \mid \dim(U_{2,i}\cap S)=\dim(U_{2,i-1}\cap S), \\
& \qquad \qquad \dim(\pi_2(S)\cap U_{2,i})=\dim(\pi_2(S)\cap U_{2,i-1})+1\}
\end{align*}
of $I_2=\{m+1,\ldots,m+n\}$ with $i_1<\cdots<i_q$ and $j_1<\cdots<j_r$. Then we can take a $b\in B_2$ such that
$$b\pi_2(S)=\bbf e_{i_1}\oplus\cdots\oplus \bbf e_{i_q}\oplus \bbf e_{j_1}\oplus\cdots\oplus \bbf e_{j_r}.$$
We can write
$$bS=\bbf v_1\oplus\cdots\oplus \bbf v_p\oplus \bbf e_{i_1}\oplus\cdots\oplus \bbf e_{i_q} \oplus \bbf (w_1+e_{j_1})\oplus\cdots\oplus \bbf (w_r+e_{j_r})$$
with some linearly independent vectors $v_1,\ldots,v_p,w_1,\ldots,w_r$ in $U_1$. So we can take a $g_1\in G_1$ such that $S_0=g_1bS$ is of the form
\begin{equation}
S_0=U_{1,p}\oplus \bbf e_{i_1}\oplus\cdots\oplus \bbf e_{i_q} \oplus \bbf (e_{p+1}+e_{j_1}) \oplus\cdots\oplus \bbf (e_{p+r}+e_{j_r})
\label{eq6.1}
\end{equation}
where $U_{1,k}=\bbf e_1\oplus\cdots\oplus \bbf e_k$ for $k=1,\ldots, m$

Let $Q$ denote the isotropy subgroup
$$Q=\{g\in G_1\times B_2\mid gS_0=S_0\}$$
of $S_0$ in $G_1\times B_2$. Let $\widetilde\pi_1: G_1\times G_2\to G_1$ denote the projection.

\begin{lemma} $\widetilde\pi_1(Q)$ is equal to the parabolic subgroup $P_1$ of $G_1$ stabilizing the flag
$U_{1,p}\subset U_{1,p+1}\subset\cdots\subset U_{1,p+r}$
in $U_1$.
\label{lem6.1}
\end{lemma}

\begin{proof} Let $g$ be an element of $P_1$. Then for $k=1,\ldots,r$ we can write
$$ge_{p+k}\in (\sum_{i=1}^k a_{i,k}e_{p+i})+U_{1,p}$$
with some $a_{i,k}\in\bbf$. Define an element $b\in B_2$ by
$$be_{j_k}=\sum_{i=1}^k a_{i,k}e_{j_k}\quad \mbox{for }k=1,\ldots,r$$
and $be_\ell=e_\ell$ for $\ell\in I_2-K$. Then we have
$$bgS_0=S_0$$
and therefore $bg\in Q$. Hence $g\in \widetilde\pi_1(Q)$.

Conversely let $g$ be an element of $Q$. Since $g$ stabilizes $U_1\oplus U_{2,k}$ for $k=0,\ldots,n$, it stabilizes
$$S_0\cap(U_1\oplus U_{2,k})=U_{1,p}\oplus \bbf e_{i_1}\oplus\cdots\oplus \bbf e_{i_{k(1)}}\oplus \bbf (e_{p+1}+e_{j_1})\oplus\cdots\oplus \bbf (e_{p+k(2)}+e_{j_{k(2)}})$$
where $i_{k(1)}\le k<i_{k(1)+1}$ and $j_{k(2)}\le k<j_{k(2)+1}$. Hence it stabilizes
$$\pi_1(S_0\cap(U_1\oplus U_{2,k}))=U_{1,p}\oplus \bbf e_{p+1}\oplus\cdots\oplus \bbf e_{p+k(2)}=U_{1,p+k(2)}.$$
Since $k(2)$ varies from $0$ to $r$, we have $\widetilde\pi_1(g)\in P_1$.
\end{proof}

\begin{proposition} Let $H_1$ be a subgroup of $G_1$ such that
$$|H_1\backslash G_1/B_1|<\infty$$
where $B_1$ is a Borel subgroup of $G_1$. Then ${\rm Gr}$ consists of finitely many $H_1\times B_2$-orbits.
\label{lem8.1}
\end{proposition}

\begin{proof} Since every $G_1\times B_2$-orbit in ${\rm Gr}$ contains an element $S_0$ of the form (\ref{eq6.1}), we have only to show that $(G_1\times B_2)S_0\cong (G_1\times B_2)/Q$ is decomposed into a finite number of $H_1\times B_2$-orbits. By the map $\widetilde\pi_1$, we have
$$(H_1\times B_2)\backslash (G_1\times B_2)/Q\stackrel{\sim}{\to} H_1\backslash G_1/\widetilde\pi_1(Q).$$
So the assertion follows from the assumption $|H_1\backslash G_1/B_1|<\infty$ and Lemma \ref{lem6.1}.
\end{proof}

\subsection{An orbit decomposition of ${\rm GL}_n(\bbf)/B$}

Let $V_1\subset V_2\subset\cdots\subset V_{n-1}$ be the canonical full flag in $\bbf^n$ defined by
$$V_1=\bbf e_1,\ V_2=\bbf e_1\oplus \bbf e_2,\ \ldots,\ V_{n-1}=\bbf e_1\oplus\cdots\oplus \bbf e_{n-1}.$$
Then
$$B=\{g\in G\mid gV_i=V_i\mbox{ for }i=1,\ldots,n-1\}=\{\mbox{upper triangular matrices in }G\}$$
is a Borel subgroup of $G={\rm GL}_n(\bbf)$.

Suppose $n=\alpha_1+\cdots+\alpha_p$ with positive integers $\alpha_1,\ldots,\alpha_p$. Define a partition $I=I_1\sqcup\cdots\sqcup I_p$ of $I=\{1,\ldots,n\}$ by
$$I_1=\{1,\ldots,\alpha_1\},\ I_2=\{\alpha_1+1,\ldots,\alpha_1+\alpha_2\},\ \ldots,\ I_p=\{\alpha_1+\cdots\alpha_{p-1}+1,\ldots,n\}.$$
Put $U_j=\oplus_{k\in I_j} \bbf e_k$. Then we have a direct sum decomposition
$\bbf^n=U_1\oplus\cdots\oplus U_p$.
Let $P$ be the parabolic subgroup of $G$ defined by
$$P=\left\{\bp A_1 && * \\ & \ddots & \\ 0 && A_p \ep \Bigm| A_k\in {\rm GL}_{\alpha_k}(\bbf)\mbox{ for }k=1,\ldots,p\right\}.$$
Then $P$ is the isotropy subgroup in $G$ for the flag
$U_1\subset U_1\oplus U_2\subset\cdots\subset U_1\oplus\cdots\oplus U_{p-1}$.

Let $S_n$ denote the symmetric group for $I$. For an element $\sigma\in S_n$, there corresponds a permutation matrix $w_\sigma$ defined by
$$w_\sigma(e_i)=e_{\sigma(i)}\quad\mbox{for }i\in I.$$
Fix a $j\in\{1,\ldots,p\}$ and $\sigma\in S_n$. Then there exists a unique sequence $r_1<\cdots<r_{\alpha_j}$ such that
$$(\sigma(r_1),\ldots,\sigma(r_{\alpha_j}))=I_j.$$
Consider the permutation $\tau$ of $I_j$ given by
\begin{equation}
\tau\sigma(r_1)=\beta+1,\ \ldots,\ \tau\sigma(r_{\alpha_j})=\beta+\alpha_j
\label{eq9.1}
\end{equation}
where $\beta=\alpha_1+\cdots+\alpha_{j-1}$. Since $w_\tau\in P$, we have
$$Pw_\sigma B=Pw_\tau w_\sigma B=PwB$$
where $w=w_\tau w_\sigma$. It follows from (\ref{eq9.1}) that
\begin{equation}
U_j\cap wV_{r_k}=\bbf e_{\beta+1}\oplus\cdots\oplus \bbf e_{\beta+k}\quad\mbox{for }k=1,\ldots,\alpha_j.
\label{eq9.2}
\end{equation}

Let $L_j$ be the subgroup of $P$ defined by
$$L_j=\left\{\bp I_\beta && 0 \\ & A_j & \\ 0 && I_\gamma \ep \Bigm| A_j\in {\rm GL}_{\alpha_j}(\bbf)\right\}$$
where $\gamma=\alpha_{j+1}+\cdots+\alpha_p$ and $N_j$ the normal subgroup of $P$ defined by
$$N_j=\left\{\bp A_1 &&&&&& * \\ & \ddots &&&&& \\ && A_{j-1} &&&& \\ &&& I_{\alpha_j} &&& \\ &&&& A_{j+1} && \\ &&&&& \ddots & \\ 0 &&&&&& A_p \ep \Bigm| A_k\in{\rm GL}_{\alpha_k}(\bbf)\right\}.$$
Then $P=L_jN_j=N_jL_j$.

\begin{lemma} {\rm (i)} $L_j\cap wBw^{-1}=L_j\cap B$.

{\rm (ii)} $P\cap wBw^{-1}=(L_j\cap B)(N_j\cap wBw^{-1}).$
\label{lem9.1}
\end{lemma}

\begin{proof} (i) Let $g$ be an element of $L_j$. Then $g\in wBw^{-1}$ if and only if $g(U_j\cap wV_k)=U_j\cap wV_k$ for $k=1,\ldots,n-1$ since
$$wV_k=(U_1\cap wV_k)\oplus\cdots\oplus (U_p\cap wV_k).$$
Hence the assertion follows from (\ref{eq9.2}).

(ii) Suppose $g\in P\cap wBw^{-1}$. Then by (\ref{eq9.2}), we have
$$ge_{\beta+k}\in (U_1\oplus\cdots\oplus U_j)\cap wV_{r_k}=(U_1\cap wV_{r_k})\oplus\cdots\oplus (U_j\cap wV_{r_k})$$
and $U_j\cap wV_{r_k}=\bbf e_{\beta+1}\oplus\cdots\oplus \bbf e_{\beta+k}$ for $k=1,\ldots,\alpha_j$. Write
$$ge_{\beta+k}=u_k+v_k$$
with $u_k\in U_1\oplus\cdots\oplus U_{j-1}$ and $v_k\in U_j$. Since $g$ defines a linear isomorphism on the factor space $(U_1\oplus\cdots\oplus U_j)/(U_1\oplus\cdots\oplus U_{j-1})$, the map $\ell: e_{\beta+k}\mapsto v_k$ defines a linear isomorphism on $U_j$. Hence $\ell\in L_j$. Since $v_k\in \bbf e_{\beta+1}\oplus\cdots\oplus \bbf e_{\beta+k}$ for $k=1,\ldots,\alpha_j$, we have $\ell\in L_j\cap B$. Hence $\ell^{-1}g\in wBw^{-1}$ by (i). Since $\ell^{-1}ge_{\beta+k}=e_{\beta+k}+u_k$, we have
$\ell^{-1}g\in N_j$.
\end{proof}

Let $H$ be a subgroup of $L_j$ such that $|H\backslash L_j/(L_j\cap B)|<\infty$ and $Q=HN_j=N_jH$.

\begin{proposition} Suppose $L_j=\bigsqcup_{k=1}^{n_H} Hg_k(L_j\cap B)$. Then
$$PwB=\bigsqcup_{k=1}^{n_H} Qg_kwB.$$
\label{prop9.2}
\end{proposition}

\begin{proof} By the map $g\mapsto gw^{-1}$,
$$Q\backslash PwB/B\cong Q\backslash PwBw^{-1}/wBw^{-1}.$$
So we have only to consider the decomposition
$Q\backslash P/(P\cap wBw^{-1})$
of $P$. Since $Q\supset N_i$, we can take representatives in $L_j$. Suppose $g_1,g_2\in L_i$ satisfies $g_1\in Qg_2(P\cap wBw^{-1})$. Then we have
$$g_1=nhg_2 \ell n'$$
with some $n\in N_j,h\in H,\ell\in L_j\cap B$ and $n'\in N_j\cap wBw^{-1}$ by Lemma \ref{lem9.1}. Since
$$hg_2 \ell n'=n^{-1}g_1=g_1(g_1^{-1}n^{-1}g_1)\in g_1N_j,$$
we have $hg_2\ell=g_1$. Hence
$Q\backslash P/(P\cap wBw^{-1})\cong H\backslash L_j/(L_j\cap B)$.
\end{proof}

Since $\displaystyle{|P\backslash G/B|=\frac{n!}{\alpha_1!\cdots \alpha_p!}}$, we have

\begin{corollary} $\displaystyle{|Q\backslash G/B|=\frac{n_Hn!}{\alpha_1!\cdots \alpha_p!}}$.
\label{cor9.3}
\end{corollary}

The following result given by T. Hashimoto is useful.

\begin{proposition} {\rm (\cite{H})} Let $H$ be a subgroup of ${\rm GL}_n(\bbf)$ of the form
$$H=\left\{\bp 1 & 0 \\ 0 & A \ep \Bigm| A\mbox{ is an upper triangular matrix in }{\rm GL}_{n-1}(\bbf)\right\}.$$
Then there are a finite number of $H$-orbits on the full flag variety of ${\rm GL}_n(\bbf)$.
\label{prop5.2}
\end{proposition}

\end{document}